\documentclass[11pt]{amsart}
\usepackage{amssymb}
\usepackage{graphicx}
\usepackage{amsfonts}
\usepackage{amsmath}
\usepackage{amsthm}
\usepackage{fancyhdr}
\usepackage[colorlinks=true,linkcolor=blue]{hyperref}
\usepackage{comment}
\usepackage{color}
\usepackage{subcaption}
\usepackage{epsfig}
\usepackage{pst-grad} 
\usepackage{pst-plot} 
\usepackage[space]{grffile} 
\usepackage{etoolbox} 
\usepackage{mathrsfs} 
\usepackage{enumitem} 

\newcommand{\A}{\mathcal{A}}
\newcommand{\al}{\alpha}

\newcommand{\V}{\mathcal{V}}
\newcommand{\R}{\mathbb{R}}
\newcommand{\N}{\mathbb{N}}
\newcommand{\mH}{\mathcal{H}}
\newcommand{\F}{\mathcal{F}}

\newcommand{\bB}{\mathbf{B}}
\newcommand{\sB}{\widetilde{B}}
\newcommand{\C}{\mathcal{C}}
\newcommand{\bC}{\mathbf{C}}
\newcommand{\lan}{\langle}
\newcommand{\ran}{\rangle}

\newcommand{\lc}{\scalebox{1.8}{$\llcorner$}}
\newcommand{\La}{\Lambda}

\newcommand{\Si}{\Sigma}
\newcommand{\de}{\delta}

\newcommand{\ep}{\epsilon}
\newcommand{\pr}{\partial}
\newcommand{\Om}{\Omega}

\newcommand{\mZ}{\mathbb{Z}}
\newcommand{\Z}{\mathcal{Z}}
\newcommand{\mS}{\mathcal{S}}
\newcommand{\f}{\mathbf{f}}

\newcommand{\M}{\mathbf{M}}
\newcommand{\iM}{\mathring{M}}
\newcommand{\pM}{\partial{M}}
\newcommand{\bL}{\mathbf{L}}
\newcommand{\bI}{\mathbf{I}}

\newcommand{\mF}{\mathbf{F}}
\newcommand{\mR}{\mathcal{R}}

\newcommand{\X}{\mathfrak{X}}
\newcommand{\sA}{\mathscr{A}}

\newcommand{\bleta}{\boldsymbol{\eta}}

\newcommand{\bangle}[1]{\left\langle #1 \right\rangle}

\newcommand{\rom}[1]{\expandafter\romannumeral #1}
\newcommand{\Rom}[1]{\uppercase\expandafter{\romannumeral #1}}
\newcommand{\pa}[2]{\frac{\partial #1}{\partial #2}}

\newcommand{\td}[2]{\frac{d #1}{d #2}}

\newcommand{\VarTan}{\operatorname{VarTan}}
\newcommand{\lip}{\operatorname{Lip}}
\newcommand{\spt}{\operatorname{spt}}
\newcommand{\dist}{\operatorname{dist}}
\newcommand{\Div}{\operatorname{div}}
\newcommand{\tr}{\operatorname{tr}}

\newcommand{\vol}{\operatorname{Vol}} 

\newcommand{\Ric}{\operatorname{Ric}}
\newcommand{\Area}{\operatorname{Area}}

\newcommand{\Clos}{\operatorname{Clos}}
\newcommand{\interior}{\operatorname{int}}
\newcommand{\An}{\operatorname{An}}
\newcommand{\wt}{\widetilde}
\newcommand{\Lap}{\Delta}

\renewcommand{\div}{\operatorname{div}}

\DeclareMathOperator{\secondfund}{II}

\def\doubleunderline#1{\underline{\underline{#1}}}

\begin{document}

\newtheorem{theorem}{Theorem}[section]
\newtheorem{proposition}[theorem]{Proposition}
\newtheorem{corollary}[theorem]{Corollary}
\newtheorem{example}[theorem]{Example}

\newtheorem{claim}{Claim}

\theoremstyle{remark}
\newtheorem{remark}[theorem]{Remark}

\theoremstyle{definition}
\newtheorem{definition}[theorem]{Definition}

\theoremstyle{plain}
\newtheorem{lemma}[theorem]{Lemma}

\numberwithin{equation}{section}

\title[Min-max theory for Capillary surfaces]{Min-max theory for Capillary surfaces}

\author[Chao Li]{Chao Li}
\address{Courant Institute of Mathematical Sciences, New York University, 251 Mercer St, New York, NY 10012, USA}
\email{chaoli@nyu.edu}

\author[Xin Zhou]{Xin Zhou}
\address{Department of Mathematics, Cornell University, Ithaca, NY 14853, USA, and Department of Mathematics, University of California Santa Barbara, Santa Barbara, CA 93106, USA}
\email{xinzhou@cornell.edu}

\author[Jonathan J. Zhu]{Jonathan J. Zhu}
\address{Department of Mathematics, University of Washington, Seattle, WA, USA}
\email{jonozhu@uw.edu}


\maketitle

\pdfbookmark[0]{}{beg}

\begin{abstract}
We develop a min-max theory for the construction of capillary surfaces in 3-manifolds with smooth boundary. In particular, for a generic set of ambient metrics, we prove the existence of nontrivial, smooth, almost properly embedded surfaces with any given constant mean curvature $c$, and with smooth boundary contacting at any given constant angle $\theta$. Moreover, if $c$ is nonzero and $\theta$ is not $\frac{\pi}{2}$, then our min-max solution always has multiplicity one. We also establish a stable Bernstein theorem for minimal hypersurfaces with certain contact angles in higher dimensions. 
\end{abstract}

\setcounter{section}{-1}

\section{Introduction}
\label{S:intro}

Capillary surfaces are the mathematical model for the interfaces between incompressible immiscible fluids. If a liquid occupies a region $\Omega^{n+1}$ in a container $M^{n+1}$ (a Riemannian manifold with boundary), then Gauss' free energy consists of the following terms: the free surface energy $\mH^n(\partial \Omega\lc \mathring M)$, the wetting energy  $\mH^n(\partial \Omega\lc \partial M)$, and a potential energy of the liquid. Assuming homogeneity of the liquid and the container, this energy can be written as
\begin{equation}
    \label{eq:A-cap}
\A(\Omega)=\mH^n(\partial \Omega\lc \mathring M) + (\cos\theta) \mH^n(\partial \Omega\lc \partial M) - c\vol(\Omega).
\end{equation}
Capillary surfaces are then critical points of the functional (\ref{eq:A-cap}); for classical solutions the boundary $\Sigma=\pr\Omega\cap \mathring{M}$ satisfies the following elliptic partial differential equation subject to Neumann-type boundary conditions
\begin{equation}
    \begin{split}
    H=c, & \qquad \text{ on } \Sigma\\
    \langle \nu, \overline{\eta}\rangle = \cos \theta, & \qquad \text{ on } \pr\Sigma,
    \end{split}
\end{equation}
where $H$ is the mean curvature and $\nu$ the outer surface normal on $\Sigma$, $\overline{\eta}$ is the outer normal of $\partial M$ in $M$ and $c,\theta$ are constants\footnote{In our convention, the mean curvature of a unit sphere in $\R^3$ is 2.}. Since then, there has been a large amount of interdisciplinary investigations on the stationary solutions and local minimizers of the $\A$ energy; see the beautiful monograph of Finn \cite{Finn1986equilibrium} for an overview from antiquity. However, there have been very few general existence results for capillary surfaces, particularly with \textit{prescribed} mean curvature and contact angle.

In this paper, we construct, via a min-max method, nontrivial capillary surfaces of any prescribed constant mean curvature and contact angle, for generic set of ambient metrics:

\begin{theorem}
\label{thm:main-intro}
Let $M^3$ be a compact manifold with smooth boundary $\partial M$. There is an open, dense set of Riemannian metrics on $M$ such that the following holds:
Given any $c\in\R$ and any $\theta\in (0, \frac{\pi}{2}]$ there exists a nontrivial, smooth, almost properly embedded surface $\Si\subset M$ which has constant mean curvature $c$ and smooth boundary $\partial \Sigma$ contacting $\partial M$ with angle $\theta$. 
\end{theorem}


Here the generic set of metrics consists of those for which the boundary mean curvature is a Morse function. In fact, our min-max theory can handle any metric satisfying: 
\begin{itemize}
\item[($\star$)]\label{def:generic} If $\Sigma \hookrightarrow M$ is an embedded surface of constant mean curvature with respect to $g$, then $\Sigma \cap \pr M$ is contained in a countable union of connected, smoothly embedded $1$-dimensional submanifolds. 
\end{itemize}

Our theory is the first min-max theory for nonzero Neumann boundary conditions and nonzero mean curvature, and builds upon work of the second and third named authors \cite{Zhou-Zhu17, Zhou-Zhu18} for prescribed mean curvature surfaces, and of the second named author and M. Li \cite{Li-Zhou16} for free boundary ($\theta=\frac{\pi}{2}$) minimal surfaces. 

As we were finishing writing the paper, we were pleased to learn that a min-max construction for the special case of capillary minimal surfaces ($c=0$) in convex domains in $\R^3$ was independently carried out by De Masi and De Philippis \cite{DePhilippisDeMasi21}. Indeed, property ($\star$) is designed to rule out large sets of improper boundary touching; in the case of minimal $\Sigma$ and convex $\pr M$, \textit{any} boundary touching is ruled out (hence ($\star$) is satisfied) by the maximum principle. 

A key component of our theory is the establishment of curvature estimates for stable capillary surfaces. To achieve these, we need a Bernstein-type theorem for stable capillary minimal surfaces in a Euclidean half-space:

\begin{theorem}
\label{thm:bernstein-intro}
    Let $\R^3_+$ denote the half space $\{(x_1,x_2,x_3):x_1\ge 0\}$, and suppose $\Sigma^2\subset \R^3_{+}$ is a properly immersed, two-sided capillary minimal surface. Suppose also that $\Sigma$ has Euclidean area growth, that is, there exists some $C>0$ such that
    \[Area(\Sigma\cap B_r(0))< C r^2\]
    for any $r>0$, and that $\Sigma$ is stable for the capillary functional. Then $\Sigma$ is planar.
\end{theorem}

Prior to the completion of this article, Hong-Saturnino \cite{HS21} independently proved Theorem \ref{thm:bernstein-intro}. Our approach differs from theirs and also admits generalisations to hypersurfaces $\Sigma^n \subset \mathbb{R}^{n+1}$ of dimension $2\leq n\leq 5$ (with restrictions on the contact angle when $n>2$); see Remark \ref{rmk:souam} and Appendix \ref{sec:bernstein}. We expect these Bernstein-type results to be of significant independent interest in the study of capillary surfaces, just as the classical Bernstein theorem is fundamental to the theory of minimal surfaces (without boundary). Note that the classical Bernstein theorem holds for stable minimal hypersurfaces of dimension $2\leq n\leq 6$. We expect the generalisations of \ref{thm:bernstein-intro} for dimension $n>2$ to play an important role in the min-max construction of capillary hypersurfaces in higher dimensions.

Historically, the study of capillary surfaces began in 1805, when Young \cite{Young1805essay} studied the equilibrium state of liquid fluids, introducing the notion of \textit{mean curvature} and proposing the boundary contact angle condition of capillarity - nowadays also known as \textit{Young's law}. These ideas were reintroduced by Laplace \cite{Laplace1805mechenics} and reformulated by Gauss \cite{Gauss1830principia} in 1830 through introducing the free energy as in (\ref{eq:A-cap}). 

We are interested in the existence and regularity of stationary solutions $\Sigma:=\pr\Omega\lc \mathring M$ for $\A$. Interior regularity of $\Sigma$ follows from the classical theory of minimal and constant mean curvature surfaces, so a major challenge is the regularity near the boundary. A classical result in this direction was due to Taylor \cite{Taylor77}, who proved $C^\infty$ boundary regularity for local minimizing solutions when $n=3$ and $M$ is a general smooth Riemannian manifold with boundary. Note that when $\theta=\tfrac{\pi}{2}$, $\Sigma$ is called a varifold with free boundary, and its boundary regularity was settled by Gr\"uter \cite{Gruter1987optimal} and Gr\"uter-Jost \cite{GruterJost1986allard}. When $n\ge 4$, De Philippis-Maggi obtained a sharp partial regularity result \cite[Theorem 1.2]{DePhilippi-Maggi15} for local minimizers of much more general capillary type functionals, allowing anisotropic surface energy, inhomogenous boundary adhesion and nonconstant gravitational energy\footnote{In particular, the singular set of $\partial\Sigma$ has Hausdorff dimension at most $n-4$. This dimension bound is conjecturally not optimal for the $\A$ functional, cf. \cite{Gruter1987optimal}.}. Recently, minimizers of $\A$ have been used as a key geometric tool by the first-named author \cite{Li2017polyhedron} in proving Gromov's geometric comparison theorems for scalar curvature via Riemannian polyhedrons (see also \cite[Section 1.2]{li2020dihedral} on the perspective of its higher dimensional extensions).

Our proof of Theorem \ref{thm:main-intro} is based on a min-max approach for constructing (unstable) critical points of the $\A$ functional. 
Min-max methods have been very successful recently in the construction of minimal hypersurfaces, constant mean curvature (CMC) hypersurfaces, and more generally hypersurfaces with prescribed mean curvature (PMC). These hypersurfaces may be seen as critical points of the area functional, possibly under certain volume constraints, or of modified area functionals. The Almgren-Pitts min-max theory \cite{Almgren62, Almgren65, Pitts, Schoen-Simon81} provided the first general existence result for closed minimal hypersurfaces, and was recently greatly improved starting from the solution of the Willmore conjecture by Marques-Neves \cite{Marques-Neves14}.  Yau's conjecture \cite{Yau82} on the existence of infinitely many closed minimal surfaces was solved by combining the works of Marques-Neves \cite{Marques-Neves17} and Song \cite{Song18}.  For generic metrics, Irie-Marques-Neves \cite{Irie-Marques-Neves18}, Marques-Neves-Song \cite{Marques-Neves-Song19} respectively proved density and equidistribution results for closed minimal hypersurfaces, using the Weyl Law for the area functional by Liokumovich-Marques-Neves \cite{Liokumovich-Marques-Neves18}. In contrast, Song and the second-named author obtained scarring results for closed minimal hypersurfaces surrounding any closed stable hypersurface \cite{Song-Zhou20} using a cylindrical version of the Weyl law in \cite{Song18}. Around the same time, a Morse theory for the area functional was established: the second-named author \cite{Zhou19} proved the Multiplicity One Conjecture raised by Marques-Neves \cite{Marques-Neves16, Marques-Neves18} (see also Chodosh-Mantoulidis \cite{Chodosh-Mantoulidis20}). Using this, Marques-Neves \cite{Marques-Neves18} proved that, for bumpy metrics, there exists a closed minimal hypersurface of Morse index $p$ for each $p\in \mathbb N$. Recently, the Morse inequalities for the area functional were proved for bumpy metrics by Marques-Montezuma-Neves \cite{Marques-Montezuma-Neves20}. Parallel to the developments for minimal hypersurfaces, which have zero constant mean curvature, the second and third named authors \cite{Zhou-Zhu17, Zhou-Zhu18} further generalized the Almgren-Pitts min-max theory to the CMC and PMC settings, and established a complete existence theory for closed CMC hypersurfaces for all prescribed mean curvatures, as well as an existence theory for closed PMC hypersurfaces for a smooth generic set of prescription functions. The latter PMC min-max theory played an essential role in \cite{Zhou19}. Very recently, a mapping approach for constructing min-max CMC spheres was developed by Cheng and the second-named author \cite{Cheng-Zhou20}. Finally, we note that the min-max theory for free boundary minimal hypersurfaces, which is a special class of capillary minimal hypersurfaces with $\pi/2$ contact angle, has also enjoyed significant advancements recently; we refer to \cite{Li-Zhou16, Wang20, Sun-Wang-Zhou20} for more details. 

\subsection{Overview of the proof}

Our min-max theory largely follows the Almgren-Pitts procedure. (The reader may consult \cite{Zhou-Zhu17} for a detailed overview of typical min-max methodology. In this procedure, a min-max \textit{width} is defined over discrete families of generalised surfaces. The space of generalised surfaces is chosen rather large to guarantee a weak limit, which transfers the core problem from one of existence to regularity of the limit. 

The regularity theory for capillary surfaces is relatively undeveloped compared to the regularity theories for closed CMC hypersurfaces or even free boundary minimal hypersurfaces, and there are new behaviours which differ markedly from the better-established theories; see for instance Figure \ref{figure.different.boundary.behaviors} or Theorem \ref{thm:tangent-varifolds}. The latter theorem classifies all possible tangent cones to the min-max limit, and is an essential step in the regularity theory. To prove it, we cannot rely on Allard regularity or reflection techniques as in other min-max theories. Instead, we devise an argument based on replacements and the monotonicity formula. The classification also accounts for nontrivial tangent cones which are new for capillary surfaces: 1 or 2 half-planes meeting at the prescribed angle $\theta$, or an $\mathbb{N}$ or $\mathbb{N}+\cos \theta$ multiple of the barrier plane. 

In our present setup, we use \textit{capillary boundaries} of Caccioppoli sets as our generalised surfaces (see Definition \ref{D:capillary boundary varifold}). Roughly, these involve adding a $\cos\theta$ multiple of a portion $B \subset\pr M$ of the ambient boundary; the resulting varifolds actually behave as varifolds with free boundary. This allows us to incorporate certain technical innovations from \cite{Zhou-Zhu17, Zhou-Zhu18} for Caccioppoli boundaries, and \cite{Li-Zhou16} for free boundary. 

The proof of regularity proceeds via a replacement method for surfaces which are almost minimising (in annular regions) with respect to the $\A$ functional. These replacements are limits of local minimisers, which must be regular. A key step is to prove a compactness (Theorem \ref{theorem.compactness.almost.properly.embedded.capillary.surfaces}) that shows the replacements retain this regularity in the limit. This compactness relies on curvature estimates for stable capillary surfaces, which to the authors' knowledge were not available prior to our work. Nevertheless, we were able to prove the stable Bernstein Theorem \ref{theorem.stable.bernstein} which implies the required estimates. 

One typically proceeds by taking overlapping \textit{secondary replacements}, to extend the first replacements towards the centre of the annular region, and also to provide continuity with the original min-max varifold $V$. These gluing procedures can be simplified greatly on the interior using the CMC regularity theory of the second and third named authors \cite{Zhou-Zhu17}. However, due to addition of the boundary portion $B$ when $\theta \neq \frac{\pi}{2}$, the replacements may not \textit{a priori} contain interior points at the gluing interface. To deal with this, we retain the double-replacement method, and devise a novel gluing method at the boundary to ensure preliminary $C^1$ regularity up to the boundary; the classical maximum principle then rules out components totally supported in the boundary $\pr M$. (See Step 2 of Theorem \ref{T:boundary regularity}.)

The extension \textit{across} the annular centre requires a removable singularity theorem for capillary surfaces, which again was not available to the authors' knowledge. To establish this key result, we adapt a technique of White \cite{White87space} to the capillary setting (see Theorem \ref{theorem:removal.singularity}) to establish regularity across an isolated boundary singularity. We then complete the proof of regularity by showing that the replacements (now fully regular) locally coincide with the original min-max limit.

\subsection*{Outline of the paper}

In Section \ref{S:preliminaries}, we outline our basic notions and prove several preliminary results, including the Bernstein-type Theorem \ref{thm:bernstein-intro}, a removable singularities theorem for capillary surfaces, and certain maximum principles. In Section \ref{S:setups}, we describe our notion of `generic' metric, as well as compactness properties of regular capillary varifolds. In the remainder of that section, we detail the min-max setup including the statement of the main min-max Theorem \ref{T:main1}, prove the existence of nontrivial sweepouts and describe the pull-tight process. 

In Section \ref{S:AM} we define a capillary notion of `almost-minimising' and construct replacements. Then, in Section \ref{S:existence}, we show that the pull-tight process yields a min-max limit which has bounded first variation and is almost-minimising in small annuli. Finally, in Section \ref{S:regularity}, we prove the main regularity Theorem \ref{T:boundary regularity}.

\subsection*{Acknowledgements}

The authors are very grateful to Brian White for helpful discussions, and to Guido De Philippis for informing us of their independent progress on this subject.

CL was supported by National Science Foundation under grant DMS-2005287. XZ was supported by NSF grant DMS-1945178, 
and an Alfred P. Sloan Research Fellowship. JZ was supported in part by the National Science Foundation under grant DMS-1802984 and the Australian Research Council under grant FL150100126.

\section{Preliminaries}
\label{S:preliminaries}

In this section we outline some basic background, and also detail certain preliminary results on capillary CMC surfaces. This includes our proofs of the stable Bernstein theorem for capillary minimal surfaces (Theorem \ref{theorem.stable.bernstein}) and a removable singularity result (Theorem \ref{theorem:removal.singularity}).

\subsubsection*{Basic notations}

\begin{itemize}

\item $\C(M)$ or $\C(U)$ the space of sets $\Om\subset M$ or $\Om\subset U\subset M$ with finite perimeter (Caccioppoli sets), \cite[\S 14]{Simon83}\cite[\S 1.6]{Giusti84};

\item $\partial\Om$ denotes the (reduced)-boundary of $[[\Om]]$ as an integral current, and $\nu_{\partial\Om}$ denotes the outward pointing unit normal of $\partial \Om$, \cite[14.2]{Simon83};

\item $\X(M)$: set of tangent vector fields on $M$; $\X_{tan}(M)$: tangent vector fields on $M$ that are tangent to $\partial M$ along $\partial M$; \cite[(2.2)]{Li-Zhou16}. For a given relatively open subset $U\subset M$, $\X_{tan}(U)$ is the set of vector fields $X\in \X_{tan}(M)$ that is supported in $U$;

\item The topological interior of a set $U$ is denoted by $\mathring{U}$;

\item Given a compact subset $K\subset M$, the relative interior $\interior_M(K)$ (or sometime abbreviated as $\interior(K)$) denote the interior of $K$ with respect to the relative topology on $M$. The relative boundary $\partial_{rel}S$ is the set of points in $M$ which are neither in $\interior_M(S)$ nor $\interior_M(M\setminus S)$. 
\end{itemize}

\subsubsection*{Notions related to Capillary surfaces} $(M^{n+1},g)$ is a smooth, compact, oriented Riemannian manifold with boundary. We can always assume that $M$ is a subset of some closed Riemannian manifold $\widetilde{M}$ of the same dimension such that $M\subset \widetilde{M}$. Assume that $\widetilde{M}$ is embedded in some $\R^L$, $L\in\N$. Let $\Sigma^n$ be an orientable $n$ dimensional compact manifold with non-empty boundary $\partial \Sigma$ and $\partial \Sigma\subset \partial M$. Assume $\Sigma$ separates $\mathring{M}$ into two connected components. Fix one component and call it $\Omega$. Denote $\overline{\eta}$ the outward pointing unit normal vector field of $\partial M$ in $M$, $\nu$ the unit normal vector field of $\Sigma$ in $\Omega$ pointing outward $\Omega$, $\eta$ the outward pointing unit normal vector field of $\partial \Sigma$ in $\Sigma$, $\overline{\nu}$ the unit normal vector field of $\partial\Sigma$ in $\partial M$ pointing outward $\Omega$. Let $A$ denote the second fundamental form of $\Sigma\subset \Omega$, $\secondfund$ denote the second fundamental form of $\partial M\subset M$. We take the convention that $A(X_1,X_2)=\bangle{\nabla_{X_1}X_2,\nu}$. Denote $H,\overline{H}$ the mean curvature of $\Sigma\subset \Omega$, $\partial M\subset M$, respectively. Note that in our convention, the unit sphere in $\R^3$ has mean curvature $2$.

\begin{figure}[htbp]
    \centering
    \includegraphics{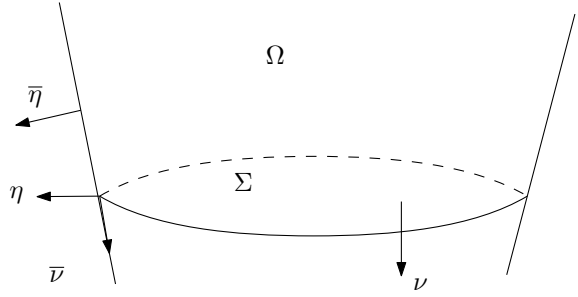}
    \caption{Capillary surfaces}
    \label{picture.capillary.notations}
\end{figure}

\subsubsection*{Several notions of balls} $B_r(p)$ and $\An_{s, r}(p)$ denote respectively Euclidean open ball of radius $r$ and Euclidean open annulus $B_r(p)\setminus \Clos(B_s(p))$ centered at $p$ in $\mathbb{R}^L$. $\widetilde{B}_r(p)$ will denote the open geodesic ball in $\widetilde{M}$, and $\widetilde{B}^+_r(p)$ denotes the open Fermi half-ball of radius $r$ centered at $p\in \partial M$. (We refer to \cite[Section 2.2 and Appendix]{Li-Zhou16} for the notion related to Fermi coordinates near $\partial M$.) We denote the Fermi half-spheres by $\widetilde{S}^+_r(p) = \partial_{rel}(\widetilde{B}^+_r(p))$. 

We define the (relatively) open annular regions 

\[ \wt{\An}_{s,r}(p) = \begin{cases} \wt{B}_r(p)\setminus \Clos(\wt{B}_s(p)) &, p\in\mathring{M} \\ \wt{B}^+_r(p)\setminus \Clos(\wt{B}^+_s(p)), & p \in \partial M.\end{cases}\]

\subsubsection*{Classical notions in GMT} We denote by $\mH^n$ the $n$-dimensional Hausdorff measure; $\bI_n(M)$ (or $\bI_n(M, \mZ_2)$) the space of $n$-dimensional integral (or mod 2) currents in $\R^L$ with support in $M$; $\Z_{n}(M)$ (or $\Z_n(M, \mZ_2)$) the space of integral (or mod 2) currents $T\in\bI_n(M)$ with $\partial T=0$; $\V_n(M)$ the closure, in the weak topology, of the space of $n$-dimensional rectifiable varifolds in $\R^L$ with support in $M$; $G_n(M)$ the Grassmannian bundle of un-oriented $n$-planes over $M$; $\F$ and $\M$ respectively the flat norm \cite[\S 31]{Simon83} and mass norm \cite[26.4]{Simon83} on $\bI_n(M)$; $\mF$ the varifold $\mF$-metric on $\V_n(M)$ 
\cite[2.1(19)(20)]{Pitts}.

\subsection{Capillary surfaces}

Capillary surfaces are critical points of the following weighted area functional defined on $\C(M)$. Given $c\in\mathbb{R}$ and $\theta\in (0, \pi/2]$, define the {\em $\A$-functional} on $\C(M)$ as
\begin{equation}
\label{E: A}
\A(\Om)=\mH^n(\partial\Om\lc \iM) + \cos\theta\,\mH^n(\partial\Om\lc \pM) -c\vol(\Omega). 
\end{equation}

In the variational point of view, it is natural to view the first two terms in the above definition as an integral term.
\begin{definition}[Capillary boundary current/varifold]
\label{D:capillary boundary varifold}
Fix a capillary contact angle $\theta\in (0, \frac{\pi}{2}]$. Given a Caccioppoli set $\Omega\in\C(M)$, the {\em capillary boundary current}, denoted as $\partial^\theta \Omega$, is defined as
\begin{equation}
    \partial^\theta\Omega = \partial\Omega\lc \iM + \cos\theta \cdot \partial\Omega\lc \pM.
\end{equation}
The {\em capillary boundary varifold}, denoted as $|\partial^\theta \Omega|$, is defined by
\begin{equation}
    \int_{G_n(M)}f(x, S)d|\partial^\theta \Omega| = \int_{\partial\Omega\lc \iM} f(x, T_x(\partial \Omega))d\mH^n + \cos\theta \int_{\partial\Omega\lc \pM} f(x, T_x(\partial M)) d\mH^n.
\end{equation}
We will also use $d\mu_\theta$ to denote the associated Radon measure of $|\partial^\theta\Om|$; that is
\begin{equation}
    \int_M f(x) d\mu_\theta(x) = \int_{\partial\Omega\lc \iM} f(x)d\mH^n(x) + \cos\theta \int_{\partial\Omega\lc \pM} f(x) d\mH^n(x). 
\end{equation}
\end{definition}

Using these notations, the $\A$-functional can be written as
\begin{equation}
    \label{E:A'}
    \A(\Om)=\M(\partial^\theta\Omega) -c\vol(\Omega). 
\end{equation}

By \cite[Section 5]{Marques-Neves18}, $\C(M)$ is identified with $\bI_{n+1}(M, \mZ_2)$. In particular, the flat $\F$-norm and the mass $\M$-norm are the same on $\C(M)$. Given $\Om_1, \Om_2\in \C(M)$, we define the {\em $\mF$-distance} between them as:
\[ \mF(\Om_1, \Om_2)=\F(\Om_1-\Om_2)+\mF(|\partial\Om_1\lc \iM|, |\partial\Om_2\lc \iM|). \]
Note that a sequence $\Om_i\to \Om_\infty$ converges under the $\mF$-distance if and only if $\Om_i\to\Om_\infty$ weakly as currents and the interior pieces $|\partial\Om_i\lc \iM|$ converge as varifolds to $|\partial\Om_\infty\lc \iM|$.

We have the following simple lemma.

\begin{lemma}
\label{L:weak convergence of reduced currents}
Suppose that a sequence $\Omega_i\in \C(M)$ converges to some limit $\Omega_\infty\in\C(M)$ under the $\mF$-distance, then $\partial^\theta\Omega_i$ converges weakly to $\partial^\theta\Omega_\infty$ as $n$-currents as well as under the $\mF$-distance. 
\end{lemma}
\begin{proof}
The weak convergence follows from the following decomposition
\begin{equation}
\label{eq:boundary current decomposition}
    \partial^\theta\Omega = (1-\cos\theta)\cdot \partial\Omega\lc \iM + \cos\theta \cdot \partial\Omega,
\end{equation}
and the fact that $\partial \Omega_i\lc\iM$ converges weakly as $n$-currents to $\partial\Omega_\infty \lc \iM$. (In fact, since $|\partial\Om_i\lc \iM|\to|\partial\Om_\infty\lc \iM|$ as varifolds, we know that the weak limit (as currents) of $\partial\Om_i\lc\iM$ has no support on $\pM$, and hence the weak limit must be $\partial\Omega_\infty \lc \iM$.)

As a result, we also have the weak convergence $\partial\Om_i\lc \pM \to \partial\Om_\infty \lc \pM$. The convergence under the $\mF$-distance is a direct consequence of the convergence of $\partial\Omega_i\lc\pM \to \partial\Omega_\infty\lc \pM$ in the $\M$-norm, which is the same as the weak convergence of $\partial \Omega_i\lc\pM$. (Note that the flat norm $\F$ is the same as the $\M$-norm for $n$-currents supported on $\pM$.) 
\end{proof}

The following result follows as a direct corollary of Lemma \ref{L:weak convergence of reduced currents} using compactness, c.f. \cite[Lemma 4.1]{Marques-Neves14}.
\begin{corollary}
\label{C:various closeness}
Let $\mathcal S$ be a subset of $\C(M)$ which is compact under the $\mF$-distance. Then for any $\ep>0$, there exists a $\de>0$, such that for every $\Om\in \C(M)$ and $\Om'\in\mathcal S$
\[ \mF(\Om, \Om')<\de \Longrightarrow \F(\partial^\theta\Om, \partial^\theta\Om') + \mF(|\partial^\theta\Om|, |\partial^\theta\Om'|) <\ep. \]
\end{corollary}

Note that $\partial^\theta: \C(M)\to \mathcal R_n(M)$ is continuous map from $\C(M)$ (under the $\mF$-distance topology) to the space of $n$-dimensional rectifiable currents $\mathcal R_n(M)$; (the image $\partial^\theta\Omega$ is not integer rectifiable as it has fractional coefficient along $\pM$). Clearly, the operation $\partial^\theta$ commutes with pushing forward by any boundary preserving diffeomorphisms, which we record as:
\begin{proposition}
\label{L:communicates lemma}
Let $f: M\to M$ be a diffeomorphism of $M$ such that the restriction $f: \pM\to\pM$ is a diffeomorphism of $\pM$. Then for any $\Om\in\C(M)$
\[ \partial^\theta f(\Omega) = f(\partial^\theta\Omega). \]
\end{proposition}

\vspace{1em}
The {\em first variation formula} for $\A$ along $X\in \X(M)$ is (see \cite[16.2]{Simon83}) 
\begin{equation}
\label{E: 1st variation for A}
\begin{split}
\de\A\vert_{\Om}(X)
     & = \int_{\partial\Om \lc \iM}\Div_{\partial \Om}X d\mu_{\partial\Om} + \cos\theta\cdot \int_{\partial\Om\lc \pM}\Div_{\partial\Om}X d\mu_{\partial\Om} - c \int_{\partial\Om} \langle X,\nu\rangle \, d\mu_{\partial\Om} \\
     & = \int_{\partial\Omega}\Div_{\partial \Om}X d\mu_\theta - c \int_{\partial\Om} \langle X,\nu\rangle \, d\mu_{\partial\Om},
\end{split}
\end{equation}
where $\nu= \nu_{\partial \Om}$ is the outward unit normal on $\partial \Om$. 

Finally, we recall the following classical regularity result for $\A$-minimizing currents.
\begin{theorem}[\cite{Taylor77},\cite{DePhilippi-Maggi15}]\label{theorem.regularity.of.minimizers}
    Suppose $M^3$ is a smooth manifold with boundary, $U$ a relatively open subset of $M$, and $\Omega\in \C(M)$ minimizes $\A$ in $U$. Let $\Sigma=\partial\Omega\cap \mathring{M}$. Then $\Sigma$ is an embedded surface in $U$. Moreover, near each point on $\overline \Sigma\cap \partial M$, $\Sigma$ is properly embedded and meets $\partial M$ transversely at angle $\theta$.
\end{theorem}

\subsection{Boundary maximum principle and Monotonicity Formula}
Suppose $\Sigma^n\subset M^{n+1}$ is a capillary CMC hypersurface with constant contact angle $\theta$ and mean curvature $c$. Suppose $p\in \partial \Sigma$, and that $\Sigma$ is properly embedded in a neighborhood $U$ of $p$. Let $W\subset \partial U$ be the region separated by $\partial \Sigma$ inside $\partial M$, and that $\nu_\Sigma$ points outward $W$. Define a varifold $V=|\Sigma|+\cos \theta|W|$. Then $V$ is a varifold with bounded first variation and free boundary (compare to \cite[Definition 2.1]{Li-Zhou16}). Denote $\rho$ the distance to $p$ in $M$. Then there exist $\delta,\alpha,\rho_0>0$, such that the quantity (see \cite[(2.6)]{Simon1980capillary})
\begin{equation}\label{eq.monotonicity}
    \exp(\delta\rho^\alpha)\frac{\|V\|(B_\rho(p))}{\rho^n}\text{ is increasing in } \rho\in (0,\rho_0).
\end{equation}

In general, if $\Omega$ is a Caccioppoli set in $M$ and is stationary for $\A$ (see Definition \ref{D:capillary boundary varifold}), $V=|\partial^\theta \Omega|$, and $p\in \spt(\partial \Omega\lc \mathring{M})\cap \partial M$, then there exist $\delta,\alpha,\rho_0>0$ such that \eqref{eq.monotonicity} holds.

We will need a maximum principle for free boundary varifolds with $c$-bounded first variation, which follows from the proof of \cite[Theorem 1.1]{Li-Zhou-MP} with trivial modifications - one may use precisely the same argument, noting that (in their notation) $\tr_P A^{S'_s}>c$.

\begin{theorem}
\label{thm:vfld-mp}
Suppose $V\in \mathcal{V}_n(M)$ is a free boundary varifold with $c$-bounded first variation in a relatively open subset $U\subset M$, where $c\geq 0$. Suppose $K\subset\subset U$ is a smooth, relatively open connected subset such that 
\begin{enumerate}
\item $\partial_{rel}K$ meets $\partial M$ orthogonally;
\item the mean curvature of $\partial_{rel}K$ with respect to the outward normal is greater than $c$;
\item $\spt\|V\|\subset \Clos(K)$.
\end{enumerate}
Then $\spt\|V\|\cap \partial_{rel}K=\emptyset$. 
\end{theorem}

\subsection{Almost properly embedded capillary CMC surfaces}
\begin{definition}
Suppose $\Sigma^n\subset M^{n+1}$ is a smooth, immersed, two-sided surface with unit normal vector $\nu$ such that $\partial \Sigma\subset \partial M$. We say that $\Sigma$ is a capillary CMC surface if the mean curvature with respect to $\nu$ is equal to $c$, and the angle between $\nu$ and $\overline{\eta}$ along $\partial \Sigma$ is everywhere equal to $\theta$. We further say that $\Sigma$ is a \textit{stable} capillary CMC surface if
\begin{equation}\label{equation.stability.capillary}
    Q(f):=-\int_{\Sigma}(f\Delta f+(|A|^2+\Ric(\nu,\nu))f^2)d\mH^2+\int_{\partial\Sigma\cap \partial M}f\left(\pa{f}{\eta}-Qf \right)d\mH^1\ge 0,
\end{equation}
for all $f\in C^2(\Sigma)$, where
\[Q=\frac{1}{\sin \theta}II(\overline{\nu},\overline{\nu}) - (\cot\theta) A(\eta,\eta).\]
\end{definition}

\begin{definition}[Almost proper embedding]
Let $U\subset M$ be a relatively open subset, and $\Sigma^n$ be a smooth $n$-dimensional manifold with boundary. A smooth immersion $\phi: \Sigma\rightarrow U$ is said to be {\em an almost proper embedding}, if $\phi(\partial \Sigma)\subset \partial M$, and at any point $p\in \phi(\Sigma)$, the following holds:
\begin{enumerate}
    \item If $p\in \phi(\partial \Sigma)$, there exists a small neighborhood $W\subset U$ of $p$, such that:
    \begin{itemize}
        \item $\Sigma\cap \phi^{-1}(W)$ is a disjoint union of connected components $\sqcup_{i=1}^l \Sigma_i$;
        \item for each $i=1,\cdots,l$, $\Sigma_i\cap \partial \Sigma\ne \emptyset$, and $\phi(\Sigma_i)$ is a proper embedding in $W$.
        \item for each $i$, any other component $\phi(\Sigma_j)$, $j\ne i$, lies on one side of $\phi(\Sigma_i)$ in $W$.
    \end{itemize}
    \item If $p\in \phi(\mathring{\Sigma})$, there exists a small neighborhood $W\subset U$ of $p$, such that:
    \begin{itemize}
        \item $\Sigma\cap \phi^{-1}(W)$ is a disjoint union of connected components $\sqcup_{i=1}^l \Sigma_i$;
        \item for each $i=1,\cdots, l$, $\Sigma_i\subset\mathring{\Sigma}$, and $\phi(\Sigma_i)$ is an embedding in $W$;
        \item for each $i$, any other component $\phi(\Sigma_j)$, $j\ne i$, lies on one side of $\phi(\Sigma_i)$ in $W$.
    \end{itemize}
\end{enumerate}
\end{definition}

With slight abuse of notation, we will denote $\phi(\Sigma)$, $\phi(\mathring{\Sigma})$, $\phi(\partial \Sigma)$ by $\Sigma$, $\mathring{\Sigma}$, $\partial \Sigma$, respectively. From the definition, we see that $\partial\Sigma\subset \partial M$, but it could happen that $\mathring{\Sigma}\cap \partial M\ne\emptyset$. We denote the subset of points in $\Sigma$ where $\Sigma$ fails to be embedded by $\mS(\Sigma)$. Finally, define $\mathcal{T}(\Sigma) = \mathring{\Sigma} \cap \pr M$. Note that $\mathcal{S}(\Sigma)$ describes the \textit{self-touching} set, while $\mathcal{T}(\Sigma)$ describes the \textit{barrier-touching} set.

We will call $\Sigma\setminus \mS(\Sigma)$ the regular set, and denote it by $\mR(\Sigma)$. Consider a point $p\in \mS(\Sigma)$, and $\{\Sigma_j\}$ the collection of components meeting at $p$. From its definition, if $p\in \mathring{\Sigma}$, then $\{\Sigma_j\}$ meet tangentially at $p$; if $p\in \partial \Sigma$, then the $\{\partial \Sigma_j\}$, as hypersurfaces in $\partial M$, meet tangentially at $p$. 
In particular, the tangent cone of $\Sigma$ at $p$ is a collection of half spaces through $p$ containing a common subspace $\R^{n-1}$.

\begin{definition}[Almost embedded capillary boundary]
\begin{enumerate}
    \item An almost embedded hypersurface $\Sigma\subset U$ is said to be a {\it boundary} if there exists an open subset $\Omega\subset \C(U)$, such that $\Sigma=\partial \Omega\lc \mathring{U}$ in the sense of currents;
    \item The outer unit normal $\nu_\Sigma$ of $\Sigma$ is the choice of the unit normal of $\Sigma$ which points outside of $\Omega$ along the regular part $\mR(U)$;
    \item $\Sigma$ is called a {\it stable capillary boundary} if $\Sigma$ is a boundary, and the associated $\Omega\in \C(U)$ satisfied $\td{^2}{t^2}|_{t=0}\A(\phi_t(\Omega))\ge 0$ for any $\phi_t$ generated by $X\in \X_{tan}(U)$.
\end{enumerate}
In either case, we call a connected component of $\Sigma\cap \phi^{-1}(W)$ a \textit{sheet} near $p$.
\end{definition}

Observe that if $\Sigma=\partial\Omega\lc \mathring{M}$ is an almost embedded capillary CMC boundary, then at any point $p\in \partial \Sigma$, the tangent plane to each properly embedded sheet is a half plane meeting $T_p (\partial M)$ at constant angle $\theta$. By the strong maximum principle, there can exist at most two such sheets. Also, $\td{^2}{t^2}|_{t=0}\A(\phi_t(\Omega))\ge 0$ is equivalent to \eqref{equation.stability.capillary}, where $f=\bangle{X,\nu}$ is the normal component of the variation. In particular, each sheet near $p$ is properly embedded and stable - that is, \eqref{equation.stability.capillary} is satisfied. For a derivation of the second variation formula, see the Appendix of \cite{RosSouam97capillarystability}.

We proceed to prove curvature estimates for capillary surfaces. We start with the following stable Bernstein theorem for capillary minimal surfaces in $\R^3$.
\begin{theorem}[$=$ Theorem \ref{thm:bernstein-intro}]\label{theorem.stable.bernstein}
    Let $\R^3_+$ denote the half space $\{(x_1,x_2,x_3):x_1\ge 0\}$, and suppose $\Sigma\subset \R^3_{+}$ is a properly immersed, two-sided capillary minimal surface. Suppose also that $\Sigma$ has Euclidean area growth, that is, there exists some $C>0$ such that
    \[Area(\Sigma\cap B_r(0))< C r^2\]
    for any $r>0$, and that $\Sigma$ is stable for the capillary functional. Then $\Sigma$ is planar.
\end{theorem}

\begin{proof}
    Let $\overline{\nabla}$ denote the connection given by the Euclidean metric, and $\nabla$ the connection on $\Sigma$. The vector field $-e_1$ is the outward unit normal vector field to $\R^3_+$ on $\{x_1=0\}$. Let $w=\bangle{e_1,\nu}$. Notice that $w$ is a Jacobi function on $\Sigma$. Moreover, along $\{x_1=0\}$, $w=\cos\theta$, and $\partial_\eta w=\partial_\eta \bangle{e_1,\nu}=\bangle{e_1,\nabla_\eta \nu}=A(\eta,\eta)\bangle{e_1,\eta}=\sin\theta A(\eta,\eta)$. Set $\varphi=1-w\cos\theta$. Then $\varphi|_{\partial \Sigma}=\sin^2\theta$, and $\partial_\eta \varphi=-\cos\theta \sin\theta A(\eta,\eta)$. Thus, we have that
    \begin{equation}
        \begin{cases}
            \Delta \varphi + |A|^2\varphi = |A|^2 \quad &\text{in }\Sigma,\\
            \partial_{\eta} \varphi= Q \varphi \quad &\text{on }\partial \Sigma.
        \end{cases}
    \end{equation}
    Here $Q=-\cot\theta A(\eta,\eta)$ on $\partial \Sigma$. 
    
    For any compactly supported function $f\in C^1(\R^3_+)$, plug $\varphi f$ into the stability inequality:
    \[0\le -\int_\Sigma \varphi f (\Delta+|A|^2)(\varphi f) + \int_{\partial \Sigma }\varphi f (\partial_\eta-Q)(\varphi f). \]
    Now
    \[\int_{\partial \Sigma} \varphi f (\partial_\eta -Q)\varphi f = \int_{\partial \Sigma} \varphi f [(\partial_\eta \varphi )f -Q\varphi f +\varphi \partial_\eta f]=\int_{\partial \Sigma}\varphi^2 f \partial_\eta f.\]
    Perform integration by parts for the other term:
    \begin{equation*}
        \begin{split}
            \int_\Sigma \varphi f (&\Delta+|A|^2)(\varphi f)\\
                &= \int_\Sigma \varphi f [(\Delta \varphi) f +\varphi \Delta f +2\bangle{\nabla \varphi,\nabla f}+|A|^2 \varphi f ]\\
                &= \int_\Sigma \varphi f(|A|^2 f+\varphi \Delta f+2\bangle{\nabla \varphi, \nabla f })\\
                &= \int_\Sigma |A|^2\varphi f^2 + \int_\Sigma \varphi^2 f\Delta f+\frac{1}{2}\int_\Sigma \bangle{\nabla (\varphi^2),\nabla (f^2)}\\
                &= \int_\Sigma |A|^2 \varphi f^2+ \int_\Sigma \varphi^2 f\Delta f +\int_{\partial \Sigma} \varphi^2 f\partial_\eta f- \int_\Sigma \varphi^2 (f\Delta f+|\nabla f|^2).
        \end{split}
    \end{equation*}
    
    Thus the stability inequality gives
    \begin{equation}\label{equation.consequence.of.stability}
        \int_\Sigma |A|^2\varphi f^2 \le \int_\Sigma \varphi^2 |\nabla f|^2
    \end{equation}
    for all compactly supported Lipschitz functions $f$. Since the function $\varphi$ satisfies the bound $1-|\cos \theta|\le \varphi\le 1+|\cos\theta|$,
    \eqref{equation.consequence.of.stability} gives
    \begin{equation}\label{equation.almost.stability}
        \int_\Sigma |A|^2 f^2 \le C_\theta \int_\Sigma |\nabla f|^2
    \end{equation}
    with $C_\theta=\frac{(1+|\cos\theta|)^2}{1-|\cos\theta|}$. Now we may take $f$ to be the standard logarithmic cutoff function and conclude that $|A|\equiv 0$.
\end{proof}

\begin{remark}
\label{rmk:souam}
Our use of the test functions above was inspired by their use in \cite{Ainouz-Souam16} due to Ainouz and Souam. During the completion of this article, Hong-Saturnino \cite{HS21} and Souam \cite{souam2021stable} also used the same test functions to study stable capillary surfaces with planar boundary. In particular, Theorem \ref{theorem.stable.bernstein} was proven independently in \cite{HS21} using Fischer-Colbrie--Schoen techniques. (The results of \cite{souam2021stable} strictly only apply to compact surfaces.) Our approach to Theorem \ref{theorem.stable.bernstein} is simpler than that of \cite{HS21} and also admits generalisations to higher dimensions (for certain contact angles) via Schoen-Simon-Yau techniques, detailed in Appendix \ref{sec:bernstein}. 
\end{remark}

Using Theorem \ref{theorem.stable.bernstein} and a blow-up argument analogous to \cite[Theorem 3.2]{GuangWangZhou2018compactness}\footnote{The subtlety here is that, a priori, the blow-up points may converge onto the interior touching set; however, since the $\partial M$ converges smoothly to a plane, this cannot happen due to the maximum principle. See \cite{GuangWangZhou2018compactness}.}, we have the following curvature estimate for stable capillary constant mean curvature surfaces.

\begin{theorem}\label{theorem.curvature.estimate}
Let $M^3$ be a Riemannian manifold with boundary, $U\subset M$ be an open subset. Suppose $(\Sigma,\partial \Sigma)\subset (U,\partial M\cap U)$ is a smooth properly immersed stable two-sided capillary constant mean curvature surface with $\Area(\Sigma)<C_0$. Then
\[|A_\Sigma|^2(x)\le \frac{C_1}{\dist^2_M(x,\partial U)} \quad\text{for all }x\in \Sigma\]
where $C_1>0$ is a constant depending on $C_0, M, U$, the mean curvature of $\Sigma$ and the angle between $\Sigma$ and $\partial M$.
\end{theorem}

Next, we give a precise local description of $\mS(\Sigma)$. We start with the following a maximum principle.

\begin{lemma}[Strong maximum principle for embedded capillary CMC surfaces]\label{lemma.strong.max.principle.embedded}
Suppose $U\subset M$ is an open subset, $\Sigma_i\subset U$, $i=1,2$, are two connected properly embedded constant mean curvature surfaces, with the same mean curvature $c> 0$ with respect to unit normal $\nu_i$ and the same contact angle $\theta$ along $\partial M$. Assume that $\Sigma_2$ lies on one side of $\Sigma_1$ and that $p \in \Sigma_1 \cap  \Sigma_2 \neq \emptyset$. Then we have the following:
\begin{enumerate}
    \item If $\nu_1(p)=\nu_2(p)$, then $\Sigma_1=\Sigma_2$;
    \item Let $U_1\subset U$ be the region divided by $\Sigma_1$ where $\nu_1$ points outward of $U_1$.  Suppose that $\Sigma_2$ lies in the closure of $U_1$. 
    If $p$ is an interior point, then $\Sigma_1 = \Sigma_2$. 
    If $p$ is a boundary point, then either $\Sigma_1=\Sigma_2$, or $\nu_2(p)$ is the reflection of $\nu_1(p)$ across an axis parallel to $T_p\pr M$. 
    

\end{enumerate}
\end{lemma}

\begin{proof}

If $p$ is in the interior of $M$, then both items follow from the interior strong maximum principle as in \cite[Lemma 2.7]{Zhou-Zhu17}. Henceforth we consider the case where $p\in \pr M$. 

Take a Fermi coordinate system $\{x_i\}$ in a neighbourhood $U$ of $p$ so that $M$ is identified with $\{x_1<0\}$, and $p$ is identified with $x=0$. By a rotation, we may assume that $\nu_1(p) = (\cos \theta) e_1 + (\sin \theta) e_3,$ and write each $\Sigma_j \cap U$, as a graph $x_3 = u_j(x_1,x_2)$, for $j=1,2$.

For item (1), we assume $\nu_1(p)=\nu_2(p)$. Similar to \cite{Zhou-Zhu17}, it follows that the difference $u= u_2-u_1$ satisfies $Lu = 0$, where $L$ is a positive linear elliptic operator. We may assume without loss of generality that $u_2 >u_1$, hence $u>0$, on $\{x_1<0\}$. Writing the normals $\nu_j$ in terms of $u_j$, one can see that $\nu_1(p)=\nu_2(p)$ implies $\pr_1 u_1(0) = \pr_1 u_2(0)$, hence $\pr_1 u(0)=0$. The Hopf lemma then gives that $u=0$ on $U$, hence $\Sigma_1=\Sigma_2$ by unique continuation. 

For item (2), since $\Sigma_2$ lies in the closure of $U_1$, we see that $T_p \Sigma_2$ must lie in the negative half-space of $T_p M$ with respect to $\nu_1(p)$. Then since $\Sigma_2$ is a capillary surface, the only options for its outer normal are $\nu_2(p) = (\cos \theta) e_1 \pm (\sin \theta) e_3$. Note that $(\cos \theta) e_1 - (\sin \theta) e_3$ is precisely the reflection of $\nu_1(p)$ across the $x_3$-axis. On the other hand, if $\nu_2(p) = (\cos \theta) e_1 + (\sin \theta) e_3 = \nu_1(p)$, then item (1) gives $\Sigma_1=\Sigma_2$ as claimed. 
\end{proof}

\begin{remark}\label{remark.regular.and.singular.parts.almost.embedded.surface}
    From Lemma \ref{lemma.strong.max.principle.embedded} and its proof, we see that for any almost properly embedded capillary CMC boundary $\Sigma^n=\partial\Omega\lc \mathring{M}$ and $p\in \mS(\Sigma)$, we have the following. If $p\in \partial \Sigma$, then locally two sheets $\Sigma_1,\Sigma_2$ meet transversely at $p$; if $p\in \mathring{\Sigma}$, then there are two sheets, $\Sigma_1,\Sigma_2$, touching each other at $p$ with $\nu_1=-\nu_2$. Since locally each sheet $\Sigma_j$ is a solution to the CMC equation, we conclude, as in \cite[Lemma 2.8]{Zhou-Zhu17} that $\mS(\Sigma)$ is contained in a finite union of $1$-dimensional submanifolds
\end{remark}


\subsection{Removable singularities}

Using Theorem \ref{theorem.curvature.estimate}, we establish a removable singularity theorem for Caccioppoli sets that are stationary for the $\A$ functional and are stable away from a single point.

\begin{theorem}
\label{theorem:removal.singularity}
Let $(M,g)$ be a smooth $3$-manifold with boundary, $U$ an open subset of $M$ and $p\in \partial M\cap U$. Suppose $\Om$ is a Caccioppoli set in $U$ such that:
\begin{enumerate}
    \item  $\Sigma:=\partial\Om\lc \left(U\setminus \{p\}\right)$ is a connected, properly embedded, stable capillary constant mean curvature surface in $U\setminus \{p\}$, and $p\in \spt \Sigma$.
    \item Some tangent cone of $\Sigma$ at $p$ is a multiplicity one capillary half plane in $\R^3_+$.
\end{enumerate}
Then $\Omega$ is stationary for $\A$, and $\Sigma$ extends smoothly across $\{p\}$ as a properly embedded surface.
\end{theorem}

\begin{remark}
    When $\theta=\pi/2$, 
    hypothesis (1) in Theorem \ref{theorem:removal.singularity} may be replaced by the assumption that $\Omega$ is $\A$-stationary in $U$ (in particular, without assuming stability), thanks to the classical Allard regularity theorem for free boundary stationary varifolds (\cite{GruterJost1986allard}). It is tempting to conjecture that an Allard type regularity theorem holds for capillary submanifolds with general contact angles. 
\end{remark}


\begin{proof}[Proof of Theorem \ref{theorem:removal.singularity}]
We divide the proof into several steps. 

\vspace{0.5em}
\textbf{Step 1}: Stationarity of $\Omega$ for $\A$.
\vspace{0.5em}

We first prove that $\Omega$ is stationary for the $\A$ functional in $U$. By assumption (2), there exists a sequence $\{r_j\}_{j=1}^\infty$ such that ${\eta_{r_j}}_*(\Sigma)$ converge as varifolds to a limit supported on a half planes in $\R^3$, here $\eta_{r_j}(x):=r_j^{-1}(x-p)$. In particular, for $j\gg 1$, $\|\Sigma\|(B_{r_j}(p))<2\pi r_j^2$. On the other hand, we have that $\|\partial \Omega\lc \partial M\|(B_{r_j}(p))\le \|\partial M\|(B_{r_j}(p))=\pi r_j^2$. Thus, there is $C_0$ independent of $j$ such that
\begin{equation}\label{E:density.ratio.upper.bound}
    \|\partial^\theta \Omega\|(B_{r_j}(p))<C_0r_j^2.
\end{equation}

Let $X\in \X_{tan}(U)$. Denote by $r$ the geodesic distance on $M$ to $p$. For $\ep>0$, take a function $\varphi_\ep (r)$ such that $\varphi_\ep (r)=0$ when $r\in [0,\ep^2]$, $\varphi_\ep(r)=1$ when $r\ge \ep$, and $|\varphi_\ep'(r)|<\tfrac{2}{\ep}$. Taking $\varphi_\ep (r) X$ into the first variation formula \eqref{E: 1st variation for A}, we have
\[\int_{\partial\Omega} \varphi_\ep (r) \Div_{\partial \Omega} X d\mu_\theta - c\int_{\partial \Omega} \varphi_\ep (r) \langle X,\nu\rangle d\mu_{\partial \Omega} = -\int_{\partial\Omega} \varphi'_\ep (r) \langle \nabla_{\partial \Omega} r ,X\rangle.\]
Note that
\begin{equation*}
    \left|\int_{\partial \Omega} \varphi_\ep'(r) \langle \nabla_{\partial \Omega} r,X\rangle \right| d\mu_\theta= \left|\int_{\partial \Omega\lc B_\ep (p)} \varphi_\ep'(r)\langle \nabla_{\partial \Omega} r,X\rangle d\mu_\theta \right|
    \le \frac{2\|X\|_{C_0(U)}}{\ep} \|\partial^\theta \Omega\|\left(B_\ep(p)\right).
\end{equation*}
Taking $\ep=r_j\searrow 0$ in the above inequality and using \eqref{E:density.ratio.upper.bound}, we conclude that $\delta \A_\Omega(X)=0$ for all $X\in \X_{tan}(U)$, and thus $\Omega$ is stationary in $U$.

\vspace{0.5em}
\textbf{Step 2}: Construction of a capillary CMC foliation near $p$.
\vspace{0.5em}

Denote by $H_0$ the mean curvature of $\Sigma$. Take a sequence $\{r_j\}_{j=1}^\infty$ converging to $0$ such that the rescaled surfaces $\Sigma_j:=r_j^{-1}(\Sigma-p)$ converge in the varifold sense to a capillary half-plane $P$ in $\R^3_+$ with multiplicity one. Moreover, by Theorem \ref{theorem.curvature.estimate}, $\Sigma_j\cap \left(B_1(0)\setminus B_{1/2}(0)\right)$ is properly embedded and has bounded curvature, and thus subsequentially they converge graphically smoothly to $P$ in $B_1(0)\setminus B_{1/2}(0)$ (we continue to denote the subsequence by $\Sigma_j$). Take local Fermi coordinates relative to $\partial M\subset M$. By an affine transformation, we may assume $p=(0,0,0)$, $P=\{x_3=0,x_1\ge 0\}$ and $\partial M=\{x_3=x_1\tan\theta\}.$ Denote by $M_{r_j}=r_j^{-1}(M-p)$. Take the transformed Fermi coordinates $\{x_1,x_2,x_3\}$ near $0\in M_{r_j}$ similar as before, and define $Y_{r_j}\in \X_{tan}(M_{r_j})$ to be the smooth vector field  $Y_{r_j}=\cos\theta \partial_{x_1}+\sin\theta \partial_{x_3}$.

For any fixed angle $\gamma\in (\pi/3,\pi/2)$, we describe a foliation of a neighborhood of $0$ as follows. For each $r\in (0,1)$, consider the domains 
\[T_r:=\left\{(x_1,x_2,0): 0\le x_1\le (r^2-x_2^2)^{1/2}-r\cos \gamma\right\},\]
\[C_r:=\left\{(x_1,x_2,x_3): 0\le x_1\le (r^2-x_2^2)^{1/2}-r\cos \gamma\right\}.\]
We also denote 
\[S_r=\left\{(x_1,x_2,0): x_1\ge 0, (x_1+r\cos\gamma)^2+x_2^2=r^2\right\},\quad \Gamma_r=\{(0,x_2,0):-r\sin\gamma \le x_2\le r\sin\gamma\}.\]
Clearly the set $\{S_r\}_{r\in (0,3/2)}$ gives a foliation of a neighborhood of $p$ on $P$. 

By Theorem \ref{theorem.curvature.estimate}, 
\begin{equation}\label{equation.removal.sing.1}
    r_j^{-1}\left(\Sigma\cap (C_{3r_j/2}(p)\setminus C_{r_j/2}(p))\right)\rightarrow  T_{3/2}(0)\setminus T_{1/2}(0)
\end{equation}
as $C^\infty$ graphs. In particular, there exists a $C^\infty$ function $\overline{w}_j$, such that 
\[r_j^{-1}(\Sigma\cap (C_{5r_j/4}(p)\setminus C_{4r_j/5}(p)))=\{\overline{w}_j(x)Y_{r_j}(x): x\in U_j\},\] 
where $U_j\subset P$ is an open domain that limits to $T_{5/4}\setminus T_{4/5}$. Denote by $w_j=\overline{w}_j|_{S_1}$. For $\alpha\in (0,1)$ to be specified later, it follows from \eqref{equation.removal.sing.1} that
\[\|w_j\|_{C^{2,\alpha}(S_1)}=\varepsilon_j\rightarrow 0.\]
We then extend $w_j$ to be a $C^{2,\alpha}$ function on $T_{1}$ (which we still call $w_j$) with $\|w_j\|_{2,\alpha,T_1}\le 2\varepsilon_j$. By the monotonicity formula, for $j$ sufficiently large, $\spt (\Sigma_j\cap C_{3/2}(0))$  is contained in a $1/2$ tubular neighborhood of $P$.

For $r_j$ sufficiently small, we construct a local foliation $\{\Sigma_j^\rho\}$ of a neighborhood of $0\in M_j$ by properly embedded capillary CMC surfaces (with mean curvature $H_0$) with boundary on $\partial M$, such that $\Sigma_j^0$ is also the graph of $w_j|_{S_1}$ over $S_1$. We do this by the implicit function theorem as follows. 

Given $r\in (0,r_0)$, we consider the rescaled manifold $M_r=r^{-1}(M-p)$. Let $C_0^{2,\alpha}(T_1)=\{u\in C^{2,\alpha}(T_1): u=0 \text{ on }S_1\}$. Given functions $u\in C_0^{2,\alpha}(T_1)$, $w\in C^{2,\alpha}(T_1)$ and $t\in \R$, consider the surface
\[\Sigma_{(u+w+t)}=\{(u(x)+w(x)+t)Y_r(x):x\in T_1\}\]
written as the graph of the function $u+w+t$ along the vector field $Y_r$ inside $M_r$. For sufficiently small $r$, define a function
\[h:\R\times \R\times C^{2,\alpha}(T_1)\times C_0^{2,\alpha}(T_1)\rightarrow C^{0,\alpha}(T_1)\]
by letting $h(r,t,w,u)=H_{(u+w+t)}-H_0r$, where $H_{(u+w+t)}$ equals to the mean curvature of $\Sigma_{(u+w+t)}$ in $M_r$. Similarly, define
\[\xi: \R\times \R\times C^{2,\alpha}(T_1)\times C_0^{2,\alpha}(T_1)\rightarrow C^{1,\alpha}(\Gamma_1)\]
by letting $\xi(r,t,w,u)=\lan \nu_\Sigma, \overline{\eta}_r\ran  - \cos\theta$, where $\nu_\Sigma$ is the outward conormal unit vector field of $\partial \Sigma_{(u+w+t)}$ in $\Sigma_{(u+w+t)}$, and $\overline{\eta}_r$ is the outward unit normal vector field of $\partial M_r$ in $M_r$. We then consider the function 
\[\Theta=(h, \xi):\R\times \R\times C^{2,\alpha}(T_1)\times C_0^{2,\alpha}(T_1)\rightarrow C^{0,\alpha}(T_1)\times C^{1,\alpha}(\Gamma_1).\]

It is straightforward that $h,\xi$ are $C^1$ functions (see the appendix of \cite{White87space}). When $r\rightarrow 0$, the manifold $M_r$ converges in $C^\infty$ Cheeger-Gromov sense to the Euclidean space. Hence (see Lemma A.2 and A.3 of \cite{Li2017polyhedron})
\[D_4(0,t,0,0)(v)=Lv= \left(\Delta v , \sin\theta \frac{\partial v}{\partial \eta}\right).\]
Clearly the linear operator $L: v\in C_0^{2,\alpha}(T_1)\mapsto (\Delta v, (\sin\theta)v_\eta)$ has trivial kernel. Now we fix $\alpha\in (0,\frac{\pi}{2\gamma})$. Since $\gamma\in (\pi/3, \pi/2)$, we may apply the elliptic regularity theory for mixed boundary conditions developed in \cite{AzzamKreyszig82solutions} and the Fredholm alternative in the appendix B of \cite{AmbrozioCarlottoSharp18compactness}, and obtain
\[\|u\|_{2,\alpha}\le C\left (\|\Delta u\|_{0,\alpha, T_1}+\|u_\eta\|_{1,\alpha, \Gamma_1}\right).\]
Thus $L$ is a Banach space isomorphism, and by the implicit function theorem, for each $t\in [-1/2,1/2]$, sufficiently small $r$ and $\|w\|_{C^{2,\alpha}(T_1)}$, there is a unique function $u=U_{r,t,w}$ in $C_0^{2,\alpha}(T_1)$ such that $\Theta(r,t,w,u)=\Theta(0,t,0,0)$. This implies that the surface $\Sigma_{(u+w+t)}$ is a surface with constant mean curvature $H_0r$ that meets $\partial M_r$ at constant angle $\theta$. Now we let $v_r^t=U_{r,t,w}+w+t$, and define
\[\Sigma_r^\rho=\{v_r^\rho(x)Y_r(x):x\in T_1\}.\]
Since $U$ depends $C^1$ on $t$, the surfaces $\Sigma_r^\rho$ foliate $C_r\cap \left\{|x_3|\le \tfrac 12\right\}$, a neighborhood of $0$ in $M_r$, if $r$ is sufficiently small. Moreover, $\partial \Sigma_r^0$ is the graph of $w$ over $S_1$.

\vspace{0.5em}
\textbf{Step 3:} Removable singularity.
\vspace{0.5em}

Fix $r_j$ sufficiently small, we apply the above construction to $M_{r_j}$, and obtain a foliation $\{\Sigma_j^\rho\}_{\rho\in [-1/2,1/2]}$ of the $\tfrac 12$ tubular neighborhood of $p\in M_{r_j}$ by capillary surfaces of constant mean curvature $H_0 r_j$. 
Let $\tilde \rho\in (-\tfrac 12, \tfrac 12)$ be the unique real number such that $\Sigma_j^{\tilde\rho}$ contains $p$. We now claim that the tangent cone of $\Sigma_j$ at $p$ is unique, and is precisely the half-plane $P = \{x_3= 0 ,x_1\geq 0\}$. We split this into two cases: 

Case 1: $\tilde \rho\ge 0$. Let $\rho_0$ be the supremum of all real numbers $\rho$ such that $\overline{\Sigma}_j\cap \Sigma_j^\rho\ne \emptyset$. Since $p\in \overline{\Sigma}_j\cap \Sigma_j^{\tilde \rho}$, we have that $\rho_0\ge \tilde \rho$. However, if $\rho_0> \tilde\rho$, then $\Sigma_j$ and $\Sigma_j^{\rho_0}$ must touch at a point which is not $p$, nor any point on $\partial \Sigma_j\cap \partial C_{r_j}$ (as $\rho_0>\tilde \rho\ge 0$), violating the strong maximum principle. Thus, $\rho_0=\tilde \rho$, and hence $\overline \Sigma_j$ is contained in the closure of the lower side of $\Sigma_j^{\tilde \rho}$ (i.e. $\{x_3\le 0\}$ in our Fermi coordinates). However, there are only two possible such planes in $T_p M_{r_j} \simeq \R^3_+$: either $T_p\Sigma_j^{\rho_0} = P$; or the reflection $\overline{P}$ of $P$ across the line $T_p(\partial \Sigma_j^{\rho_0}) \subset T_p(\partial M_{r_j})$ (i.e. $\{x_3 = \tan(2\theta- \pi)x_1,x_1\ge 0\}$). But the set of all possible tangent varifolds of $\Omega$ (and hence of $\Sigma_j$) at $p$ must be connected, as it is equal to the limit points of $\{\eta_{p,r}(\Sigma_j), r>1\}$. Since $P,\overline{P}$ are clearly separated in the space of (half-)planes, it follows that the tangent cone of $\Sigma_j$ at $p$ is unique and must be $P$. 

Case 2: $\tilde \rho \le 0$. By a similar argument as above (taking $\rho_0$ to be the infimum of all real numbers $\rho$ such that $\overline{\Sigma}_j\cap \Sigma_j^\rho\ne \emptyset$), we conclude that $\rho_0=\tilde \rho$, and $\overline{\Sigma}_j$ is contained in the closure of the upper side of $\Sigma_j^{\tilde \rho}$ (i.e. $\{x_3 \geq 0\}$). In this case, $T_p\Sigma_j^{\tilde \rho}=P$ is the only such half plane.

Finally, since the tangent cone of $\Sigma_j$ at $p$ is unique, $\overline \Sigma_j$ is, in a neighborhood of $p$, the graph of a $C^1$ function $v$. By standard elliptic regularity, this implies that $v$ is in fact $C^\infty$, and hence $\overline \Sigma_j$ is a smooth embedded surface.
\end{proof}

\section{Setups and main results}
\label{S:setups}

In this section, we establish the technical setup for our min-max theory, including the generic metrics we consider, and compactness results for regular capillary varifolds. This section also contains the statement of the main min-max Theorem \ref{T:main1}. 

Recall that by \cite[Section 5]{Marques-Neves18}, $\C(M)$ is identified with $\bI_{n+1}(M, \mZ_2)$. In particular, the flat $\F$-norm and the mass $\M$-norm are the same on $\C(M)$. Given $\Om_1, \Om_2\in \C(M)$, recall that we have defined the {\em $\mF$-distance} between them as:
\[ \mF(\Om_1, \Om_2)=\F(\Om_1-\Om_2)+\mF(|\partial\Om_1\lc \iM|, |\partial\Om_2\lc \iM|). \] 
Given $\Om\in\C(M)$, we will denote $\overline{\bB}^\mF_\ep(\Om)=\{\Om'\in \C(M): \mF(\Om', \Om)\leq \ep\}$.

\subsection{Generic metrics, regular capillary varifold and compactness results}

We consider metrics with the following property:

\begin{itemize}
\item[($\star$)] If $\Sigma^n \hookrightarrow M$ is an embedded submanifold of constant mean curvature with respect to $g$, then $\Sigma \cap \pr M$ is contained in a countable union of connected, smoothly embedded $(n-1)$-dimensional submanifolds. 
\end{itemize}

Let $\mathcal{S}$ be the set of smooth metrics $g$ on $M$ such that the boundary mean curvature $H_g^{\pr M}$ is a Morse function on $\pr M$. Note that (\hyperref[def:generic]{$\star$}) is clearly satisfied for any $g\in \mathcal{S}$. For instance, one may adapt the argument of \cite[Lemma 2.8]{Zhou-Zhu17} as follows: Given $p\in\Sigma\cap \pM$, write $\Sigma$ and $\pM$ as graphs over their common tangent plane at $p$. The difference $u$ of the graph-defining functions satisfies a linear elliptic PDE $Lu = H^{\pM} - c$, and $u$ does not change sign. Wherever $H^{\pM} - c \neq 0$, the touching set $\Sigma\cap \pM$ near $p$ is contained in the critical set $\{u(x) = 0, Du(x)=0\}$ and is a union of $(n-1)$-dimensional submanifolds by \cite[Lemma 2.8]{Zhou-Zhu17}. On the other hand, the set $\{H^{\pM} - c=0\}$ is automatically contained in a union of $(n-1)$-dimensional submanifolds since $H^{\pM}$ is Morse.

\begin{lemma}\label{lemma.generic.metrics}
    Let $n\ge 2$ and let $(M^{n+1},\partial M)$ be a smooth manifold with boundary. Then $\mathcal{S}$ is open and dense in the space of $C^m$ Riemannian metrics on $M$, for any $m\ge 4$.\end{lemma}
\begin{proof}
Openness is clear. To show that it is dense, let $H_g$ be the mean curvature of $\partial M$ with respect to the outer unit normal $\nu$ and a given metric $g$. We may choose a Morse function $H'$ on $\partial M$ with $\|H-H'\|_{C^m(\partial M,g)}\le \ep/16$. Denote $h=H'-H$. Then there exists a smooth function $u$ on $M$ such that 
    \[u=0,\quad \pa{u}{\nu}=-\frac{h}{n}\quad \text{ on }\partial M,\quad \|u\|_{C^m(M,g)}\le \ep/8.\]
    Let $g'=e^{2u}g$. It follows that $H_{g'}=e^u\left(H_g - n \pa{u}{\nu}\right)=H_g+h=H'$. 
\end{proof}

Henceforth, by a generic metric on $M$ we mean one satisfying (\hyperref[def:generic]{$\star$}). 


\begin{definition}[Regular capillary varifold]\label{definition.regular.capillary.varifold}

    Let $(M,\partial M)$ be a smooth manifold with boundary, $U$ a relatively open set of $M$, $g$ a smooth metric satisfying (\hyperref[def:generic]{$\star$}). For a fixed capillary contact angle $\theta\in (0,\pi)$ and $c\in  \R$, we say $V$ is a \textit{regular capillary varifold} in $U$, if:
    \begin{enumerate}
        \item $c\neq 0$ and $V = |\partial^\theta\Omega|$; where $\Omega\in \C(U)$ is a Caccioppoli set such that $\Sigma=\spt(|\partial^\theta\Omega|\lc \mathring{U})$ is a regular, almost properly embedded, capillary CMC surface in $U$.
        \item $c=0$ and $V = |\partial^\theta \Omega| + \sum_i m_i\Sigma_{(i)}$; where:
    \begin{itemize} 
    \item $\Omega\in \C(U)$ is a Caccioppoli set such that $\Sigma=\spt(|\partial^\theta\Omega|\lc \mathring{U})$ is a regular, almost properly embedded minimal surface in $U$; 
    \item $\Sigma_{(i)}$ is a connected, smooth, embedded, minimal surface with no boundary in $U$;
    \item $m_i\geq 2$ are integers.
    \end{itemize}
    \end{enumerate}
    
    We say $V$ is stable if $\Omega$ is stable (in $U$) for the capillary functional $\A$ (and each $\Sigma_{(i)}$ is stable, when $c=0$).
\end{definition}

Note that CMC components without (topological) boundary in $U$ may be included in the capillary boundary $\pr^\theta \Omega$, but they occur with multiplicity 1 and are locally boundaries (of Caccioppoli sets).

\begin{example}
    We illustrate in this example how the generic property (\hyperref[def:generic]{$\star$}) is useful. If one does not enforce (\hyperref[def:generic]{$\star$}), then one can construct, as illustrated in Figure \ref{figure.non.generic.touching}, $(M^{n+1},\partial M,g)$, such that part of $\partial M$ has constant mean curvature, and an embedded capillary CMC surface $\Sigma$ coincide with $\partial M$ on an open set of $\partial M$. Moreover, $\Sigma$ bounds a Caccioppoli set $\Omega$. We emphasize here that $\spt(\partial \Omega\lc \mathring{M})\ne \Sigma$, and that $\partial \Omega\lc \mathring{M}$ is not a capillary CMC surface.
\end{example}

\begin{figure}[htbp]
    \centering
    \includegraphics{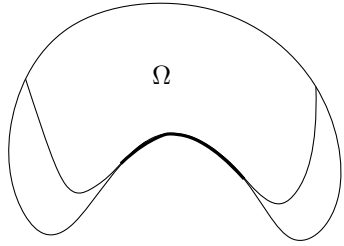}
    \caption{Example: $n$-dimensional touching in non-generic metrics}
    \label{figure.non.generic.touching}
\end{figure}

By simply taking $M\setminus \Omega$ in place of $\Omega$, one obtains the same capillary surface with $H=-c$ and contact angle $\pi-\theta$. Thus, we will focus on the case when $\theta\in (0,\tfrac{\pi}{2}]$ from now on.

Suppose $V=|\partial^\theta \Om|$ is a regular capillary varifold in $(M,\partial M,g)$, $g$ is generic, $\spt(\partial \Omega\lc \mathring{M})$ is an almost properly embedded capillary CMC $\Sigma$, and $p\in \spt V\cap \partial M$. If $p\notin \Sigma$, then in a neighborhood of $p$, $V$ is supported on $\partial M$ with density $\cos\theta$. Assume that $p\in \Sigma\cap \partial M$. We list all possible behaviors of $V$ near $p$; (see also Figure \ref{figure.different.boundary.behaviors}, where shaded regions indicate $\Omega$).

\begin{enumerate}
    \item If $p\in \partial \Sigma\cap \mR(\Sigma)$, $\Sigma$ is locally a properly embedded surface meeting $\partial M$ at $\theta$ near $p$. In particular, $\Theta^n(\|V\|,p)=\frac{1}{2}(1+\cos\theta)$.
    \item If $p\in \partial \Sigma \cap \mS(\Sigma)$, by Remark \ref{remark.regular.and.singular.parts.almost.embedded.surface}, $\Sigma$ is locally the union of two properly embedded surfaces meeting $\partial M$ both at $\theta$ near $p$. Thus $\Theta^n(\|V\|,p)=1$.
    \item If $p\in \interior(\Sigma)\cap \mR(\Sigma)$, $\Sigma$ is locally properly embedded and touches $\partial M$ at $p$. We have $\Theta^n(\|V\|,p)=1$ or $\Theta^n(\|V\|,p) = 1+\cos\theta$ depending on the orientation relative to $\pr M$.
    \item If $p\in \interior(\Sigma)\cap \mS(\Sigma)$, by Remark \ref{remark.regular.and.singular.parts.almost.embedded.surface}, $\Sigma$ is locally the union of two embedded surfaces, each touching $\partial M$ at $p$. In this case $\Theta^n(\|V\|,p)=2$ or $\Theta^n(\|V\|,p) = 2+\cos\theta$ depending on the orientation relative to $\pr M$. 
\end{enumerate}

\begin{figure}[htbp]
    \centering
    \includegraphics{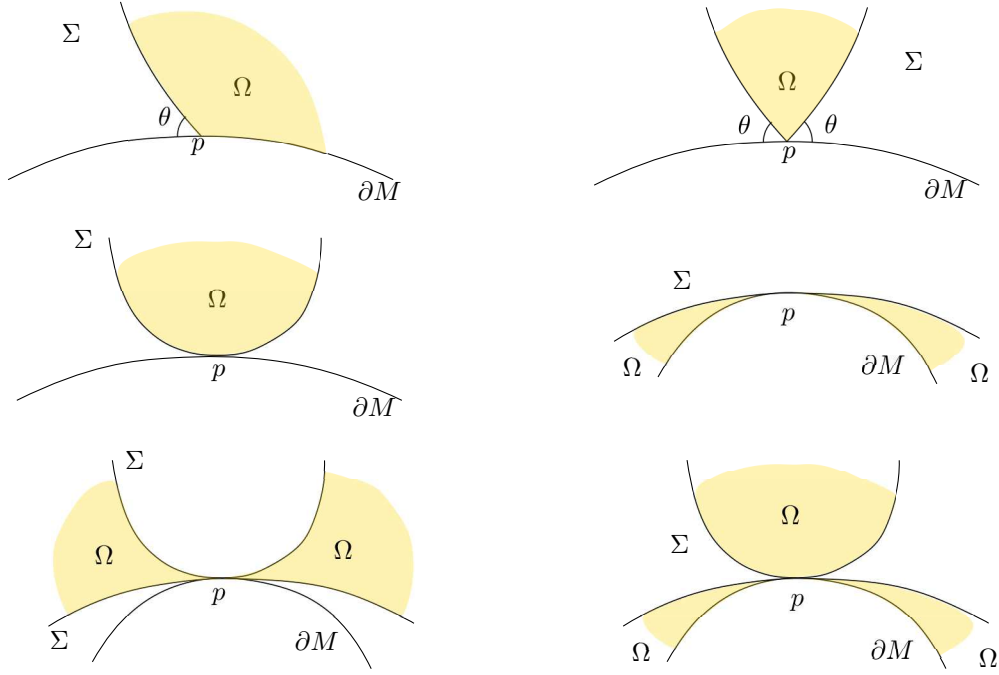}
    \caption{Different boundary behaviors of a regular capillary varifold. }
    \label{figure.different.boundary.behaviors}
\end{figure}

We have the following important compactness property for regular capillary varifolds in generic metrics.


\begin{theorem}
\label{theorem.compactness.almost.properly.embedded.capillary.surfaces}
Let $(M^3,\partial M,g)$ be a smooth Riemannian manifold with boundary, where $g$ satisfies (\hyperref[def:generic]{$\star$}), and $U$ be a relatively open subset of $M$. Suppose $V_k=|\partial^\theta\Omega_k|+\sum_i m_i \Sigma_{k,(i)}$ is a sequence of stable regular capillary varifolds in $U$ with mean curvature $c_k$, constant contact angle $\theta\in (0,\tfrac{\pi}{2}]$, and $\sup_k \|V_k\|(M)<\infty$.  Setting $\Sigma_k=\spt(\|V_k\|\lc \mathring{U})$, the following statements hold.
\begin{enumerate}
    \item Assume that either: $\inf |c_k|>0$; or $\theta\ne \pi/2$ and each connected component of $\Sigma_k$ has nonempty boundary. Then $V_k$ subsequentially converges as varifolds to a stable regular capillary varifold $V_\infty=\partial^\theta \Omega_\infty$. Moreover, $\Sigma_k\rightarrow \Sigma_\infty$ with multiplicity $1$ on $\mR(\Sigma_\infty)$.
    \item If $\inf |c_k|>0$ and $\theta\ne \pi/2$, for any $p\in \mS(\Sigma_\infty)\cap \partial \Sigma_\infty$, there exists a neighborhood $U_p$ of $p$, such that $\mS(\Sigma_\infty)\cap U_p\subset \partial \Sigma_\infty$.
    \item If $c_k\rightarrow 0$, then up to a subsequence, $\Sigma_k$ converges locally smoothly (with multiplicity) to some smooth almost properly embedded stable minimal capillary surface $\Sigma_\infty$ in $U$. Moreover, $\mS(\Sigma_\infty)\subset \partial \Sigma_\infty$.
    \item If $\theta=\pi/2$, then up to a subsequence, $\Sigma_k$ converges locally smoothly to some almost properly embedded free-boundary stable CMC surface $\Sigma_\infty$ in $U$.
\end{enumerate}
\end{theorem}

\begin{proof}
By assumption, $\sup \Area(\Sigma_k)<\infty$. Thus, by Theorem \ref{theorem.curvature.estimate}, $\sup |A_{\Sigma_k}|<\infty$. Therefore, by standard elliptic PDE estimates, a subsequence of $\Sigma_k$ (still denoted by $\Sigma_k$) locally smoothly converges (possibly with multiplicity) to a stable regular capillary surface $\Sigma_\infty$. Also, since each $\Sigma_k$ is almost properly embedded with uniform curvature bounds, we conclude that $\Sigma_\infty$ is almost properly embedded. By the generic assumption on $g$, $\Sigma_\infty$ does coincide with $\partial M$ on any open subset of $\partial M$.

We now prove assertion (1). Suppose $\Sigma_k=\partial \Omega_k\lc \mathring{U}$ for $\Omega_k\in \C(U)$. By the standard compactness theorem, up to a subsequence, $\Omega_k\rightarrow \Omega_\infty$, $\partial \Omega_k\rightarrow \partial \Omega_\infty$ as integral currents, and $\Omega_\infty\in \C(U)$. Further more, replacing $\Omega_k$ by $U\setminus \Omega_k$ if necessary, we assume that $c_k\ge 0$ and $\theta\in (0,\pi)$. We must prove $\Sigma_\infty=\spt(\partial\Omega_\infty\cap \mathring{U})$ as varifolds. 

We first assume $\inf c_k>0$. Fix a point $p\in \mR(\Sigma_\infty)$. By Lemma \ref{lemma.strong.max.principle.embedded}, there is an open neighborhood $U_p$ of $p$, such that for sufficiently large $k$, $\Sigma_k\cap U_p$ has a graphical decomposition $\sqcup_{i=1}^{m_k} \Sigma_k^i$, where $m_k\ge 1$ is the multiplicity of convergence, and that the outward unit normal of $\Sigma_k^i$ points to the same direction (without loss of generality, say they are all upward, when viewed as an oriented graph over $\Sigma_\infty$).

We claim that $m_k=1$. By the constancy theorem, $\Omega_k\lc U_p=\sum_{i=0}^{m_k}a_i U_i$, here $a_i\in \mZ$, $U_i$ are the open subsets of $U_p$ separated by $\{\Sigma_k^i\}$. Thus, $\Sigma_k\lc U_p=\partial (\Omega_k\lc U_p)=\sum_i a_i \partial U_i$. Hence, along $\{\Sigma_k^i\}$ where each $\Sigma_k^i$ viewed as an oriented graph over $\Sigma_\infty$, the number of upward unit normal and the number of downward unit normal may differ by at most $1$. This is only possible when $m_k=1$.

Alternatively, we assume that $\theta\ne \pi/2$, and that each connected component of $\Sigma_k$ has nonempty boundary. Thus $\Sigma_\infty$ satisfies the same assumption. Fix a point $p\in \partial\Sigma_\infty$. We separate two cases. First suppose $p\in \mR(\Sigma_\infty)$. Let $\nu_p$ be the outward unit normal of $\Sigma_\infty$ at $p$. By taking a small neighborhood $U_p$ of $p$ and sufficiently large $k$, we can decompose $\Sigma_k\cap U_p$ graphically as $\sqcup_{i=1}^{m_k}\Sigma_k^i$. Let $\nu_k^i$ be the outward unit normal vector field of $\Sigma_k$ at $p$. By smooth convergence, we have either $\nu_k^i\rightarrow \nu_p$ or $\nu_k^i\rightarrow -\nu_p$. However, since $\theta\ne \pi/2$ and each $\Sigma_k$ satisfies $\bangle{\nu_k^i,\overline{\eta}}=\cos\theta\ne 0$ along the boundary, we must have that $\nu_k^i\rightarrow \nu_p$, and in particular, all of the outward unit normal of $\Sigma_k^i$ points to the same direction as $\nu_p$. Thus, by the same proof as before, we must have $m_k=1$ near $p$, and hence $\Sigma_k\rightarrow \Sigma_\infty$ in multiplicity $1$ as currents.

Now assume that $p\in \mS(\Sigma_\infty)$. In a small neighborhood $U_p$ of $p$, let $\Sigma_\infty^1,\cdots,\Sigma_\infty^l$ be the properly embedded sheets of $\Sigma$. By Lemma \ref{lemma.strong.max.principle.embedded}, $T_p \Sigma_\infty^j$, $j=1,\cdots,l$, are capillary half planes in $\R^3_+$ without interior intersection. Thus $l=2$, and $\Omega\cap U_p$ is the region bounded by $\Sigma_\infty^1$ and $\Sigma_\infty^2$. We then perform the same analysis as before over $\Sigma_\infty^1$, $\Sigma_\infty^2$, and conclude that the convergence $\Sigma_k\rightarrow \Sigma_\infty$ is with multiplicity $1$.

For assertion (2), suppose, for the sake of contradiction, that there exists $p\in \mS(\Sigma_\infty)\cap \partial \Sigma_\infty$, and a sequence $p_j\in \mS(\Sigma_\infty)\cap \mathring{\Sigma}_\infty$ converges to $p$. Since each $p_j$ is in the interior of $\Sigma_\infty$, by the maximum principle, at each $p_j$, there exist two sheets of $\Sigma_\infty$ touching at $p_j$, and the outward unit normal of $\Sigma_\infty$ at $p_j$ are opposite. Since $p_j\rightarrow p$ and $\Sigma_\infty$ is smooth, the same holds for the two sheets meeting at $p$. On the other hand, the angle between the outward unit normal vectors of $\Sigma$ and $M$ is constant $\theta$. This implies $2\theta=\pi$, contradiction.

Assertion (3) and (4) follows directly from Theorem \ref{theorem.curvature.estimate} and the interior strong maximum principle for minimal surfaces.
\end{proof}

\subsection{Main theorems}
\label{sec:setup-main}

\vspace{1em}
Let $I=[0, 1]$ be the unit interval. Let $\Phi_0: I\to (\C(M), \mF)$ be a continuous map under the $\mF$-topology, and so that $\Phi_0(0)=\emptyset$ and $\Phi_0(1) = M$. We usually call such $\Phi_0$ a {\em sweepout}. We use $\Pi$ to denote the set of all continuous maps $\Psi: I\to (\C(M), \mF)$ such that  $\Psi(0)=\emptyset, \Psi(1)=M$ and $\Psi$ and $\Phi_0$ are homotopic to each other in $C(M)$ with the flat toplogy.

\begin{definition}[Width and min-max sequences]
The {\em $(c, \theta)$-width} (or simply called {\em width}) of $\Pi$ is defined by
\[\bL(\Pi)= \inf_{\Phi\in\Pi}\sup_{x\in I} \A(\Phi(x))\]
A sequence $\{\Phi_i\}_{i\in\N} \subset \Pi$ is called a min-max sequence if $\bL(\Phi_i) := \sup_{x\in I} \A(\Phi_i(x))$ satisfies
\[ \bL(\{\Phi_i\}_{i\in\N}) :=\limsup_{i\to\infty}\bL(\Phi_i) =\bL(\Pi). \]
\end{definition}

\begin{definition}
Given a min-max sequence $\{\Phi_i\}_{i\in\N}$ in $\Pi$, the {\em critical set} of $\{\Phi_i\}$ is defined as
\[ \bC(\{\Phi_i\})=\big\{ V=\lim_{j\to\infty}|\partial^\theta \Phi_{i_j}(x_j)| \text{ as varifolds: with } \lim_{j\to\infty}\A(\Phi_{i_j}(x_j))=\bL(\Pi) \big\}.\]
\end{definition}

Now we state the min-max theorem for capillary surfaces in a three dimensional compact manifold.

\begin{theorem}[Min-max theorem]
\label{T:main1}
Let $(M^3, g)$ be a compact manifold with smooth boundary $\partial M$ and metric $g$ satisfying (\hyperref[def:generic]{$\star$}), $c\in\R$, and $\theta\in (0, \frac{\pi}{2}]$. Given a map $\Phi_0: I \to (\C(M), \mF)$ continuous in the $\mF$-topology with $\Phi_0(0)=\emptyset$ and $\Phi_0(1)=M$, let $\Pi$ be the associated homotopy class. Suppose
\begin{equation}
\label{E:nontriviality}
    \bL(\Pi)>\max\{\A(\emptyset), \A(M)\}=\max\{\cos\theta\cdot\mH(\partial M)-c\vol(M), 0\}.
\end{equation}
Let $\{\Phi_i\}_{i\in\N}\subset \Pi$ be a min-max sequence for $\Pi$. Then there exists $V\in \bC(\{\Phi_i\})$ such that $V\lc \iM$ is induced by a nontrivial, smooth, almost properly embedded surface $\Si\subset M$ of prescribed mean curvature $c$ and smooth boundary $\partial \Sigma$ contacting $\partial M$ at angle $\theta$. Moreover, if $c\neq 0$ and $\theta\neq \frac{\pi}{2}$, then $V$ has multiplicity one. 
\end{theorem}
\begin{remark}
When $c=0$ and $\theta=\frac{\pi}{2}$, this is the free boundary min-max theory established by M. Li and the second author in \cite{Li-Zhou16}. When $M$ is closed, i.e. $\partial M=\emptyset$, this is the constant mean curvature (CMC) min-max theory established by the last two authors in \cite{Zhou-Zhu17}. The results were later generalized to the cases when $c\neq 0, \theta= \frac{\pi}{2}$ in \cite{Sun-Wang-Zhou20}. 
\end{remark}

\begin{proof}[Proof of Theorem \ref{T:main1}.]
As in the above remark, the remaining case is when $\theta<\frac{\pi}{2}$. The theorem follows from combining Theorem \ref{T:existence of almost minimizing capillary varifolds} and Theorem \ref{T:boundary regularity}.
\end{proof}

\subsection{Existence of nontrivial sweepout}
We begin with the following lemma:

\begin{lemma}\label{lemma.A>0.for.small.region}
    Let $(M^{n+1},g)$ be a smooth Riemannian manifold with boundary, $c\in \R,\theta\in (0,\pi)$. Then there exists $\delta>0$, and constants $C_0,V_0>0$ depending only on $M,c,\theta$, such that 
    \[\A(\Omega)>\delta\mH^{n+1}(\Omega)^{\frac{n}{n+1}}, \text{ whenever }\Omega\in \C(M) \text{ and }|\Omega|< V_0.\]
\end{lemma}
\begin{proof}
    For $V>0$, consider the following two (relative) isoperimetric problems in $M$ for an open subset $\Omega\subset M$:
    \begin{equation*}
        \begin{split}
                 I_1(V)&=\inf\left\{\frac{\mH^n(\partial\Omega\lc \mathring{M})}{\mH^n(\partial\Omega\lc \partial M)}: \mH^{n+1}(\Omega)=V, \Omega\subset M \right\},  \\
             I_2(V)&=\inf \left\{ \mH^n(\partial \Omega\lc \mathring{M}), \mH^{n+1}(\Omega)=V, \Omega\subset M\right\}.
        \end{split}
    \end{equation*}
    $I_2(V)$ is the isoperimetric profile of $M$. It is known from \cite[Proposition 2.1]{BayleRosales2005isoperimetric} that for all sufficiently small $V$, $I_2(V)\ge \mu V^{\frac{n}{n+1}}$. Here $\mu$ is an explicit constant depending only on $n$. 
    
    Since $M$ is compact and $\partial M$ is $C^2$, there exists a vector field $X\in \X(M)$ with $X=-\nu_{\partial M}$ on $\partial M$ and $|X|\le 1$. For instance, let $r_0>0$ be the injectivity radius of $M$. Take a cut-off function $\phi:[0,\infty)\rightarrow [0,1]$, such that  $\phi=1$ on $[0,r_0/2]$, $\phi=0$ on $[r_0,\infty)$, and let $X=-\phi\nabla \dist(\cdot, \partial M)$. Then for any $\Omega\in \C(M)$, $|\Omega|=V$, we have
    \begin{multline*}
        \mH^n(\partial \Omega\lc \partial M)\le \mH^n(\partial \Omega\lc \mathring{M})+\left|\int_\Omega \Div X\right|\\
            \le \mH^n(\partial \Omega\lc \mathring{M})+\lip(X) V\le \left(1+\frac{\lip(X)V^{\frac{1}{n+1}}}{\mu}\right)\mH^n(\partial \Omega\lc \mathring{M}).
    \end{multline*}
    
    Thus, for any $\ep>0$, there exists $V_0=V_0(M,\ep)$, such that $I_1(V)>1-\ep$ for all $V\in (0,V_0)$. Now fix $\ep=\frac{1}{2}(1-|\cos\theta|)$. Then for all $\Omega\in \C(M)$, $|\Omega|=V\in (0,V_0)$, 
    \[\A(\Omega)>\ep\mH^n(\partial \Omega\cap\mathring{M})-|c|\mH^{n+1}(\Omega)>\ep \mu V^{\frac{n}{n+1}}-|c|V>\delta V^{\frac{n}{n+1}},\]
    by letting $\delta=\ep\mu/2$ and possibly further shrinking $V_0=V_0(M,\theta,c)$.
\end{proof}

\begin{theorem}
There exists a homotopy class $\Pi$ satisfying (\ref{E:nontriviality}) for any $c\in\R$ and $\theta\in(0, \frac{\pi}{2}]$. 
\end{theorem}
\begin{proof}
Take a Morse function $\phi: M\rightarrow [0,1]$, and define a sweepout $\Phi_0:[0,1]\rightarrow \C(M)$ by letting $\Phi_0(t)=\{x\in M,\phi(x)<t\}$. Let $\Pi$ be the set of sweepouts homotopic to $\Phi_0$, and $\Phi\in \Pi$. We prove $\sup_{x\in [0,1]}\A(\Phi(x))>0$ uniformly in $\Phi$.

We separate the proof into two cases. Suppose first that $\cos\theta\cdot\mH^n(\partial M)-c\vol(M)\le 0$. Let $V_0$ be given as in Lemma \ref{lemma.A>0.for.small.region}. Since $\Phi$ is continuous in the $\mF$-topology, there exists $x_0$ such that $\mH^n(\Phi(x_0))\in (V_0/2,V_0)$. Therefore $\A(\Phi(x_0))>\delta \mH^{n+1}\Phi(x_0))^{\frac{n}{n+1}}>\delta(V_0/2)^{\frac{n}{n+1}}$. Since this holds for any $\Phi\in \Pi$, we conclude that $\bL(\Pi)>\delta(V_0/2)^{\frac{n}{n+1}}$.

Now suppose that $\cos\theta\cdot\mH^n(\partial M)-c\vol(M)> 0$. As before there exists $x_0$ such that $\mH^{n+1}(\Phi(x_0))\in (\vol(M)-V_0,\vol(M)-V_0/2)$. Denote $\Omega=M\setminus \Phi(x_0)$. Applying Lemma \ref{lemma.A>0.for.small.region} to $\Omega$, with $\pi-\theta$ in place of $\theta$ and $-c$ in place of $c$, we obtain
\[\mH^n(\partial \Omega\lc \mathring{M})-\cos\theta \mH^n(\partial \Omega\lc \partial M)+ c\mH^{n+1}(\Omega)>\delta(V_0/2)^{\frac{n}{n+1}}.\]
Adding $\cos\theta\cdot \mH^n(\partial M)-c\vol(M)$ to both sides, we obtain that $\A(\Phi(x_0))>\cos\theta\cdot\mH^n(\partial M)+ \vol(M)+\delta (V_0/2)^{\frac{n}{n+1}}$. Thus $\bL(\Pi)>\cos\theta\cdot\mH^n(\partial M)+ \vol(M)+\delta (V_0/2)^{\frac{n}{n+1}}$.
\end{proof}


Now we sketch the {\bf pull-tight} process applied to $\{\Phi_i\}$. The goal is to make sure that every element in the critical set $\bC(\{\Phi_i\})$ has uniformly bounded first variation with respect to boundary preserving diffeomorphisms. 

\begin{definition}
$V\in \V_n(M)$ is said to have {\em $c$-bounded first variation with respect to $\X_{\tan}(M)$} (or $\X_{\tan}(U)$ for a given relatively open subset $U\subset M$) if:
\[ \int_{G_n(M)} \Div_S X(x) dV(x, S) \leq c\cdot \int_M |X(x)| d\mu_V(x), \, \forall X\in\X_{tan}(M)\, (\text{or} \in \X_{tan}(U)).\]
\end{definition}

\vspace{1em}
Let $L^c=2(\bL(\Pi)+c\vol(M))$. Denote 
\[ \begin{split}
A^c_\infty =
     & \{V\in\V_n(M): \|V\|(M)\leq L^c, V \text{ has $c$-bounded first variation} \\
     & \text{with respect to } \X_{tan}(M)\}.
\end{split}\]

We can follow \cite[Section 4]{Zhou-Zhu17} by changing $\X(M)$ to $\X_{tan}(M)$ to construct a continuous map: 
\[ H: [0, 1]\times (\C(M), \mF)\cap\{\M(\partial \Om)\leq L^c\}\to (\C(M), \mF)\cap \{\M(\partial \Om)\leq L^c\} \]
such that:
\begin{itemize}
\item[(\rom{1})] $H(0, \Om)=\Om$ for all $\Om$;
\item[(\rom{2})] $H(t, \Om)=\Om$ if $|\partial^\theta\Om|\in A^c_\infty$;
\item[(\rom{3})] if $|\partial^\theta\Om|\notin A^c_\infty$,
\[ \A(H(1, \Om))-\A(\Om)\leq -L(\mF(|\partial\Om|, A^c_\infty))<0; \]
here $L: [0, \infty) \to [0, \infty)$ is a continuous function with $L(0)=0$, $L(t)>0$ when $t>0$.
\end{itemize}

\begin{lemma}
\label{L:tightening lemma}
Given a min-max sequence $\{\Phi_i^*\}_{i\in\N}\in\Pi$, we define $\Phi_i(x)=H(1, \Phi^*_i(x))$ for every $x\in I$. Then $\{\Phi_i\}_{i\in\N}$ is also a min-max sequence in $\Pi$. Moreover, $\bC(\{\Phi_i\})\subset \bC(\{\Phi^*_i\})$ and every element of $\bC(\{\Phi_i\})$ has $c$-bounded first variation with respect to $\X_{\tan}(M)$.
\end{lemma}
\begin{proof}
By continuity of $H$, we know that $\Phi_i$ is homotopic to $\Phi^*_i$ in the flat topology, so $\{\Phi_i\}\in\Pi$. By (\rom{2})(\rom{3}), $\A(\Phi_i(x))\leq \A(\Phi^*_i(x))$ for every $x\in I$, so $\{\Phi_i\}$ is also a min-max sequence. Finally, given any $V\in \bC(\{\Phi_i\})$, then $V=\lim_{j\to\infty}|\partial^\theta \Phi_{i_j}(x_j)|$ where $\lim_{j\to\infty}\A(\Phi_{i_j}(x_j))=\bL$. Denote $V^*=\lim_{j\to\infty}|\partial^\theta \Phi^*_{i_j}(x_j)|$. By (\rom{3}), $\lim_{j\to\infty}\mF(|\partial \Phi^*_{i_j}(x_j)|, A^c_\infty)=0$ (as $\lim_{j\to\infty}\A(\Phi_{i_j}(x_j))=\lim_{j\to\infty}\A(\Phi^*_{i_j}(x_j))=\bL$), so $V^*\in A^c_\infty$. On the other hand,
\[ V= \lim_{j\to\infty}|\partial^\theta H(1, \Phi^*_{i_j}(x_j))| = H(1, \lim_{j\to\infty} |\partial^\theta\Phi^*_{i_j}(x_j)|)=H(1, V^*)=V^*.\]
We used Proposition \ref{L:communicates lemma} in the second equality. Note that $H$ is also well defined as a continuous map $H: [0, 1]\times \{V\in\V_n(M), \|V\|(M)\leq L^c\} \to \{V\in\V_n(M), \|V\|(M)\leq L^c\}$. Hence $\bC(\{\Phi_i\})\subset \bC(\{\Phi^*_i\})$ and the proof is finished.
\end{proof}

\section{Almost minimizing capillary varifolds}
\label{S:AM}

In this section we define almost minimizing varifolds with respect to the capillary functional $\A$, and establish the existence and properties of replacements. We fix $c\in \mathbb{R}$ and $\theta\in (0, \frac{\pi}{2})$. 

\begin{definition}[capillary almost minimizing varifolds]
\label{D:c-am-varifolds}
Let $\doubleunderline{\nu}$ be the $\F$, $\M$-norms or the $\mF$-metric. For any given $\ep, \de>0$ and a relative open subset $U\subset M$, we define $\sA(U; \ep, \de; \doubleunderline{\nu})$ to be the set of all $\Om\in\C(M)$ such that if $\Om=\Om_0, \Om_1, \Om_2, \cdots, \Om_m\in\C(M)$ is a sequence with:
\begin{itemize}
\item[(i)] $\spt(\Om_i-\Om)\subset U$;
\item[(ii)] $\doubleunderline{\nu}(\Om_{i+1}, \Om_i)\leq \de$;
\item[(iii)] $\A(\Om_i)\leq \A(\Om)+\de$, for $i=1, \cdots, m$,
\end{itemize}
then $\A(\Om_m)\geq \A(\Om)-\ep$.

We say that a varifold $V\in\V_n(M)$ is {\em capillary almost minimizing in $U$} if there exist sequences $\ep_i \to 0$, $\de_i \to 0$, and $\Om_i\in \sA(U; \ep_i, \de_i; \F)$, such that $\mF(|\partial^\theta\Om_i|, V)\leq \ep_i$.
\end{definition}

The following simple fact says that capillary almost minimizing implies $c$-bounded first variation with respect to $\X_{\tan}(U)$.
\begin{lemma}
\label{L:c-am implies c-bd-first-variation}
Let $V\in\V_n(M)$ be capillary almost minimizing in $U$, then $V$ has $|c|$-bounded first variation with respect to $\X_{\tan}(U)$.
\end{lemma}
\begin{proof}
Suppose not, then there exist $\ep_0>0$ and a smooth vector field $X\in\X_{tan}(U)$ compactly supported in $U$, such that
\[ \left|\int_{G_n(M)} \div_S X(x) dV(x, S)\right|\geq (|c|+\ep_0)\int_M|X|d\mu_V>0. \]
By changing the sign of $X$ if necessary, we have
\[ \int_{G_n(M)} \div_S X(x) dV(x, S)\leq -(|c|+\ep_0)\int_M |X|d\mu_V. \]
By continuity, we can find $\ep_1>0$ small enough depending only on $\ep_0, V, X$, such that if $\Om\in \C(M)$ with $\mF(|\partial^\theta\Om|, V)<2\ep_1$, then 
\[ \de\A|_\Om(X) \leq \int_{\partial\Om} \div_{\partial\Om} X d\mu_\theta + |c|\int_{\partial\Om} |X|d\mu_{\partial\Om}\leq -\frac{\ep_0}{2}\int_{M} |X| d\mu_{V} <0. \]

If $\mF(|\partial^\theta\Om|, V)<\ep_1$, then by deforming $\Om$ along the $1$-parameter flow $\{\Phi^X(t): t\in[0, \tau)\}$ of $X$ for a uniform short time $\tau>0$, we can obtain a $1$-parameter family $\{\Om_t=\Phi^X(\Om)\in\C(M): t\in[0, \tau)\}$, such that $t\to\Om_t$ is continuous under the $\mF$-topology, with $\spt(\Om_t-\Om)\subset U$, $\mF(|\partial^\theta\Om_t|, V)<2\ep_1$ and $\A(\Om_t)\leq \A(\Om_0)=\A(\Om)$ for all $t\in[0, \tau)$, but with $\A(\Om_{\tau})\leq \A(\Om)-\ep_2$ for some $\ep_2>0$ depending only on $\ep_0, \ep_1, V, X$. 

Summarizing the above, given any $\ep<\min\{\ep_1, \ep_2\}$ and $\de>0$, if $\Om\in\C(M)$ and $\mF(|\partial^\theta\Om|, V)<\ep$, then $\Om\notin \sA(U; \ep, \de; \F)$; this contradicts the capillary almost minimizing property of $V$.
\end{proof}

Now we formulate and solve a natural constrained minimization problem which will be used in the construction of replacements.
\begin{lemma}[A constrained minimization problem]
\label{L:minimisation}
Given $\ep, \de>0$, a relatively open subset $U\subset M$ and any $\Om \in \sA(U;\ep,\de;\F)$, fix a compact subset $K\subset U$. Let $\C_\Om$ be the set of all $\La\in\C(M)$ such that there exists a sequence $\Om=\Om_0, \Om_1, \cdots, \Om_m=\La$ in $\C(M)$ satisfying:
 \begin{itemize}
\item[(a)] $\spt(\Om_i-\Om)\subset K$;
\item[(b)] $\F(\Om_i - \Om_{i+1})\leq \de$;
\item[(c)] $\A(\Om_i)\leq \A(\Om)+\de$, for $i=1, \cdots, m$.
\end{itemize}

Then there exists $\Omega^* \in \C(M)$ such that:
\begin{itemize}
\item[(i)] $\Om^* \in \C_\Om$, and \[ \A(\Om^*)=\inf\{\A(\La):\ \La\in\C_\Om\};\]
\item[(ii)] $\Om^*$ is stable and locally $\A$-minimizing in the relative interior $\interior_M(K)$ (relative to $\pM$);
\item[(iii)] $\Om^*\in \sA(U;\ep,\de;\F)$.
\end{itemize}
\end{lemma}
\begin{proof}

The proof follows a similar structure to \cite[Lemma 5.7]{Zhou-Zhu17}:

\textit{Proof of (i):} Take any minimizing sequence $\{\La_j\}\subset\C_{\Om}$, i.e.
\[ \lim_{j \to \infty} \A(\La_j) = \inf\{\A(\La):\ \La\in\C_{\Om}\}.\]
Notice that $\spt(\La_j-\Om) \subset K$ and $\A(\La_j) \leq \A(\Om) + \de$ for all $j$. By standard compactness \cite[Theorem 6.3]{Simon83}, after passing to a subsequence, $\La_j$ converges weakly to some $\Om^*$ with $\Om^* \in\C(M)$ and $\spt(\Om^*-\Om) \subset K$. 

We will show that $\Om^*$ is our desired minimizer. Since $\La_j$ converges weakly to $\Om^*$, we have that $\M(\partial^\theta\Om^*)\leq \lim_{j\to\infty}\M(\partial^\theta\La_j)$ (by Lemma \ref{L:weak convergence of reduced currents}), and $\mH^{n+1}(\Om^*)=\lim_{j\to\infty}\mH^{n+1}(\La_j)$ (by definition of weak convergence). Therefore,
\begin{equation}
\label{E:Om^*}
\A(\Om^*) \leq \lim_{j\to\infty} \A(\La_j) = \inf\{\A(\La):\ \La\in\C_{\Om}\}.
\end{equation}
It remains to show that $\Om^*\in \C_{\Om}$. For $j$ sufficiently large, we have $\F(\La_j-\Om^*)<\de$. Since $\La_j \in \C_{\Om}$, there exists a sequence $\Om=\Om_0, \Om_1, \cdots, \Om_m=\La_j$ in $\C(M)$ satisfying conditions (a-c) above. Consider now the sequence $\Om=\Om_0, \Om_1, \cdots, \Om_m=\La_j, \Om_{m+1}=\Om^*$ in $\C(M)$; it trivially satisfies conditions (a) and (b). Moreover, using (\ref{E:Om^*}), we also have
\[ \A(\Om^*)\leq \A(\La_j)\leq \A(\Om)+\de.\]
Therefore, $\Om^*\in \C_{\Om}$ and hence (i) has been proved. 

\vspace{0.5em}
\textit{Proof of (ii):} For $p$ in the relative interior $\interior_M(K)$ of $K$, we claim that there exists a small relative open ball $\sB_r(p) \subset \interior_M(K)$ such that
\begin{equation}
\label{E:Om^*2}
\A(\Om^*)\leq \A(\La),
\end{equation}
for any $\La\in\C(M)$ with $\spt(\La-\Om^*)\subset \sB_r(p)$. To establish (\ref{E:Om^*2}), first choose $r>0$ small so that $\vol(\tilde B_r(p))<\delta/2$. 
Suppose (\ref{E:Om^*2}) were false, then there exists $\Om'\in\C(M)$ with $\spt(\Om'-\Om^*)\subset \sB_r(p)$ such that $\A(\Om')<\A(\Om^*)$. We will show that $\Om'\in\C_{\Om}$, which contradicts that $\Om^*$ is a minimizer from part (i). 

To see that $\Om' \in \C_\Om$, take a sequence $\Om=\Om_0, \Om_1, \cdots, \Om_m=\Om^*$ in $\C(M)$ satisfying (a-c) above, and append $\Om_{m+1}=\Om'$ to the sequence. Since $\spt(\Om^*-\Om) \subset K$ and $\spt(\Om'-\Om^*) \subset K$, we have $\spt(\Om'-\Om)\subset K$. 
By the facts that $\spt(\Om'-\Om^*)\subset \sB_r(p)$, we have $\F(\partial\Om'-\partial\Om^*) \leq \vol(\tilde B_r(p))<\de$. Finally note $\A(\Om') < \A(\Om^*) \leq \A(\Om)+ \de$. Therefore $\Om'\in\C_{\Om}$.

Similarly we can prove that $\Omega^*$ is stable with respect to the $\A$ functional in $\interior_M(K)$. If this were not true, then there exists a flow $\{\Phi_t\}_{|t|\ll 1}$ supported in $\interior_M(K)$ ($\Phi_t$ preserves $K\cap\pM$ if $K\cap\pM \neq\emptyset$), such that $\A(\Phi_t(\Omega^*))<\A(\Omega^*)$ for $0<|t|\ll 1$. Append $\Omega_{m+1} = \Phi_t(\Omega^*)$ to the sequence $\Om=\Om_0, \Om_1, \cdots, \Om_m=\Om^*$ above for some small enough $0<|t|\ll 1$ with $\F(\partial \Phi_t(\Omega^*), \partial\Omega^*)<\delta$; hence we have $\Phi_t(\Omega^*) \in\C_{\Om}$, which again contradicts $\Omega^*$ being a minimizer. This proves part (ii). 

\vspace{0.5em}
\textit{Proof of (iii):} Suppose that the claim is false. Then by Definition \ref{D:c-am-varifolds} there exists a sequence $\Om^*=\Om^*_0, \Om^*_1, \cdots, \Om^*_\ell $ in $\C(M)$ satisfying
\begin{itemize}
\item $\spt(\Om^*_i-\Om^*)\subset U$;
\item $\F(\Om^*_i - \Om^*_{i+1})\leq \de$;
\item $\A(\Om^*_i)\leq \A(\Om^*)+\de$, for $i=1, \cdots, \ell$,
\end{itemize}
but $\A(\Om^*_\ell) < \A(\Om^*)-\ep$. Since $\Om^* \in \C_\Om$ by part (i), there exists a sequence $\Om=\Om_0, \Om_1, \cdots, \Om_m=\Om^*$ satisfying conditions (a-c) above (with $K$ changed by $U$ in (a)). Then the sequence $\Om=\Om_0,\Om_1,\cdots, \Om_m, \Om^*_1, \cdots, \Om^*_\ell$ in $\C(M)$ still satisfies those conditions (a-c), since $\A(\Om^*)\leq \A(\Om)$ implies that $\A(\Om^*_i)\leq \A(\Om)+\de$. Therefore $\Om\in \sA(U;\ep,\de;\F)$ implies that $\A(\Om^*_\ell)\geq \A(\Om)-\ep\geq \A(\Om^*)-\ep$, which is a contradiction. This proves part (iii). 
\end{proof}

\begin{proposition}[Existence and properties of capillary replacements]
\label{P:good-replacement-property}
Let $V\in\V_n(M)$ be a capillary almost minimizing in a relatively open set $U \subset M$ and $K \subset U$ be a compact subset, then there exists $V^{*}\in \V_n(M)$, called \emph{a capillary replacement of $V$ in $K$} such that
\begin{enumerate}[label=$(\roman*)$]
\item $V\lc (M\backslash K) =V^{*}\lc (M\backslash K)$;
\item $-|c| \vol(K)\leq \|V\|(M)-\|V^{*}\|(M) \leq |c| \vol(K)$;
\item $V^{*}$ is capillary almost minimizing in $U$;
    \item moreover, $V^{*} =\lim_{i \to \infty} |\partial^\theta\Om^*_i|$ as varifolds for some $\Om^*_i\in\C(M)$ such that $\Om^*_i\in \sA(U; \ep_i, \de_i; \F)$ with $\ep_i, \de_i \to 0$; furthermore $\Om^*_i$ is stable and locally minimizes $\A$ in $\interior_M(K)$;
\item if $V$ has $|c|$-bounded first variation with respect to $\X_{tan}(M)$, then so does $V^*$.
\end{enumerate}
\end{proposition}

\begin{proof}
The proof follows similarly to \cite[Proposition 5.8]{Zhou-Zhu17}:

Let $V\in \V_n(M)$ be capillary almost minimizing in $U$. By definition there exists a sequence $\Om_i\in \sA(U; \ep_i, \de_i; \F)$ with $\ep_i, \de_i \to 0$ such that $V$ is the varifold limit of $|\partial^\theta\Om_i|$. By Lemma \ref{L:minimisation} we can construct a minimizer $\Om_i^* \in \C_{\Om_i}$ for each $i$. Since $\M(\partial^\theta\Om_i^*)$ is uniformly bounded, by compactness there exists a subsequence $|\partial^\theta\Om_i^*|$ converging as varifolds to some $V^* \in \V_n(M)$. We claim that $V^*$ satisfies items (i)-(v) in Proposition \ref{P:good-replacement-property} and thus is our desired capillary replacement. 
\begin{itemize}[leftmargin=0.8cm]
\item First, by part (i) of Lemma \ref{L:minimisation}, we have $\Om_i^* \in \C_{\Om_i}$ and thus $\spt(\Om_i^*- \Om_i) \subset K$. Hence the varifold limits satisfy $V^* \lc (M \backslash K) = V \lc (M \backslash K)$. 
\item Second, as $\Om_i \in \sA(U,\ep_i,\de_i; \F)$ and $\Om_i^* \in \C_{\Om_i}$, we have 
\[  \A(\Om_i)-\ep_i\leq \A(\Om_i^*)\leq \A(\Om_i);  \]
thus by (\ref{E:A'}), 
\[ \M(\partial^\theta\Om_i)-c\mH^{n+1}(\Om_i)-\ep_i\leq \M(\partial^\theta\Om_i^*)-c\mH^{n+1}(\Om_i^*)\leq \M(\partial^\theta\Om_i)-c\mH^{n+1}(\Om_i). \] 
Note that $|\mH^{n+1}(\Om_i)-\mH^{n+1}(\Om_i^*)|\leq \vol(K)$; taking $i \to \infty$, we have \[-|c| \vol(K)\leq \|V\|(M)-\|V^{*}\|(M) \leq |c| \vol(K).\] 
\item Since each $\Om_i^* \in \sA(U; \ep_i, \de_i; \F)$ by Lemma \ref{L:minimisation}(iii), by definition $V^*$ is capillary almost minimizing in $U$. 
\item (iv) follows from Lemma \ref{L:minimisation}(ii). 
\item Finally by (iii) and Lemma \ref{L:c-am implies c-bd-first-variation}, $V^*$ has $|c|$-bounded first variation with respect to $\X_{tan}(U)$. By (i) and a standard cutoff trick it is easy to show that $V^*$ has $|c|$-bounded first variation with respect to $\X_{\tan}(M)$ whenever $V$ does. 
\end{itemize}
\end{proof}

Finally, we establish the regularity of our replacements. 

\begin{proposition}
\label{L:regularity of capillary replacement}
Let $V^*$ be as given by Proposition \ref{P:good-replacement-property}. Then $V^*$ is a stable regular capillary varifold in $\interior_M(K)$, in the sense of Definition \ref{definition.regular.capillary.varifold}. 
\end{proposition}

\begin{proof}
By Lemma \ref{P:good-replacement-property}(iv) and Theorem \ref{theorem.regularity.of.minimizers}, $V^*$ is a limit of stable, properly embedded capillary surfaces in $\interior_M(K)$. Theorem \ref{theorem.compactness.almost.properly.embedded.capillary.surfaces} implies that this is a smooth limit, which implies the result. 
\end{proof}

\section{Existence of Almost minimizing capillary varifolds}
\label{S:existence}

In this section, we complete the pull-tight process by showing that there is a min-max varifold which has $c$-bounded first variation, and is capillary almost minimizing. 

Let $p\in M$ and $r>0$. Assume that $r<\dist_M(p, \partial M)$ if $p\notin\partial M$. Recall that we defined the open annular neighbourhoods as 
\[ \wt{\An}_{s,r}(p) = \begin{cases} \wt{B}_r(p)\setminus \Clos(\wt{B}_s(p)), & p\in\mathring{M} \\ \wt{B}^+_r(p)\setminus \Clos(\wt{B}^+_s(p)), & p \in \partial M.\end{cases}\]

\begin{definition}
\label{D:almost minimizing in annuli}
A varifold $V\in\V_n(M)$ is called {\em capillary almost minimizing in small annuli}, if for each $p\in M$, there exists $r_{am}(p)>0$ such that $V$ is capillary almost minimizing in $\An_{s, r}(p)$ for all $0<s<r\leq r_{am}(p)$. If $p\notin \partial M$, we further require that $r_{am}(p)<\dist_{\R^L}(p, \partial M)$. (We may equivalently use the intrinsic annuli $\wt{\An}_{s, r}(p)$ as they are comparable at small scales.)
\end{definition} 

We need the following equivalence result among several capillary almost minimizing concepts using the three topologies induced by $\F$, $\mF$, and $\M$. It implies that we can work with the strongest $\M$-norm at the expense of shrinking the annuli. The proof follows from a standard interpolation process (see \cite[Proof of Proposition 5.3(c) on page 483]{Zhou-Zhu17}) using Lemma \ref{L:interpolation lemma} in place of \cite[Lemma A.1]{Zhou-Zhu17}.

\begin{proposition}
\label{P:equivalence of a.m.}
Given $V\in \V_n(M)$, the following statements satisfy $(a)\Longrightarrow (b)\Longrightarrow (c)\Longrightarrow (d)$:
\begin{itemize}
    \item[(a)] $V$ is capillary almost minimizing in some open set $U\subset M$;
    \item[(b)] For any $\ep>0$, there exists $\de>0$ and $\Om\in \sA(U; \ep, \de; \mF)$ such that $\mF(|\partial^\theta\Om|, V)\leq \ep$;
    \item[(c)] For any $\ep>0$, there exists $\de>0$ and $\Om\in \sA(U; \ep, \de; \M)$ such that $\mF(|\partial^\theta\Om|, V)\leq \ep$;
    \item[(d)] $V$ is capillary almost minimizing in $W$ for any relatively open set $W\subset\subset U$ with compact closure in $U$.
\end{itemize}
\end{proposition}

\begin{theorem}
\label{T:existence of almost minimizing capillary varifolds}
Under the assumption of Theorem \ref{T:main1}, there exists $V\in \bC(\{\Phi_i\})$ such that 
\begin{itemize}
    \item $V$ has $c$-bounded first variation with respect to $\X_{tan}(M)$;
    \item $V$ is capillary almost minimizing in small annuli.
\end{itemize} 
\end{theorem}

\begin{proof}
We first note that the discretization and interpolation results needed here have been addressed in Theorem 3.3, Theorem 3.4, and Proposition 3.5 in \cite{Sun-Wang-Zhou20}. The proof follows that of Theorem 1.7 in \cite{Zhou19} verbatim. We will only provide a sketch here with an emphasis on some minor differences.

We start with a pulled-tight minimizing sequence $\{\Phi_i\}_{i\in\N}\subset \Pi$ (existence guaranteed by Lemma \ref{L:tightening lemma}); namely, every element of $\bC(\{\Phi_i\})$ has $c$-bounded first variation with respect to $\X_{\tan}(M)$. In the following we will need notions of discrete sweepouts, which can be found in \cite[Section 4]{Zhou-Zhu18}.

To use Almgren-Pitts combinatorial argument to deduce the existence of a capillary almost minimizing varifold in $\bC(\{\Phi_i\})$, we need to discretize the continuous sweepouts to discrete sweepouts. For each $\Phi_i$, we can apply \cite[Theorem 3.3]{Sun-Wang-Zhou20} to find a sequence of discrete maps:
\[ \phi_i^j: I(k_i^j)_0 \to \C(M), \]
with $k_i^j<k_i^{j+1}$ and a sequence of positive numbers $\de_i^j\to 0$ (as $j\to \infty$), such that 
\begin{itemize}
    \item[(i)] the fineness
    \[\f(\phi_i^j) = \sup\{\M\big((\partial\phi_i^j(x)- \partial\phi_i^j(y))\lc \iM\big): \text{$x, y$ are adjacent vertices in $I(k_i^j)_0$}\}  \leq \de_i^j;\]
    \item[(ii)] $\sup\{\F(\phi_i^j(x), \Phi_i(x)): x\in I(k_i^j)_0\}< \de_i^j$;
    \item[(iii)] $\M(\partial\phi_i^j(x)\lc \iM) \leq \M(\partial \Phi_i(x)\lc \iM)+\de_i^j$;
    \item[(iv)] $\sup\{\mF(\phi_i^j(x), \Phi_i^j(x)): x\in I(k_i)_0\} \leq \de_i^j$.
\end{itemize}
By (ii) (iv) and Corollary \ref{C:various closeness} (letting $\mathcal S = \Phi_i(I)$), we can find for each each $i$ a sufficiently large $j(i)$, such that, denoting $\varphi_i = \phi_i^{j(i)}$,  
\begin{equation}
\label{eq:first discretization inequality}
  \sup\{ \F(\partial^\theta\varphi_i(x), \partial^\theta\Phi_i(x)) + \mF(|\partial^\theta\varphi_i(x)|, |\partial^\theta\Phi_i(x)|): x\in \text{dmn}(\varphi_i) \} \leq \eta_i,  
\end{equation}
where $\eta_i\to 0$ as $i\to\infty$. 

Note that 
\[\partial\Om\lc \pM  = \frac{1}{1-\cos\theta}\big(\partial\Om - \partial^\theta\Om \big), \] so (ii) and ~\eqref{eq:first discretization inequality} imply that 
\[ 
\begin{aligned}
    \M\big(\partial\varphi_i(x)\lc \pM & - \partial \Phi_i(x)\lc \pM \big) = \F\big( \partial\varphi_i(x)\lc \pM - \partial \Phi_i(x)\lc \pM \big) \\
    & \leq \frac{1}{1-\cos\theta}\Big( \F(\partial\varphi_i(x) - \partial\Phi_i(x)) + \F(\partial^\theta\varphi_i(x) - \partial^\theta\Phi_i(x)) \Big) \to 0
\end{aligned} 
\]
uniformly in $x \in \text{dmn}(\varphi_i)$ as $i\to\infty$ (note that $\F$ in the first line denotes the flat norm in $\pM$, which coincides with the flat norm in $M$ for $n$-dimensional currents supported in $\pM$). This together with (ii) and (iii) implies that
\begin{equation}
\label{eq:second discretization inequality}
    \sup\{ \A(\varphi_i(x)): x\in \text{dmn}(\varphi_i) \} \leq \sup\{\A(\Phi_i(x)): x\in I\} + \eta_i',
\end{equation}
where $\eta_i'\to 0$ as $i\to\infty$.

Denote by $S=\{\varphi_i\}$, and define as usual the associated critical value $\bL(S)$ and critical set $\bC(S)$ as
\[ \bL(S) = \limsup_{i\to\infty} \sup\{\A\big(\varphi_i(y)\big): y\in \text{dmn}(\varphi_i)\}, \]
\[ \bC(S) = \{V=\lim_{j\to\infty}|\partial^\theta\varphi_{i_j}(y_j)| \text{ as varifolds: with } \lim_{j\to\infty}\A(\varphi_{i_j}(y_j)) = \bL(S)\}. \]

\begin{claim}
By ~\eqref{eq:first discretization inequality} and ~\eqref{eq:second discretization inequality}, We have $\bL(S)=\bL(\{\Phi_i\})$ and $\bC(S)=\bC(\{\Phi_i\})$.
\end{claim}

Now we are in the place to apply the Almgren-Pitts combinatorial argument to find a $V\in \bC(S)$ which is capillary almost minimizing in small annuli with respect to the $\M$-norm. 

\begin{claim}
There exists $V\in \V_n(M)$ satisfying: for any $p\in M$ and any small annulus $\mathcal A_{r_1, r_2}(p)$, there exists two sequences of positive numbers $\ep_j\to 0$, $\de_j\to 0$, a subsequence $\{i_j\}\subset \{i\}$, and $y_j\in \text{dmn}(\varphi_{i_j})$ (the domain of $\varphi_{i_j}$), such that
\begin{itemize}
    \item $\lim_{j\to\infty} \A(\varphi_{i_j}(y_j)) = \bL(S)$;
    \item $\varphi_{i_j}(y_j)\in \sA(\An_{r_1, r_2}(p); \ep_j, \de_j; \M)$; and
    \item $\lim_{j\to\infty} |\partial^\theta \varphi_{i_j}(y_j)|=V$.
\end{itemize}
\end{claim}
If Claim 2 were not true, then using the Almgren-Pitts combinatorial argument (see \cite[Theorem 1.16]{Zhou19}), we can find another sequence of discrete maps $\widetilde{S} = \{\widetilde{\varphi}_i\}$ so that $\widetilde{\varphi}_i$ is homotopic to $\varphi_i$ in the discrete sense (see \cite[page 22]{Sun-Wang-Zhou20}), but the critical value decreases strictly $\bL(\widetilde S)< \bL(S)$. One can then use the interpolations results \cite[Theorem 3.4, Proposotion 3.5]{Sun-Wang-Zhou20} to extend $\varphi_i$ and $\widetilde\varphi_i$ to continuous maps in $\mF$-topology
\[ \overline{\Phi}_i: \widetilde{\Phi}_i: I\to \big(\C(M), \mF\big). \]
Moreover, the maps are both homotopic to $\Phi_i$ in the $\mF$-topology, and
\[ \limsup_{i\to\infty}\sup\{\A(\widetilde\Phi_i(x)): x\in I\}\leq \bL(\widetilde{S})< \bL(S) =\bL(\{\Phi_i\}). \]
This is a contradiction to the definition of critical value. Hence we finish checking Claim 2. 

By Proposition \ref{P:equivalence of a.m.}, $V$ is capillary almost minimizing in any smaller annuli inside $\An_{r_1, r_2}(p)$. The proof is completed.  
\end{proof}

\section{Regularity of almost minimizing capillary varifolds}
\label{S:regularity}

In this section we complete the proof of Theorem \ref{T:main1} by proving regularity for the varifold obtained by Theorem \ref{T:existence of almost minimizing capillary varifolds}. Throughout this section $(M^3, g)$ will be a compact manifold with smooth boundary $\partial M$ and metric $g$ satisfying (\hyperref[def:generic]{$\star$}), $c\in\R$, and $\theta\in (0, \frac{\pi}{2}]$.

\subsection{Characterization of tangent varifolds and rectifiability}

We start this section by describing the possible types of tangent varifolds of a capillary almost minimizing varifold which has $c$-bounded first variation with respect to $\X_{tan}(M)$. We will identify the tangent space $T_pM$ of $M$ at a boundary point $p\in\pM$ with the upper half space $\mathbb R^3_+=\{(x_1, x_2, x_3)\in\mathbb R^3: x_3\geq 0\}$. Hence $T_p(\pM)$ is identified with 
\[B=\{x_3=0\}.\] Given a half space $H\subset \mathbb R^3_+$ so that $H\cap \partial \mathbb R^3_+=L$ is a line passing through the origin, we denote by $B_-$ the half of $\partial\mathbb R^3_+$ cut by $L$ which makes angle $\pi-\theta$ with $H$. 


\begin{theorem}[Characterization of tangent varifolds and rectifiability]
\label{thm:tangent-varifolds}
Let $V$ be the varifold obtained by Theorem \ref{T:existence of almost minimizing capillary varifolds}. Fix $c\in\mathbb{R}$ and $\theta < \frac{\pi}{2}$. Then the tangent varifolds of $V$ at $p\in\spt\|V\|\cap \pM$ 
is one of the following 4 cases: 

\begin{enumerate}
    \item\label{type1} $[H]+\cos\theta[ B_-]$ for some half plane $H$ and the associated half plane $B_-\subset B$;
    \item\label{type2} $[H+H']$, where $H,H'$ are two half planes that intersect along a common line $L\subset B$;
    \item\label{type3} $k[B]$ for some $k\in\mathbb N$;
    \item\label{type4} $(k+\cos\theta)[B]$ for some $k\in \mathbb N$.
\end{enumerate}

Furthermore, the type of tangent varifold is unique, that is: if one tangent varifold is of a particular type (1-4), then all are of the same type.
\end{theorem}

\begin{remark}
As the main challenge to obtain characterization of tangent cones for $V$, 
there is no natural way to extend $V$ to a varifold with bounded first variation, and one can not directly invoke Allard Rectifiability Theorem \cite{Allard72}. Nevertheless, we note that blowups of $V$ can be simply extended by reflection (across $T_p\pM$ which is a plane) as a stationary varifold. Similar observation first appeared in \cite{Li-Zhou16}.
\end{remark}

\begin{proof}
By the classical min-max theory for minimal hypersurfaces \cite{Almgren65, Pitts, Schoen-Simon81} and the CMC min-max theory \cite{Zhou-Zhu17}, we know that $V\lc G_2(\iM)$ is induced by a smooth almost embedded hypersurface $\Si$ of constant mean curvature $c$ (with multiplicity one when $c\neq 0$). 

 Now we divide the proof into the following three steps. 

\vspace{0.5em}
\noindent{\bf Step 1}: $V$ has uniform density lower bound at $p\in \spt\|V\|\cap \partial M$. 
\vspace{0.5em}

Fix $p\in \spt\|V\|\cap \pM$. Choose $0<r_0<r_{am}(p)/4$ small enough such that for any $r<r_0$ the mean curvature $H$ of $\partial B_r(p)\cap \widetilde{M}$ is greater than $c$. Let $V^*$ be a capillary replacement of $V$ in $\Clos(A_{r, 2r}(p))\cap M$ produced by Proposition \ref{P:good-replacement-property}.

By a standard argument using the Maximum Principle \ref{thm:vfld-mp}, 
 for any $r<s<2r$, we have 
\begin{equation}
\label{E:first replacement nontrivial}
 \spt \|V^*\|\cap (\partial B_s(p)\cap M)\neq \emptyset.   
\end{equation}
By Proposition \ref{L:regularity of capillary replacement}, inside $\An_{r, 2r}(p)\cap M$ the replacement $V^*$ is induced by $\partial^\theta\Omega + \sum_{i} m_i \Sigma_{(i)}$, where $\Si=\spt (\|\pr^\theta\Omega\| \lc \mathring{M})$ is a smooth, almost embedded, capillary constant mean curvature surface of mean curvature $c$, and the $\Sigma_{(i)}$ are disjoint, connected, smooth embedded, minimal surfaces (clearly no $\Sigma_{(i)}$ arise if $c\neq 0$). 

We remark that $\Si$ may be empty (when $\partial^\theta\Omega$ is supported entirely on $\pM$) or a surface with no boundary. 

By (\ref{E:first replacement nontrivial}), there exists some $q\in \spt\|V^*\|\cap (\partial B_{3r/2}(p)\cap M)$. We may choose $q$ to have a uniform density lower bound; we do so differently according to the following cases:

\begin{itemize}
    \item[Case 1:] $\Si$ is empty. Then $\partial^\theta\Omega$, which is nontrivial, must be identical to $\cos\theta (\pM \cap \An_{r, 2r}(p))$. In this case, we pick $q\in \partial B_{3r/2}(p)\cap \pM$. Then $\Theta^2(\|V^*\|, q)\geq \cos\theta$.
    \item[Case 2:] At least one $\Si_{(i)}\neq \emptyset$, or $\Sigma\neq \emptyset$ and at least one connected component $\Sigma_0$ of $\Si$ satisfies $\partial\Si_0\cap \mathcal \An_{r, 2r}(p)=\emptyset$. That is, one of the $ \Sigma_{(i)}$ or $\Sigma_0$ is a smooth nonempty surface with no boundary in $\An_{r, 2r}(p)\cap M$. In this case, we pick $q\in \Si_0\cap \partial B_{3r/2}(p)$ or $q\in \Si_{(i)}\cap \partial B_{3r/2}(p)$. Then $\Theta^2(\|V^*\|, q)\geq 1$.
    \item[Case 3:] $\Si\neq \emptyset$, and all components of $\Si$ have nontrivial boundary on $\An_{r, 2r}(p)\cap \pM$. In this case, we pick $q\in \partial B_{3r/2}(p)\cap \Si$. Then $\Theta^2(\|V^*\|, q)\geq \frac{1}{2}(1+\cos\theta)$. 
\end{itemize}

With the choice of $q$, we know by the monotonicity formulae and Proposition \ref{P:good-replacement-property} (i) and (ii) that ($C_i$, $i=1, 2, \cdots$ denote uniform constants depending only on $\widetilde{M}$)
\[
\begin{aligned}
\frac{\|V\|(B_{4r}(p))}{\pi(4r)^2} 
& \geq \frac{\|V^*\|(B_{4r}(p))-c\cdot \vol(\An_{r, 2r}(p)\cap M)}{\pi(4r)^2} \geq \frac{\|V^*\|(B_{r/2}(q))-c\cdot C_1 r^3}{\pi(4r)^2}\\
& \geq \frac{1}{64 C_2} \lim_{s\to 0}\frac{\|V^*\|(B_s(q))}{\pi s^2}- c\cdot C_3r = \frac{1}{64 C_2} \Theta^2(\|V^*\|, q)- c\cdot C_3r\\
& \geq \frac{1}{128 C_2} -c\cdot C_3 r.
\end{aligned}
\]
As a crucial ingredient, in the second inequality we have used the fact that $\vol(\An_{r, 2r}\cap M)\leq C_1 r^3$. Note that we have used the interior Monotonicity Formula \cite{Simon83} in Case 2, and the boundary version \cite[Theorem 2.3]{Li-Zhou16} in Cases 1 and 3. 

Since the above inequality holds true for all $r>0$ small enough, letting $r\to 0$, we derive the uniform lower density bound $\Theta^2(\|V\|, p)\geq \frac{1}{128 C_2}$.


\vspace{0.5em}
\noindent{\bf Step 2}: Given a tangent varifold $C\in \VarTan(V, p)$, where $p\in \spt\|V\|\cap \partial M$, $C$ is a stationary $2$-rectifiable cone in $T_pM$ with free boundary. 
\vspace{0.5em}

Here by \cite[Definition 2.1]{Li-Zhou16}, a varifold is stationary with free boundary in $M$ if and only if it has $0$-bounded first variation with respect to $\X_{tan}(M)$. Let $V_i=(\bleta_{p, r_i})_{\#}V$ for some $r_i\to 0$ such that 
\[C=\lim V_i\in \V_2(T_pM) \quad \text{as varifolds}.\] 
Here $\bleta_{p, r}: \R^L\to \R^L$ is the rescaling map $\bleta_{p, r}(x)=\frac{x-p}{r}$. We know that $\bleta_{p, r_i}(M)$ converges smoothly to the half space $T_pM$. Hence the blowup limit $C$ is stationary with respect to $\X_{tan}(T_pM)$, and is stationary with free boundary in $T_pM$. By the varifold convergence and the uniform area ratio lower bound of $V$ as in Step 1, we have $\Theta^2(\|C\|, q)\geq \theta_0>0$ for all $q\in\spt\|C\|$.

Consider the doubled varifold $\overline{C}$ (obtained by taking the union of $C$ and its reflection to $T_p\widetilde{M}\setminus T_pM$ across $T_p(\pM)$). Then $\overline{C}$ is stationary in $T_p\widetilde{M}$ by \cite[Lemma 2.2]{Li-Zhou16}. Note that $\overline{C}$ has uniform density lower bound $\Theta^2(\|C\|, q)\geq \min\{1, 2\theta_0\}>0$ for every $q\in \spt\|\overline{C}\|$, so $\overline{C}$ is rectifiable by the Recifiability Theorem \cite[42.4]{Simon83}. Therefore $\overline{C}$ is a rectifiable cone by \cite[Theorem 19.3]{Simon83}, and so is $C$.

\vspace{0.5em}
\noindent{\bf Step 3}: Now we use the smooth convergence for stable capillary surfaces to obtain characterization of tangent varifolds. 
\vspace{0.5em}

Pick $\al\in (0, \frac{1}{4})$. For each $i$ large enough, we can apply Proposition \ref{P:good-replacement-property} with the compact set $K_i=\Clos(\An_{\al r_i, r_i}(p)\cap M)$ to obtain a replacement $V^*_i\in \V_2(M)$ of $V$ in $K_i$. As in (\ref{E:first replacement nontrivial}), we know that $\|V^*_i\|\lc \An_{\al r_i, r_i}(p) \neq 0$.
By Proposition \ref{P:good-replacement-property} (ii), there exists some $C_1>0$ depending only on $M$ so that for any $R>r_i$,
\begin{equation}
\label{E:mass comparison between V and its replacement}
    -c\cdot C_1 r_i^3 \leq \|V\|(B_R(p))-\|V^*_i\|(B_R(p))\leq c\cdot C_1 r_i^3.
\end{equation}
By weak compactness of varifolds with bounded total mass, after passing to a subsequence, we obtain varifold limit:
\[ C^*:=\lim_{i\to\infty} (\bleta_{p, r_i})_{\#}V_i^* \in \V_2(T_pM). \]

Now $C^*$ satisfies the following properties.
\begin{enumerate}

    \item By Proposition \ref{P:good-replacement-property} (i), $C$ is identical to $C^*$ outside $\Clos(\An_{\al, 1}(0))$.
    \item By Proposition \ref{P:good-replacement-property} (iii) and (v), $V^*_i$ is also capillary almost minimizing in small annuli and has $c$-bounded first variation with respect to $\X_{tan}(M)$. So by similar arguments as in Step 1 and 2, we know that $V^*_i$ has uniform lower density bound, and hence $C^*$ is rectifiable and stationary with free boundary in $T_pM$.
    \item As a crucial corollary of (\ref{E:mass comparison between V and its replacement}), for any $R> 1$, we have
    \[ -c\cdot C_1 r_i \leq \|(\bleta_{p, r_i})_{\#}V\|(B_R(p)) - \|(\bleta_{p, r_i})_{\#}V_i^*\|(B_R(p)) \leq c\cdot C_1 r_i. \]
    Therefore, letting $i\to\infty$, we have
    \[ \|C^*\|(B_R(0)) = \|C\|(B_R(0)),\quad \text{for any } R>1. \]

    \item Lastly, by Proposition \ref{L:regularity of capillary replacement} and Theorem \ref{theorem.compactness.almost.properly.embedded.capillary.surfaces}, we know that $C^*\lc (\An_{\al, 1}(0)\cap \iM)$ is an integer multiple of an embedded, stable, capillary minimal surface, and if a connected component $\Si$ of $C^*\lc (\An_{\al, 1}(0)\cap \iM)$ has non-empty boundary on $\An_{\al, 1}(0)\cap \pM$, then the multiplicity of $\Si$ is one.
\end{enumerate}

\vspace{0.5em}
\textit{Claim 1: $C=C^*$}.
\vspace{0.5em}

\textit{Proof of Claim 1}: We consider the doubled varifolds $\overline{C}, \overline {C^*} \in \V_2(T_p\widetilde{M})$ of $C, C^*\in \V_2(T_pM)$ across $T_p(\pM)$ as in Step 2. We will to show that $\overline{C}$ and $\overline{C^*}$ coincide. This follows from a standard argument as in \cite[Theorem 7.8]{Pitts}. Indeed, Since $\overline{C}$ is a rectifiable cone, we have 
\[ \frac{\|\overline{C}\|(B_r(0))}{\pi r^2} \equiv \text{constant}; \]
by (1) and (3), $\overline{C}$ and $\overline{C^*}$ have the same density at $0$ and the same area ratio near infinity. Using that $\overline{C^*}$ is stationary and rectifiable, the classical monotonicity formula implies that
\[ \frac{\|\overline{C^*}\|(B_r(0))}{\pi r^2} \equiv \text{constant}; \]
thus $\overline{C^*}$ is also a cone by \cite[Theorem 19.3]{Simon83}; by (1) again, we have $\overline{C}=\overline{C^*}$ and hence $C=C^*$.\qed

\vspace{1em}
Next, we prove the characterization of $C$. 

\vspace{0.5em}
\textit{Claim 2: $C$ is one of the following,}
\begin{enumerate}[label=$(\roman*)$]
    \item $[H+\cos\theta B_-]$, where $H$ intersects with $T_p(\pM)$ with angle $\theta$ along a line $L\subset T_p(\pM)$, and $B_-$ is half of $T_p(\pM)$ cut by $L$ making angle $\pi-\theta$ with $H$;
    \item $[H+H']$ where $H,H'$ are two half planes which intersect along a line $L\subset T_p(\pM)$ and each make angle $\theta$ with $T_p(\pM)$;
    \item $(k+\cos\theta)[T_p(\pM)]$ or $k'[T_p(\pM)]$ for some $k, k'\in\N$.
\end{enumerate}
    
\textit{Proof of Claim 2}: Since $\overline{C}$ is a rectifiable stationary cone, the restriction to the unit sphere $\overline{C}\lc S_1(0)$ is a rectifiable stationary varifold. Also as the densities of $\overline{C}$ have a uniform lower bound by Step 2, so does $\overline{C}\lc S_1(0)$. Hence by \cite{Allard-Almgren76}, $\overline{C}\lc S_1(0)$ consists of geodesic networks. Since $C=C^*$ and by (4), we know that $C\lc \interior(T_pM)$ is smoothly embedded, so $C$ consists of finitely many half planes intersecting along a line which passes $0$ and lies in $T_p(\pM)$. 

We first assume that $C\lc \interior(T_pM)\neq \emptyset$. By Theorem \ref{theorem.compactness.almost.properly.embedded.capillary.surfaces}, $C\lc \interior(T_pM)$ consists of one half plane $H$ intersecting with $T_p(\pM)$ along a line $L$ with angle $\theta$, or two half planes $H_1$ and $H_2$ which each intersect $T_p(\pM)$ along a common line $L$ with angle $\theta$. In the first case, by Theorem \ref{theorem.compactness.almost.properly.embedded.capillary.surfaces} again, $C=[H]+\cos\theta [B_-]$, where $B_-$ is half of $T_p(\pM)$ cut by $L$ making an angle $\pi-\theta$ with $H$. In the second case, by by Theorem \ref{theorem.compactness.almost.properly.embedded.capillary.surfaces}, $C=[H]+[H']$.

Now we assume that $C\lc \interior(T_pM)= \emptyset$. In this situation, $C$ can only be supported on $T_p(\pM)$, so it is a constant multiple of $T_p(\pM)$ by the Constancy Theorem \cite[\S 41]{Simon83} (applied to $\overline{C}$). By Theorem \ref{theorem.compactness.almost.properly.embedded.capillary.surfaces} and (4), $C$ is then the smooth limit of almost properly embedded stable CMC surfaces together with possible a $\cos\theta$ multiple of $\bleta_{p, r_i}(\pM)$, so $C$ is either $(k+\cos\theta)[T_p(\pM)]$ or $k'[T_p(\pM)]$ for some $k, k'\in\N$.\\

It remains to show that the type of tangent varifold is unique. The density immediately distinguishes types (1) and (4) as unique: Note that in type (1), the density is $\frac{1}{2}(1+\cos\theta) <1$, whilst types (2-3) have integer density and type (4) has non-integer density $k+\cos\theta >1$. Thus we only need to show that one cannot have one tangent varifold of type (2) and another of type (3; $k=1$). 

Let $\mathcal{C}$ be the set of all planes in $T_pM$ intersecting $B=T_p\partial M$ at angle $\theta$. Note that all elements of $\mathcal{C}$ are rotations of each other around $B$. Then $\mathcal{C}$ is clearly separated from $B$ in $G_2(\mathbb{R}^L)$; in particular there is a neighbourhood $\mathcal{U}$ of $B$ in $G_2(\mathbb{R}^L)$, disjoint from $\mathcal{C}$, and a smooth function $\beta$ supported in $\mathcal{U}$ such that $\beta(B)=1$ (and $\beta(\mathcal{C})=0$)

It then follows from the definition of $\mathbf{F}$-metric that $\mathbf{F}([B], \mathcal{C}_+)>0$, where $\mathcal{C}_+$ be the set of all type (2) varifolds, i.e. induced by the union of two capillary half-planes intersecting along a common line. But the set of tangent varifolds must be connected (as it corresponds to the limit points of the connected set $\{(\eta_{p,r})_\#(V)| r>1\}$), so this implies the uniqueness of tangent cone type.

\end{proof}


\subsection{Regularity of capillary min-max varifolds}

In this section, we prove the main regularity Theorem \ref{T:boundary regularity}, and by doing so complete the proof of Theorem \ref{T:main1}. The idea for regularity is to get as far as we can using the interior gluing and regularity which was already established in \cite{Pitts, Schoen-Simon81, Zhou-Zhu17}. However, there is a subtlety at the boundary, which can be overcome using the smooth maximum principle so long as the replacement is $C^1$ up to the gluing interface. For this reason, we use a double replacement method. 

In this section, a \textit{regular capillary surface} will refer to a smooth, almost properly embedded surface $\Sigma$ of constant mean curvature $c$ and (smooth) boundary $\pr \Sigma$, which meets $\pr M$ with angle $\theta$. The smooth maximum principle we need is as follows:

\begin{lemma}
\label{lem:gluing-smooth-mp}
Let $c\in \mathbb{R}$ and $\theta\in(0,\frac{\pi}{2})$. Consider a connected, regular capillary surface $\Sigma$ in the annular region $\wt{\An}_{s_1,s_2}(p)$ which is $C^1$ up to the outer interface $\wt{S}^+_{s_2}(p)$. 

There exists $s_*=s_*(n,c,\theta)$ and $s_0 = s_0(n,c,\theta,s_2)$ such that if $s_1<s_0<s_2<s_*$, and $\Sigma$ is as above, then either:
\begin{enumerate}
    \item $\Clos(\Sigma) \cap \wt{S}^+_{s_2}(p) \cap \mathring{M} \neq \emptyset$; or
    \item $\Clos(\Sigma)$ is tangent to $\pr M$ at every point of $\Clos(\Sigma) \cap \wt{S}^+_{s_2}(p) \subset \pr M$.
\end{enumerate}
\end{lemma}

That is, either $\Sigma$ meets the outer interface at an interior point, or it only meets the outer interface along false boundary points. Note that $\Clos(\Sigma) \cap \wt{S}^+_{s_2}(p)$ is nonempty for small enough $s_*$ by Theorem \ref{thm:vfld-mp}. The proof of Lemma \ref{lem:gluing-smooth-mp} is deferred to Appendix \ref{sec:gluing-smooth-mp}. With the lemma in hand, we can establish the main regularity result:

\begin{theorem}\label{T:boundary regularity}
Let $V$ be as in Theorem \ref{T:existence of almost minimizing capillary varifolds}. Then $V$ is a regular capillary varifold in $M$. That is, $V\lc \iM$ is induced by a nontrivial, smooth, closed, almost properly embedded surface $\Si\subset M$ with (possibly empty) smooth boundary $\partial \Si\subset \pM$, and $\Sigma$ has prescribed mean curvature $c$ and meets $\pM$ with contact angle $\theta$. Moreover, when $c\neq 0$ and $\theta\neq \frac{\pi}{2}$, $V$ has multiplicity one.
\end{theorem}
\begin{proof}
Recall that we fix $c\in \mathbb{R}$ and $\theta\in(0,\frac{\pi}{2})$. 

The interior regularity of $V\lc \mathring{M}$ was established by \cite{Pitts, Schoen-Simon81, Zhou-Zhu17}. In particular, $\Sigma= \spt\|V\|\cap\mathring{M}$ is a smooth almost properly embedded surface of constant mean curvature $c$. To prove the boundary regularity, we need to show that $\Sigma$ extends smoothly to $\pr M$ as a regular capillary surface.

Let $p\in \spt\|V\|\cap \pM$. Recall that $\widetilde{B}^+_r(p)$ and $\widetilde{S}^+_r(p) = \partial_{rel} \widetilde{B}^+_r(p)$ denote the Fermi half-balls and half-spheres respectively. Fix $r_0\in (0, r_{am}(p)/4)$ sufficiently small so that Theorem \ref{thm:vfld-mp} applies on $\widetilde{B}^+_r(p)$ for all $0<r\leq r_0$ to show that for any $W \in \mathcal{V}_n(M)$ which is stationary with $c$-bounded first variation in $\widetilde{B}^+_r(p)$, and $\|W\|\lc \widetilde{B}^+_r(p) \neq 0$, we have
\begin{equation}
\label{eq:mp-inside-outside}
0\neq \spt\|W\|\cap \wt{S}^+_r(p) = \Clos(\spt \|W\| \setminus \Clos\big(\widetilde{B}^+_r(p))\big) \cap \wt{S}^+_r(p). 
\end{equation}

The idea will be to use the previously established interior gluing to simplify matters, using unique continuation and noting that the first replacement is regular through the gluing interface. The subtle case to rule out is when the second replacement intersects the gluing interface in a subset of the barrier $\partial M$ which is disjoint from $\Sigma'$. 

\vspace{0.5em}
\textbf{Step 1:} \textit{Replacement with boundary structure. }
\vspace{0.5em}

Let $s_*$ be as given by Lemma \ref{lem:gluing-smooth-mp}.

Take a replacement $V^*$ on $\wt{\An}_{s,t}(p)$, with $s<t\ll \min(r_{am}(p),s_*)$. 
By Proposition \ref{L:regularity of capillary replacement}, $V^*$ is regular in $\wt{\An}_{s,t}(p)$, and we set $\Sigma' = \spt (\|V^*\| \lc \wt{\An}_{s,t}(p)\cap \iM)$. 
By regularity of $V^*$ (up to the boundary), the self-touching set $\mathcal{S}(\Sigma')$ and boundary $\pr\Sigma'$ are $1$-rectifiable. Moreover, by property (\hyperref[def:generic]{$\star$}), the barrier-touching set $\mathcal{T}(\Sigma')$ is also $1$-rectifiable. Then we may choose $s_2 \in (s,t)$ so that:
\begin{itemize}
    \item[(T)] \label{def:transverse} $\wt{S}^+_{s_2}(p)$ intersects $\Sigma',\partial \Sigma'$ transversely, and $\mathcal{S}(\Sigma'), \mathcal{T}(\Sigma')$ \textit{sparsely}, that is, in sets of zero $1$-dimensional measure. 
\end{itemize}

We then likewise take a second replacement $V^{**}$ of $V^*$ on $\wt{\An}_{s_1,s_2}(p)$, where $s_1< \min(s,s_0)$ and $s_0$ is again given by Lemma \ref{lem:gluing-smooth-mp}. In particular note $s_1 < s < s_2 < t$.

As in \cite[Step 1, Section 6]{Zhou-Zhu17}, the second replacement $V^{**}$ may be constructed so that:

\vspace{0.5em}
\textit{Claim 1:} There exists $\Omega^{**}\in\mathcal{C}(M)$ satisfying the following: Suppose $q\in \spt\|V^{**}\| \cap \wt{S}^+_{s_2}(p)$, and that $\Sigma', \Sigma''$ have multiplicity 1 in a neighbourhood of $q$. Then there exists $\epsilon>0$ so that:
\begin{enumerate}[label=(\alph*)]
\item $V^*\lc \wt{\An}_{s_2,t}(p)$ and $V^{**}$ are given by by $|\partial^\theta\Omega^{**}|$ in $\wt{\An}_{s_2,t}(p)\cap \wt{B}^+_\epsilon(q)$ and $\wt{\An}_{s_1,s_2}(p)\cap \wt{B}^+_\epsilon(q)$ respectively. 
\item if $\|V^{**}\|(\wt{S}^+_{s_2}(p)\cap\wt{B}^+_\epsilon(q))=0$, then $V^{**}$ is given by $|\partial^\theta\Omega^{**}|$ in $\wt{\An}_{s_1,t}(p)\cap \wt{B}^+_\epsilon(q)$. 
\end{enumerate}


\begin{proof}[Proof of Claim 1]
Fix $0<\tau<s_1$. By Proposition \ref{P:good-replacement-property} we have $V^*= \lim_{i\to \infty} |\partial^\theta \Omega_i^*|$ for some $\Om^*_i\in \sA(\wt{\An}_{\tau,r_0}(p); \ep_i, \de_i; \F)$ with $\ep_i, \de_i \to 0$. Let $U= \wt{\An}_{s_1,s_2}(p)\cap \wt{B}^+_\epsilon(q)$. The regularity of $\partial^\theta\Omega_i^*$ in $\wt{\An}_{s,t}(p)$, the compactness Theorem \ref{theorem.compactness.almost.properly.embedded.capillary.surfaces} and the multiplicity 1 assumption imply that the $\pr^\theta\Omega_i^*$ converge locally smoothly to $\Sigma'$ in $U \cap \mathring{M}$, for small enough $\epsilon$. 

For each $\Omega_i^*$, we may apply Lemma \ref{L:minimisation} in $K=\Clos(\wt{\An}_{s_1,s_2}(p))$ to construct a new sequence $\Omega_i^{**}$ satisfying:
\begin{itemize}
    \item $\spt(\Omega_i^*- \Omega_i^{**})\subset K$;
    \item $\Omega_i^{**}$ is locally $\A$-minimising in $\interior_M(K)$;
    \item $V^{**}= \lim_{i\to \infty} |\partial^\theta \Omega_i^{**}|$;
    \item $\partial^\theta \Omega_i^{**} \lc (U\cap\mathring{M})$ converges locally smoothly to $\Sigma''$ (as in the proof of Proposition \ref{P:good-replacement-property}).
\end{itemize}

Now by Lemma \ref{L:weak convergence of reduced currents}, up to taking a subsequence $\partial^\theta\Omega_i^{**}$ will converge weakly as currents to $\partial^\theta \Omega^{**}$ for some $\Omega^{**} \in \mathcal{C}(M)$. Claim 1(a) follows from the locally smooth convergence. The weak convergence implies that $\|\partial^\theta \Omega^{**}\|(U) \leq \|V^{**}\|(U)$. If $\|V^{**}\|(\wt{S}^+_{s_2}(p)\cap\wt{B}^+_\epsilon(q))=0$, then the locally smooth convergence implies that, in fact, $\|\partial^\theta \Omega^{**}\|(U) = \|V^{**}\|(U)$. Claim 1(b) then follows from \cite[2.1(18)(f)]{Pitts} (which only assumes rectifiability).
\end{proof}

Again, by Proposition \ref{L:regularity of capillary replacement}, $V^{**}$ is regular in $\wt{\An}_{s_1,s_2}(p)$. We set $\Sigma'' = \spt(\|V^{**}\| \lc \wt{\An}_{s_1,s_2}(p)\cap \iM)$.  

\vspace{0.5em}
\textbf{Step 2:} \textit{Gluing $V^*$ and $V^{**}$ along $\wt{S}^+_{s_2}(p)$. }
\vspace{0.5em}

The goal is to show that $V^* \lc \wt{\An}_{s,s_2}(p) = V^{**} \lc \wt{\An}_{s,s_2}(p)$. It suffices to show that $\Sigma' \cap \wt{\An}_{s,s_2}(p) = \Sigma'' \cap \wt{\An}_{s,s_2}(p)$ and that the densities match along $\wt{S}^+_{s_2}(p)$.


Let $\Sigma'_i$ be (the relative closures in $\wt{\An}_{s,t}(p)$ of) the connected components of $\Sigma' \cap \mathring{M}$. Similarly let $\Sigma''_j$ be (the relative closures in $\wt{\An}_{s_1,s_2}(p)$ of) the connected components of $\Sigma'' \cap \mathring{M}$. We will use the following claim, which is immediate by unique continuation:

\vspace{0.5em}
\textit{Claim 2 (Gluing)}: Suppose that for some $i,j$, there is $q \in \wt{S}^+_{s_2}(p)$ and a neighbourhood $U$ of $q$ so that we have the local gluing $\Sigma'_i \cap U \cap \wt{\An}_{s,s_2}(p) = \Sigma''_j \cap U$. Then the components match, i.e. 
$\Sigma'_i \cap \wt{\An}_{s,s_2}(p) = \Sigma''_j \cap \wt{\An}_{s,s_2}(p)$.

Recall that $\Sigma'$ is a regular capillary surface on $\wt{\An}_{s,t}(p)$, in particular across $\wt{S}^+_{s_2}(p)$. Then Lemma \ref{lem:gluing-smooth-mp} and property (\hyperref[def:generic]T) imply that:
\begin{itemize}
    \item[($\dagger$)] Every component $\Sigma'_i$ must intersect $\wt{S}^+_{s_2}(p)$ inside $\mathring{M}$. 
\end{itemize}
(In particular, the sparse intersection with $\mathcal{T}(\Sigma')$ rules out case (2) of Lemma \ref{lem:gluing-smooth-mp}, and transversality implies that each component has points on both sides of $\wt{S}^+_{s_2}(p)$.)

\vspace{0.5em}
\textbf{Step 2a:} \textit{Gluing across interior points. }

In the interior $\mathring{M}$, the gluing procedures in \cite[Step 2, Section 6]{Zhou-Zhu17} (for $c\neq 0$) and \cite[Lemma 7.10]{Pitts} (for $c=0$) proceed verbatim. In particular, from the arguments in \cite[Step 2, Section 6]{Zhou-Zhu17}, we have that $\Sigma'$ glues smoothly with $\Sigma''$ across the interior intersection
\[\Gamma:= \Clos(\Sigma'') \cap \wt{S}^+_{s_2}(p) \cap \mathring{M} = \Sigma'\cap \wt{S}^+_{s_2}(p) \cap\mathring{M}.\] 

(For the readers' convenience we sketch that: for the forward inclusion, one may apply (\ref{eq:mp-inside-outside}) to $V^{**}$ and take the intersection with $\mathring{M}$; 
the reverse inclusion follows from the transverse intersection of $\Sigma'$ with $\wt{S}^+_{s_2}(p)$.)

By ($\dagger$) and Claim 2, the gluing across $\Gamma$ implies that $\Sigma' \cap \wt{\An}_{s,s_2}(p) \subset \Sigma'' \cap \wt{\An}_{s,s_2}(p).$ (It remains to rule out new components of $\Sigma''$ inside $\wt{\An}_{s,s_2}(p)$.)

In particular, $\Clos(\Sigma'')$ is a $C^1$-regular surface with boundary near any point of $\Clos(\Gamma)$. 

\vspace{0.5em}
\textbf{Step 2b:} \textit{$C^1$ regularity in the exceptional case. }

We now investigate the regularity of the remaining interface points of $\Sigma''$, that is, those $q\in (\Clos(\Sigma'') \cap \wt{S}^+_{s_2}(p)\cap \partial M)\setminus \Clos(\Gamma)$. As $(\Sigma',\pr \Sigma')$ intersect $\wt{S}^+_{s_2}(p)$ transversely, any such $q$ has a neighbourhood $U$ on which $U\cap \Clos(\Sigma') = \emptyset$ and $U\cap \Clos(\Sigma'') \cap \wt{S}^+_{s_2}(p)\subset \partial M$.

Consider now a tangent cone $C^{**} = \lim_{i\to\infty} (\mathbf{\eta}_{q,r_i})_\# V^{**}$ of $V^{**}$ at $q$, and the corresponding tangent cone $C^{*} = \lim_{i\to\infty} (\mathbf{\eta}_{q,r_i})_\# V^{*}$ of $V^{*}$. The gluing interface blows up to $T_q\wt{S}^+_{s_2}(p)$, which divides $T_q M$ into two quarter-spaces corresponding to $\wt{\An}_{s,s_2}(p)$ and $\wt{\An}_{s_2,t}(p)$. We refer to these as the $V^{**}$ and $V^*$ side respectively. 

Since $V^{**}$ coincides with $V^*$ on $\wt{\An}_{s_2, t}(p)$, the tangent cones $C^*$ and $C^{**}$ must coincide on the $V^*$ side. But $U\cap \Clos(\Sigma')=\emptyset$, so $C^*$ has no support in the interior and must be either $0$ or $\cos\theta[T_q\partial M]$. Thus the only possibilities for $C^{**}$ are types (1) or (4; $k=0$) in Theorem \ref{thm:tangent-varifolds}.

Now as the density at any point on $\Sigma''$ is a positive integer, so the $A_{s,s_2}(p)$ side of $C^{**}$ contributes density at least $\frac{1}{2}$, This rules out case (4) for $C^{**}$, which has density $\cos\theta$ everywhere. This leaves case (1) - a capillary half-plane - and since $U\cap \Clos(\Sigma'') \cap \wt{S}^+_{s_2}(p) \subset \partial M$, it can only be the unique capillary half-plane on the $V^{**}$ side which has boundary $T_q(\wt{S}^+_{s_2}(p) \cap \partial M)$. 

Thus we have shown that any tangent cone at $q$ is a unique capillary half-plane. This implies that $V^{**}$ is a $C^1$-regular surface with boundary near $q$.  

\vspace{0.5em}
\textbf{Step 2c:} \textit{There is no exceptional case.}

Combining Steps 2a and 2b, we have that $\Clos(\Sigma'')$ is (at least) $C^1$ regular up to the gluing interface $\wt{S}^+_{s_2}(p)$. Then by the smooth maximum principle (Lemma \ref{lem:gluing-smooth-mp}), we have that:
\begin{itemize}
    \item[($\ddagger$)] Each component $\Sigma''_j$ either has nontrivial intersection with $\wt{S}^+_{s_2}(p)\cap \mathring{M}$, or intersects $\wt{S}^+_{s_2}(p)$ only in points tangent to $\pr M$. 
\end{itemize} 

Suppose $\Sigma''_j$ is tangent to $\pr M$ at $q\in \wt{S}^+_{s_2}(p)$. Then $T_q V^{**} = \Theta[T_q \pr M]$, where $\Theta \geq 1$. Since $V^{**}$ coincides with $V^{*}$ on $\wt{\An}_{s_2,t}(p)$, it follows that $T_q V^*= T_q V^{**} = \Theta[T_q \pr M]$ and in particular $q\in \spt\|V^*\|$. If $q\in \spt\|V^*\| \setminus \Sigma'$ then the density at $q$ would be $\cos\theta<1$, which contradicts $\Theta\geq1$. So we must have $q\in \Sigma'$. Then since $\wt{S}^+_{s_2}(p)$ intersects $(\Sigma',\pr\Sigma')$ transversely and $\mathcal{T}(\Sigma')$ sparsely, we conclude that $q\in \Clos(\Sigma' \cap \wt{S}^+_{s_2}(p)\cap \mathring{M})=\Clos(\Gamma)$. 

Thus in either case of ($\ddagger$), the gluing across $\Clos(\Gamma)$, together with Claim 2, now implies that $\Sigma''_j \cap \wt{An}_{s,s_2}(p)\subset \Sigma' \cap \wt{\An}_{s,s_2}(p)$ for each $j$. In particular, the exceptional case of Step 2b does not occur \textit{a posteriori}. Together with Step 2a we conclude that $\Sigma' \cap \wt{An}_{s,s_2}(p) = \Sigma'' \cap \wt{\An}_{s,s_2}(p)$.

\vspace{0.5em}
\textbf{Step 2d:} \textit{Density matching.}

Consider $q\in \Sigma' \cap \wt{S}^+_{s_2}(p) = \Clos(\Sigma'') \cap \wt{S}^+_{s_2}(p)$. As in Step 2b, the tangent cones of $V^*$ and $V^{**}$ at $q$ must coincide on the quarter-space of $T_qM$ corresponding to $\wt{An}_{s_2,t}(p)$. Since $\pr \Sigma'$ intersects $\wt{S}^+_{s_2}(p)$ transversely, it follows from the classification of tangent cones Theorem \ref{thm:tangent-varifolds} that $\Theta^n(\|V^*\|,q) = \Theta^n(\|V^{**}\|,q)$. We conclude that $V^* \lc \wt{\An}_{s,s_2}(p) = V^{**} \lc \wt{\An}_{s,s_2}(p)$ as desired.

\vspace{0.5em}
\textbf{Step 3:} \textit{Unique continuation up to the centre. }
\vspace{0.5em}

By unique continuation, we may repeat the above process for arbitrarily small $s_1$, and define $\tilde{V} = \lim_{s_1\to0} V^{**}_{s_1}$. The surface $\tilde{\Sigma} = \bigcup_{s_1} \Sigma''_{s_1}$ will be a regular capillary surface in $\wt{B}^+_{s_2}(p)\setminus \{p\}$. 

\vspace{0.5em}
\textbf{Step 4:} \textit{Removable singularity at $p$. }
\vspace{0.5em}

We now wish to extend $\tilde{\Sigma}$ as a regular capillary surface through $p$. A blowup argument as in the proof of Theorem \ref{thm:tangent-varifolds} implies that all tangent cones of $\tilde{V}$ at $p$ are of one of the types (1-4). (Noe that $\tilde{V}$ was already constructed via continued replacements, so that argument actually goes through without taking further replacements.) We now divide into cases depending on the tangent cone type:

\vspace{0.5em}
\textit{Case 1:} All tangent cones of $\tilde{V}$ at $p$ are of type (1). 

That is, every tangent cone is induced by a multiplicity 1 capillary half-plane. Then by Claim 1, $\tilde{\Sigma}$ is locally the boundary of some $\tilde{\Omega}$ in $\wt{B}^+_\sigma(p)\setminus \{p\}$. (Note that the hypothesis of Claim 1(b) is automatically satisfied as a consequence of interior gluing.) So we immediately deduce from Theorem \ref{theorem:removal.singularity} that $\tilde{\Sigma}$ extends as a regular capillary surface through $p$ (with unique tangent cone). 

\vspace{0.5em}
\textit{Case 2:} All tangent cones of $\tilde{V}$ at $p$ are of type (2). 

That is, every tangent cone is induced by a pair of multiplicity 1 capillary half-planes. Again by Claim 1, $\tilde{\Sigma}$ is locally the boundary of some $\tilde{\Omega}$ in $\wt{B}^+_\sigma(p)\setminus \{p\}$.
The plan is to apply Theorem \ref{theorem:removal.singularity} to `both sides'. First, since $\tilde{\Sigma}$ is (boundary) regular away from $p$, for sufficiently small $\sigma_0>0$ we have $\partial \tilde{\Sigma}\cap \Clos(\wt{B}^+_{\sigma_0}(p)) = \gamma$, where each $\gamma$ is the union of finitely many continuously embedded segments in $\Clos(\wt{B}^+_{\sigma_0}(p))$, smooth in $\wt{B}^+_{\sigma_0}(p)\setminus \{p\}$, each with one endpoint at $p$ and the other on $\wt{S}^+_{\sigma_0}(p)$. 

Take local coordinates so that $M$ corresponds to $\{x_3 >0\}$, and let $C_\gamma$ be the surface corresponding to the cylinder $\gamma \times \mathbb{R}_+$ over $\gamma$. We say that a surface $S$ is \textit{separated by} $C_\gamma$ if it consists of precisely 2 connected components $S_1,S_2$, each in a distinct component of $(B_s(p) \cap \mathring{M}) \setminus C_\gamma$. We claim that for any sufficiently small $\sigma>0$, the surface $\tilde{\Sigma}\cap \An_{\sigma/2,\sigma}(p)$ is separated by $C_\gamma$. 

First, by regularity of $\tilde{\Sigma}$ away from $p$, we have that 
\begin{equation}
\label{eq:step4-2-0}
(\tilde{\Sigma} \cap \mathcal{N}_\sigma(\gamma)) \setminus B_{\sigma/10}(p) \text{ is separated by $C_\gamma$}
\end{equation} for sufficiently small $\sigma>0$, where $\mathcal{N}_\sigma$ denotes the $\sigma$-neighbourhood. Now since any tangent cone is of type (2), for any sequence $\sigma_i\to 0$ there is a line $L\subset T_p \partial M$ such that 
\begin{equation}\label{eq:step4-2} \eta_{p,\sigma_i}(\tilde{\Sigma})\to [H+H'],\end{equation}
where $H,H'$ is the unique pair of half-planes which each intersect $T_p\partial M$ along $L$ at angle $\theta$. In particular, $\eta_{p,\sigma_i}(\gamma)$ Hausdorff converges to $L$. 

The point is that even though $L$ may depend on the blow-up sequence, the claimed property does not depend on $L$. Indeed, suppose that the claim is false, and consider a sequence of counterexamples $\sigma_i \to 0$. Then we have the convergence in (\ref{eq:step4-2}), and by stability and our curvature estimates, this convergence will be smooth away from $L$. In particular, the convergence is smooth in $\An_{1/2,1}(p) \setminus \mathcal{N}_{1/10}(L)$, which implies that $\tilde{\Sigma} \cap \An_{\sigma_i/2, \sigma_i}(p) \setminus \mathcal{N}_{\sigma_i/9}(\gamma)$ is separated by $C_\gamma$ for large enough $i$. Combining with (\ref{eq:step4-2-0}) by unique continuation implies that $\tilde{\Sigma} \cap \An_{\sigma_i/2,\sigma_i}(p)$ is indeed separated by $C_\gamma$. This contradiction completes the proof of the claim. 

Thus $\tilde{\Sigma}\cap \An_{\sigma/2,\sigma}(p)$ is separated by $C_\gamma$ into $\tilde{\Sigma}_i(\sigma)$, $i=1,2$. As $\sigma$ was arbitrary we can fix $\sigma_0$ and continue $\tilde{\Sigma}_i(\sigma_0)$ to  $B_{\sigma_0}(p)\setminus\{p\}$. For convenience we denote the continuations by $\tilde{\Sigma}_i$. Then each $\tilde{\Sigma}_i$ is a regular capillary surface in $B_{\sigma_0}(p)\setminus\{p\}$, lying in a distinct component of $(B_{\sigma_0}(p) \cap \mathring{M}) \setminus C_\gamma$. (In particular $\tilde{\Sigma}$ is separated by $C_\gamma$ into $\tilde{\Sigma}_i$, $i=1,2$.)


But now all tangent cones of either $\tilde{\Sigma}_i$ are of type (1), so we may apply Theorem \ref{theorem:removal.singularity} to each $\tilde{\Sigma}_i$ to conclude that each extends to a regular capillary surface in $B_{\sigma_1}(p)$. This implies that $\tilde{\Sigma}$ extends to a regular capillary surface in $\wt{B}^+_s(p)$, as desired. 

\vspace{0.5em}
\textit{Case 3:} Suppose the tangent cones of $\tilde{V}$ at $p$ are of type (3) or (4). Then the tangent cone at $p$ is in fact unique, since it is a multiple of $[T_p\partial M]$ distinguished by density. This implies that for sufficiently small $\rho$, $\tilde{\Sigma} \cap(\wt{B}^+_\rho(p)\setminus\{p\})$ has a graphical decomposition over $T_p\partial M$. Applying Allard regularity and then elliptic regularity to each sheet (as in the interior setting) gives the extension across $p$. 

\vspace{0.5em}
\textbf{Step 5:} \textit{$V$ coincides with $\tilde{V}$ on a small enough ball.}
\vspace{0.5em}

It is enough to show that $\Sigma$ glues with $\tilde{\Sigma}$ along $\wt{S}^+_\rho(p)$, for sufficiently small $\rho<s_2$. Following exactly the same gluing process as Step 2, with $V^{**}_\rho,V$ taking the place of $V^*, V^{**}$ respectively, implies the result. In particular, the (boundary) regularity of $\Sigma$ up to $\wt{S}^+_\rho(p)$ is not known \textit{a priori}, but now $\Sigma''_{\rho}$ is regular across $\wt{S}^+_\rho(p)$ (since it extends as $\tilde{\Sigma}$). The gluing procedure of Step 2 shows that $\Sigma$ glues smoothly with $\Sigma''_\rho$ across $\tilde{\Gamma}:= \Sigma \cap \wt{S}^+_{\rho}(p) = \Sigma''_\rho\cap \wt{S}^+_{\rho}(p)$, and that $\Clos(\Sigma)\cap \wt{S}^+_\rho(p) = \Clos(\tilde{\Gamma})$; thus by unique continuation $\Sigma$ coincides with $\tilde{\Sigma}$ and is regular on $\wt{B}^+_\rho(p)$. 
\end{proof}

\appendix

\section{An interpolation lemma}

The following interpolation lemma is an adaption of \cite[Lemma A.1]{Zhou-Zhu17}. Compared with \cite[Lemma A.1]{Zhou-Zhu17}, the main difference is item (iv). Now we need to control the mass of the capillary boundary currents, whereas previously we only need to control the mass of the usual boundary currents.
\begin{lemma}
\label{L:interpolation lemma}
Suppose $L>0$, $\eta>0$, $W$ is a compact subset of a relatively open set $U\subset M$, and $\Omega\in \C(M)$. Then there exists $\delta = \delta(L, \eta, U, W, \Omega)>0$, such that for any $\Omega_1, \Omega_2\in \C(M)$ satisfying
\begin{itemize}
    \item[(a)] $\spt(\Omega_i-\Omega)\subset W$, $i=1, 2$,
    \item[(b)] $\M(\partial\Omega_i)\leq L$, $i=1, 2$,
    \item[(c)] $\F(\Omega_1, \Omega_2)\leq \delta$,
\end{itemize}
there exists a sequence $\Omega_1 =\Lambda_0, \Lambda_2, \cdots, \Lambda_m = \Omega_2 \in \C(M)$ such that for each $j=0, 1, \cdots, m-1$,
\begin{itemize}
    \item[(i)] $\spt(\Lambda_j - \Omega) \subset U$,
    \item[(ii)] $\A (\Lambda_j) \leq \max\{ \A(\Omega_1), \A(\Omega_2)\} + \eta$,
    \item[(iii)] $\M(\partial \Lambda_{j} - \partial\Lambda_{j+1}) \leq \eta$,
    \item[(iv)] $\M(\partial^\theta\Lambda_j) \leq \max\{ \M(\partial^\theta \Omega_1), \M(\partial^\theta\Omega_2)\} + \eta/2$, 
    \item[(v)] $\M(\Lambda_j - \Omega_i) \leq \frac{\eta}{2c}$, for $i =1, 2$.
\end{itemize}
\end{lemma}

Here We only need to consider the case when $U\cap \pM \neq \emptyset$, as otherwise the result is exactly \cite[Lemma A.1]{Zhou-Zhu17}. An interpolation result between two relative integral cycles were discussed in \cite[Lemma B.1, Lemma B.3]{Li-Zhou16}. Our results can be adapted from there with several minor modifications, which we will point out. As noted in the proof of \cite[Lemma A.1]{Zhou-Zhu17}, we only need to prove the case when $\Omega_2$ is fixed, and the general case follows by a covering argument. By a contradiction argument, we can find a sequence $\{\Omega_l\}_{l\in\N}\subset\C(M)$ satisfying (a)(b) and $\F(\Omega_l, \Omega)\to 0$. We need to construct a sequence $\{\Lambda_j\}_{j=1}^m$ connecting $\Omega_l$ to $\Omega$ satisfying (i)-(v) when $l$ is sufficiently large. Since we focus on the mass of capillary boundary currents, the key step in \cite[Lemma B.3]{Li-Zhou16} to modify is to change the varifold limit in the line above \cite[(B.1)]{Li-Zhou16} to $V = \lim_{j\to\infty}|\partial^\theta\Omega_j|$. By the lower semi-continuity of mass and~\eqref{eq:boundary current decomposition}, we have
\[ |\partial^\theta\Omega|(A) \leq \|V\|(A), \text{ for all Borel } A\subset M. \]
One can then follow exactly the remaining construction in \cite[Lemma B.3]{Li-Zhou16} to construct the desired sequence of $\{\Lambda_j\}_{j=1}^m$ which satisfy (i)(iii)(iv). Item (v) will follow from the same reasoning as \cite[Page 483, (2)]{Zhou-Zhu17}. Item (ii) is a direct corollary of (iv) and (v). Note that in terms of notations, we just need to change the $S_j$'s (in \cite{Li-Zhou16}) to $\Omega_l$'s.

\section{Foliations of Fermi half-balls and the maximum principle}

\label{sec:gluing-smooth-mp}

\begin{lemma}
Fix $\theta\in(0,\frac{\pi}{2})$. There exists $s^*$ depending only on $(M,g)$ and $\theta$ such that for any $s< s^*$, $p\in M$, there are smooth surfaces $\wt{S}_{s, \gamma}(p)$, $\gamma \in [\theta, \frac{\pi}{2}]$ which foliate a region in $\wt{B}^+_s(p)$ such that:
\begin{enumerate}
    \item $\wt{S}_{s, \frac{\pi}{2}}(p) = \wt{S}^+_s(p)$;
    \item $\wt{S}_{s,\gamma}(p)$ intersects $\pr M$ precisely on $\wt{S}^+_s(p)\cap \pr M$, with contact angle at most $\gamma$;
    \item The mean curvature on $\wt{S}_{s,\gamma}(p)$ is bounded below by $h(s,\gamma)>0$, which satisfies \[\lim_{s\to 0} h(s,\gamma) = \infty\] for every $\gamma$. 
\end{enumerate}
\end{lemma}
\begin{proof}
  
  \begin{figure}[ht]
      \centering
      \includegraphics{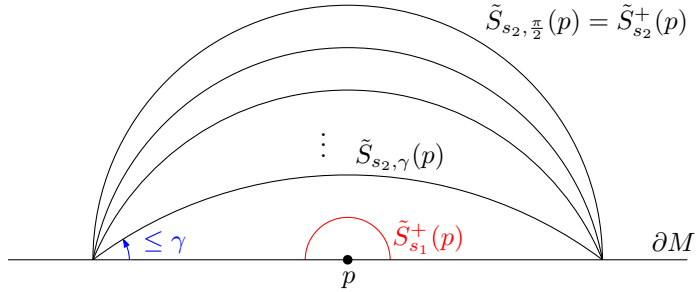}
      \caption{Foliation of $\tilde S_s^+(p)$. Observe that (by the curvature bound) the foliation does not intersect Fermi half-balls $\wt{S}^+_{s_1}(p)$ for small $s_1$.}

  \end{figure}

    We will take $s^*$ smaller than the normal injectivity radius of $\pr M$.
    
    Let $S_{s,\alpha}$ be the spherical cap in $\mathbb{R}^{n+1}_+$ centred on the $x_{n+1}$-axis, which contacts the barrier $\mathcal{B}:=\pr\mathbb{R}^{n+1}_+$ along $S^+_s\cap \mathcal{B}$ at constant angle $\alpha$. (Note $S_{s,\frac{\pi}{2}} = S^+_s$ is the hemisphere of radius $s$ centred at the origin.) We define $\wt{S}^+_{s,\gamma}(p)$ to be the image of $S_{s,\alpha(\gamma)}$ under the normal exponential map centred at $p$, where $\frac{\pi}{2}-\alpha(\gamma) = (1+\mu)(\frac{\pi}{2}-\gamma)$ for some $\mu\in(0,\frac{\theta}{\frac{\pi}{2}-\theta})$. Then item (1) and the first statement of (2) are certainly satisfied.
    
    For the second statement of (2), note that the rescaling $\eta_{0,1/s_i}(\wt{S}^+_{s_i,\gamma}(p_i)-p_i)$ converges smoothly to the spherical cap $S_{1,\alpha(\gamma)}$ (for any sequence $p_i \in \pr M$). Since $\frac{\frac{\pi}{2}-\alpha(\gamma)}{\frac{\pi}{2}-\gamma} =1+\mu>1$, this implies that for small enough $s^*$, $\wt{S}^+_{s,\gamma}$ has boundary contact angle at most $\gamma$ (in fact, equality if and only if $\gamma =\frac{\pi}{2}$). Similarly as the mean curvature of $S_{1,\alpha}$ is strictly positive, item (3) also follows from the smooth convergence. 
\end{proof}

    
    
    
    

\begin{proof}[Proof of Lemma \ref{lem:gluing-smooth-mp}]
Take $s_*$ small enough so that $h(s,\gamma)>c$ for $s<s_*$. Given $s_2<s_*$, take $s_0$ small enough so that $\wt{S}^+_{s_0}(p)\cap \wt{S}_{s_2, \theta}=\emptyset$. As in the statement of the lemma, let $\Sigma$ be a connected, regular capillary surface in the annular region $\wt{\An}_{s_1,s_2}(p)$, which is $C^1$ up to the outer boundary $\wt{S}^+_{s_2}(p)$. Let $\Gamma:= \Clos(\Sigma) \cap \wt{S}^+_{s_2}(p)$ and suppose for the sake of contradiction that $\Gamma \cap \mathring{M} = \emptyset$ and there is some $q_0\in\Gamma$ at which $\Sigma$ is not tangent to $\pr M$. Then $\Sigma$ must meet $\pr M$ at $q_0$ with contact angle $\theta$.

Now suppose that $\Sigma$ is contained in the region bounded by $\wt{S}_{s_2,\theta}(p)$ and $\pr M$. Since $\wt{S}_{s_2,\theta}(p)$ has contact angle at most $\theta$, and mean curvature at least $h(s_2, \theta)>c$, this violates the maximum principle at $q_0$. 

But then $\Sigma$ must intersect some $\wt{S}_{s_2,\gamma}(p)$ in $\mathring{M}$; in particular consider $\gamma_* = \sup \{ \gamma | \Sigma\cap \wt{S}_{s_2,\gamma}(p) \cap \mathring{M} \neq \emptyset\}$. Then $\Sigma$ is contained in the region bounded by $\wt{S}_{s_2,\gamma}(p)$ and $\pr M$. Moreover, since $\wt{S}_{s_2,\gamma}(p) \cap \pr M = \wt{S}^+_{s_2}(p)\cap \pr M$ for every $\gamma$, we see that $\Sigma$ must intersect $\wt{S}_{s_2,\gamma_*}(p)$ tangentially at some point $q\in M$ (which may be either interior or boundary). Since $\wt{S}_{s_2,\gamma_*}(p)$ has mean curvature at least $h(s_2,\gamma_*)>c$, this violates the maximum principle at $q$. 

Note that we needed the $C^1$ regularity to apply the Hopf lemma at boundary points where the contact angle of $\Sigma$ is greater than that of the barriers $\wt{S}_{s_2,\gamma}$, and the $C^1$ regularity plus Allard regularity to apply the maximum principle at tangent points. In either case we have a contradiction, which completes the proof. 
\end{proof}

\section{Stable Bernstein theorems for capillary surfaces}
\label{sec:bernstein}

In this appendix we prove stable Bernstein theorems for capillary hypersurfaces of dimension $2\leq n\leq 5$, for certain ranges of contact angles. In particular, the goal is to show:
\begin{theorem}
\label{thm:stable-bernstein-higher}
Consider $n,\theta$ satisfying either
\begin{itemize}
\item $n=2$, and $\theta\in(0,\pi)$;
\item $n=3$, and $\frac{5}{3}+\frac{3\sqrt{2}}{2}+(\sqrt{2}-1)C_\theta - \frac{5\sqrt{2}}{2}C_\theta^2 >0$;
\item $n=4$, and $\frac{3}{2}+\frac{3\sqrt{3}}{2} +(\sqrt{3}-1)C_\theta - \frac{5\sqrt{3}}{2}C_\theta^2 >0$;
\item $n=5$, and $\frac{32}{5}+\frac{29}{4}C_\theta -\frac{27}{2}C_\theta^2>0$. 
\end{itemize}
where $C_\theta = \frac{(1+|\cos\theta|)^3}{\sin^2\theta}$. Suppose that $(\Sigma, \pr \Sigma) \subset \mathbb{R}^{n+1}_+$ is a two-sided, stable, capillary minimal hypersurface with contact angle $\theta$, and $\sup_{r>0} \frac{\vol(\Sigma\cap B_r)}{r^n} \leq C_V <\infty$. Then $\Sigma$ must be planar. 
\end{theorem}

Recall the notation of Section \ref{S:preliminaries}: $\nu$ is the inward pointing unit normal vector field on $\Sigma$, $\eta$ is the outward conormal of $\partial \Sigma$, $\bar\eta$ is the outward conormal of the barrier, $\bar\nu$ is the outward conormal of $\partial \Sigma$ in the barrier. The contact angle condition becomes $\bangle{\nu,\bar \eta} = \cos\theta$.

To prove Theorem \ref{thm:stable-bernstein-higher}, we will use the Schoen-Simon-Yau \cite{Schoen-Simon-Yau75} technique. As such, we first compute the boundary conditions for the total curvature:

\begin{lemma}
    \label{lem:A-bc}
    Let $(\Sigma,\pr\Sigma)\subset \R^{n+1}_+$ be a capillary minimal hypersurface so that $\bangle{\nu,\bar\eta}=\cos\theta$. Then along $\partial \Sigma$,
    \[\pa{|A|^2}{\eta}=2\cot\theta \left(2|A|^2 A(\eta,\eta) - \sum_{j=1}^n \lambda_j^3\right),\]
    here $\lambda_1,\cdots,\lambda_n$ are the principal curvatures of $\Sigma$. In particular,
    \[\left| \pa{|A|^2}{\eta}\right|\le 6\sqrt{n-1} |\cot\theta| |A|^3.\]
\end{lemma}

\begin{proof}
    Fix a point $p\in \partial \Sigma$. At $p$, we have:
    \[\eta=\cot\theta \nu + \frac{1}{\sin\theta} \bar\eta,\quad \bar\nu = \frac{1}{\sin\theta} \nu + \cot\theta \bar \eta.\]
    Take Fermi coordinates around $p$, such that the metric can be written as $g=dx_1^2 + g_{x_1}(x_2,\cdots,x_n)$, and $g_{0}$ is given by the geodesic normal coordinates. Note that in this choice, $\partial_1 = \eta$. By our convention on the second fundamental form, $A_{ij} = -\bangle{\nabla_{\partial_i}\partial_j,\nu}$. Now take indices $i,j,\alpha>1$. First, observe that at $p$,
    \begin{align*}
        A_{1j}&=-\bangle{\partial_j \eta,\nu}=-\bangle{-\partial_j(\cot \theta \nu +\frac{1}{\sin\theta} \bar\eta), \nu}\\
                &=-\frac{1}{\sin\theta} \bangle{\partial_j \bar\eta ,\nu} = -\frac{1}{\sin\theta} \bangle{\partial_j \bar\eta, -\cos\theta \bar\eta + \sin\theta\bar\nu}=0,
    \end{align*}
    as $\partial\R^{n+1}_+$ is flat (and hence $\partial_j \bar\eta=0$). Therefore $\partial_j \nu = A_{j\alpha} \partial_\alpha + A_{j1} \nu = A_{j\alpha}\partial_\alpha$, and similarly $\partial_1 \nu = -A_{11} \eta$. 
    
    Let $x$ denote the immersion that gives $\Sigma$. At $p$, we have
    \begin{equation}\label{eq:partial1_Aij}
        \begin{split}
            \partial_1 A_{ij} &= -\bangle{\partial_1\partial_i\partial_j x,\nu} - \bangle{\partial_i\partial_j x, \partial_1 \nu}= -\bangle{\partial_i\partial_j \eta, \nu} - A_{11} \bangle{\partial_i \partial_j x , \eta}\\
            &=-\bangle{\nu,\partial_i(\cot\theta \partial_j \nu+\frac{1}{\sin\theta} \partial_j \bar\eta) } - A_{11}\bangle{\nabla_{\partial_i}\partial_j, (\cot\theta \nu +\frac{1}{\sin\theta} \bar\eta)}\\
            &= - \cot\theta \bangle{\nu, \partial_i (A_{j\alpha} \partial_\alpha)} + \cot\theta A_{11} A_{ij}\\
            &= \cot\theta (A_{i\alpha}A_{\alpha j} + A_{11}A_{ij}).
        \end{split}
    \end{equation}
    Similarly,
    \begin{equation}
        \begin{split}
            \partial_1 A_{1j} &= -\bangle{\partial_1\partial_j\eta , \nu} - A_{11}\bangle{\partial_1\partial_j x, \eta} = -\bangle{\nu, \partial_1(\cot\theta \partial_j \nu)} -A_{11}\bangle{\partial_j\partial_1 x,\eta}\\
            &=\cot\theta A_{j\alpha} A_{\alpha 1} - A_{11} \bangle{\nabla_{\partial j} \eta,\eta} = 0.
        \end{split}
    \end{equation}
    Additionally, $\partial_i \eta = \partial_i (\cot\theta \nu + \frac{1}{\sin\theta} \bar\eta)=\cot\theta A_{i\alpha }\partial_\alpha$. Thus, $A(\nabla_{\partial_i} \eta, \partial_j) = A((\partial_i \eta)^T,\partial_j) = \cot\theta A_{i\alpha}A_{\alpha j}$. Combining this with \eqref{eq:partial1_Aij}, we have
    \[(\nabla_{\partial_1}A)_{ij} = \partial_1A_{ij} - A(\nabla_{\partial_1}\partial_i,\partial_j) - A(\partial_i,\nabla_{\partial_1}\partial_j) = \cot\theta (A_{11}A_{ij} - A_{i\alpha}A_{\alpha j}).\]
    We next compute $(\nabla_{\partial_1}A)_{11}=\nabla_{\partial_1} A_{11} - 2A(\nabla_{\partial_1}\partial_1, \partial_1)$. Note that $\nabla_{\partial_1}\partial_1$ has zero component in $\partial_1$, and thus equals $c_j \partial_j$, giving $A(\nabla_{\partial_1}\partial_1, \partial_1) = c_j A_{j1}=0$. Using minimality of $\Sigma$,
    \[\nabla_{\partial_1} A_{11}= - \nabla_{\partial_1} A_{jj} = \cot\theta (A_{j\alpha}A_{\alpha j} - A_{11}A_{jj})=\cot\theta |A|^2.\]
    Therefore,
    \begin{equation}
        \begin{split}
            \nabla_{\partial_1} |A|^2 &= 2\bangle{\nabla_{\partial_1}A,A}\\
                                      &= 2\cot\theta (|A|^2 A_{11} + A^{ij}(-A_{i\alpha}A_{\alpha j}+A_{11}A_{ij}))\\
                                      &= 2\cot\theta (2|A|^2A_{11} - A_{11}^3 - A^{ij}A_{i\alpha}A_{\alpha j} ).
        \end{split}
    \end{equation} 
    We have $A_{11}^3 + A^{ij}A_{i\alpha}A_{\alpha j}=\sum_{k=1}^n \lambda_k^3$, where $\lambda_1,\cdots,\lambda_n$ are the principal curvatures of $\Sigma$.  
    
    Finally, observe that for any $k=1,\cdots,n$, 
    \[|\lambda_k| = \left|\sum_{l\ne k} \lambda_l\right| \le \sqrt{n-1} \left(\sum_{l\ne k} \lambda_l^2\right)^{1/2}\le \sqrt{n-1} |A|.\]
    Thus $|\sum \lambda_k^3|\le \sum \sqrt{n-1} |A| \lambda_k^2 =\sqrt{n-1} |A|^3$. Similarly, $|A_{11}|\le \sqrt{n-1}|A|$, and thus
    \[\left|\nabla_{\partial_1}|A|^2\right|\le 6\sqrt{n-1}|\cot\theta| |A|^3\]
\end{proof}

We need the following trace estimate:

\begin{lemma}
\label{lem:trace}
    Suppose $\Sigma^n\subset \R^{n+1}_+$ is a smoothly immersed hypersurface with smooth boundary meeting $\pr\mathbb{R}^{n+1}_+$ at constant angle $\theta$. Then for any $u\in C_c^1(\overline\Sigma)$, we have
    \[\int_{\partial \Sigma} u \le \frac{1}{\sin\theta} \int_\Sigma |\nabla u| + |H_\Sigma u|.\]
\end{lemma}

\begin{proof}
    Assume $\R^{n+1}_+ = \{x_{n+1}\ge 0\}$. For $R>0$, let $\varphi_R:[0,\infty)\to \R$ be a smooth function such that $\varphi_R(t)=1$ when $t\le R$, $\varphi_R(t)=0$ when $t\ge 2R$, and $|\varphi_R'|\le \tfrac 2R$. Define the vector field $X=-\varphi_R(x_{n+1}) \partial_{n+1}$. Let $\eta$ be the outward unit conormal vector field on $\partial \Sigma$. By assumption, $\eta\cdot X = \sin\theta$ along $\partial \Sigma$. Therefore
    \begin{align*}
        \int_{\partial \Sigma} u &= \frac{1}{\sin\theta} \int_{\partial \Sigma} uX \cdot \eta\\
                                 &= \frac{1}{\sin\theta} \int_\Sigma \Div_\Sigma (uX^T) \le  \frac{1}{\sin\theta} \int_\Sigma |\nabla u| |X| + u\Div X + |u H_\Sigma|\\
                                 &\le \frac{1}{\sin\theta}\int_\Sigma |\nabla u| + \frac 2R u +|H_\Sigma u|.
    \end{align*}
    Letting $R\to \infty$, we obtain the desired inequality.
\end{proof}

\begin{lemma}
\label{lem:ssy-p}
Fix $n,\theta$ and suppose $(\Sigma,\pr \Sigma)\subset \mathbb{R}^{n+1}_+$ is a two-sided, stable, capillary minimal hypersurface with contact angle $\theta$. 
Let $a,b>0$ and $q\geq 0$ be such that \[B_{q,a,b} = \frac{2}{n} +1+2q - \frac{3+2q}{2}bc_{n,\theta} -\left(1+ ac_{n,\theta} + \frac{3+2q}{2}b^{-1}c_{n,\theta} \right)(1+q)^2C_\theta >0,\]
where $c_{n,\theta} = 3\sqrt{n-1} \frac{|\cos\theta|}{\sin^2\theta}$ and $C_{\theta} = \frac{(1+|\cos\theta|)^2}{1-|\cos\theta|} = \frac{(1+|\cos \theta|)^3}{\sin^2\theta}.$ Then setting $p=2+q$, for any $\phi \in C^1_c(\overline{\Sigma})$, there exists $C=C(n,\theta,p,a,b)$ such that 
\[\int_\Sigma |A|^{2p}\phi^{2p} \leq C\int_\Sigma |\nabla \phi|^{2p}.\]
\end{lemma}
\begin{proof}
The main point is to prove
\begin{equation}
    \label{eq:ssy-grad}
    \int_\Sigma |\nabla|A||^2 |A|^{2q}f^2 \leq C\int_\Sigma |A|^{2+2q}|\nabla f|^2.
\end{equation}

Recall that Simons' inequality states 
\begin{equation}
    \label{eq:simons}
    |A|\Lap |A| + |A|^4 \geq \frac{2}{n}|\nabla|A||^2.
\end{equation}
Multiplying (\ref{eq:simons}) by $|A|^{2q}f^2$ and integrating by parts, we have 
\begin{align}
\frac{2}{n}\int_\Sigma |\nabla|A||^2 |A|^{2q}f^2 &\leq&& \int_\Sigma |A|^{4+2q}f^2 - 2\int_\Sigma f|A|^{1+2q}\langle \nabla f,\nabla |A|\rangle - (1+2q)\int_\Sigma |\nabla|A||^2 |A|^{2p}f^2 \\\nonumber&&& + \int_{\pr\Sigma} f^2|A|^{1+2q}\pr_\eta |A|
\\\nonumber &\leq && \int_\Sigma |A|^{4+2q}f^2 - 2\int_\Sigma f|A|^{1+2q}\langle \nabla f,\nabla |A|\rangle - (1+2q)\int_\Sigma |\nabla|A||^2 |A|^{2p}f^2 \\\nonumber&&& + 2c_{n,\theta}\int_{\Sigma}\left(|f\nabla f| |A|^{3+2q} + \frac{3+2q}{2} f^2|A|^{2+2q}|\nabla|A||\right).
\end{align}
Here we have used Lemmas \ref{lem:A-bc} and \ref{lem:trace} for the second inequality. 

We now estimate \[2|f\nabla f||A|^{3+2q} \leq a|A|^{4+2q}f^2 + a^{-1} |A|^{2+2q}|\nabla f|^2,\] \[\qquad 2f^2|A|^{2+2q}|\nabla |A|| \leq b|A|^{2q} |\nabla|A||^2 f^2 + b^{-1} |A|^{4+2q}f^2.\] 

This gives
\begin{align}
    \label{eq:ssy-3}
    \left(\frac{2}{n}+1+2q-\frac{3+2q}{2}bc_{n,\theta}\right)\int_\Sigma |\nabla|A||^2 |A|^{2q}f^2 &\leq&& \left(1+ ac_{n,\theta} + \frac{3+2q}{2}b^{-1}c_{n,\theta} \right) \int_\Sigma |A|^{4+2q}f^2 \\\nonumber &&&+a^{-1}c_{n,\theta} \int_\Sigma |A|^{2+2q}|\nabla f|^2 \\\nonumber &&& -2\int_\Sigma f|A|^{1+2q}\langle \nabla f,\nabla |A|\rangle.
\end{align}

Substituting $|A|^{1+q}f$ into the almost stability inequality (\ref{equation.almost.stability}) gives 
\begin{equation}
    \label{eq:ssy-4}
    \frac{1}{C_{\theta}}\int_\Sigma |A|^{4+2q}f^2 \leq (1+q)^2 \int_\Sigma |\nabla|A||^2 |A|^{2q}f^2 + \int_\Sigma |A|^{2+2q}|\nabla f|^2 + 2(1+q) \int_\Sigma f|A|^{1+2q}\langle \nabla f,\nabla |A|\rangle.
\end{equation}
Now for $\epsilon>0$ we can estimate \begin{equation}\label{eq:ssy-5}\qquad 2|f\nabla f||A|^{1+2q}|\nabla |A|| \leq \epsilon|A|^{2q} |\nabla|A||^2 f^2 + \epsilon^{-1} |A|^{2+2q}|\nabla f|^2,\end{equation}
so using (\ref{eq:ssy-4}) in (\ref{eq:ssy-3}) gives
\[ B_{q,a,b,\epsilon}\int_\Sigma |\nabla|A||^2 |A|^{2q}f^2 \leq C_{q,a,b,\epsilon} \int_\Sigma |A|^{2+2q}|\nabla f|^2,\]
where 
\[
\begin{split}
    B_{q,a,b,\epsilon} =& \frac{2}{n} +1+2q - \frac{3+2q}{2}bc_{n,\theta} -\left(1+ ac_{n,\theta} + \frac{3+2q}{2}b^{-1}c_{n,\theta} \right)(1+q)^2C_\theta \\& -\left(\left(1+ ac_{n,\theta} + \frac{3+2q}{2}b^{-1}c_{n,\theta} \right)(1+q)C_\theta-1\right)\epsilon, 
\end{split}
    \]

\[C_{q,a,b,\epsilon} = \left(1+ ac_{n,\theta} + \frac{3+2q}{2}b^{-1}c_{n,\theta} \right)C_\theta + \left(\left(1+ ac_{n,\theta} + \frac{3+2q}{2}b^{-1}c_{n,\theta} \right)(1+q)C_\theta-1\right) \epsilon^{-1}. \]

Now since
\[B_{q,a,b} = \frac{2}{n} +1+2q - \frac{3+2q}{2}bc_{n,\theta} -\left(1+ ac_{n,\theta} + \frac{3+2q}{2}b^{-1}c_{n,\theta} \right)(1+q)^2C_\theta >0,\] we
may always choose $\epsilon$ small enough so that $B_{q,a,b,\epsilon}>0$ and hence establish (\ref{eq:ssy-grad}). 

With (\ref{eq:ssy-grad}) in hand, we substitute it back into (\ref{eq:ssy-4}) and use (\ref{eq:ssy-5}) again to deduce that 
\begin{equation}
\int_\Sigma |A|^{4+2q}f^2 \leq C\int_\Sigma |A|^{2+2q}|\nabla f|^2. 
\end{equation}

Setting $p=2+q$ and $f=\phi^p$, H\"{o}lder's inequality gives that 
\[\int_\Sigma |A|^{2p}\phi^{2p} \leq C\int_\Sigma |A|^{2p-2}\phi^{2p-2}|\nabla\phi|^2 \leq C\left(\int_\Sigma |A|^{2p}\phi^{2p}\right)^\frac{p-1}{p}\left(\int_\Sigma |\nabla \phi|^{2p}\right)^\frac{1}{p},\]
which implies the result. 

\end{proof}

\begin{corollary}
\label{cor:ssy-p}
Let $n,\theta,\Sigma$ be as in Lemma \ref{lem:ssy-p}. The conclusions of that lemma hold for some $2p>n$ so long as either:
\begin{itemize}
\item $n=3$, and $\frac{5}{3}+\frac{3\sqrt{2}}{2}+(\sqrt{2}-1)C_\theta - \frac{5\sqrt{2}}{2}C_\theta^2 >0$;
\item $n=4$, and $\frac{3}{2}+\frac{3\sqrt{3}}{2} +(\sqrt{3}-1)C_\theta - \frac{5\sqrt{3}}{2}C_\theta^2 >0$;
\item $n=5$, and $\frac{32}{5}+\frac{29}{4}C_\theta -\frac{27}{2}C_\theta^2>0$. 
\end{itemize}
\end{corollary}
\begin{proof}
Take $a=b=1$ and estimate $c_{n,\theta} \leq \sqrt{n-1} (C_\theta-1)$. Using these, straightforward calculations yield the following:

If $n=3$, then $B_{0,1,1} \geq \frac{5}{3}+\frac{3\sqrt{2}}{2}+(\sqrt{2}-1)C_\theta - \frac{5\sqrt{2}}{2}C_\theta^2 >0$. If $n=4$, we again note that $B_{0,1,1} \geq \frac{3}{2}+\frac{3\sqrt{3}}{2} +(\sqrt{3}-1)C_\theta - \frac{5\sqrt{3}}{2}C_\theta^2 >0$. Then by continuity, $B_{q,1,1}>0$ for small enough $q>0$. If $n=5$, we note that $B_{\frac{1}{2},1,1} \geq \frac{32}{5}+\frac{29}{4}C_\theta -\frac{27}{2}C_\theta^2>0$. Again by continuity, $B_{q,1,1}>0$ for some $q> \frac{1}{2}$. In any case, $2p=2(2+q)>n$ which completes the proof. 
\end{proof}

\begin{remark}
Clearly, one may obtain a slightly larger range for $\theta$ by optimising the conditions $B_{q,a,b}>0$ in the above. For $n=3$, one may also obtain a larger range by following the argument in \cite[Appendix D]{CL21}. We have chosen to present the most straightforward estimate, as we suspect that even these improved estimates are not sharp, and that the stable Bernstein theorem in fact holds for a significantly larger range of $\theta$. 
\end{remark}

\begin{proof}[Proof of Theorem \ref{thm:stable-bernstein-higher}]
We proved the $n=2$ case above in Theorem \ref{theorem.stable.bernstein}. Let $\phi$ be a smooth cutoff function such that $\phi=1$ on $B_r$, $\phi=0$ outside $B_{2r}$ and $|D \phi| \leq \frac{2}{r}$ on $B_{2r}\setminus B_r$. 
By Corollary \ref{cor:ssy-p}, we have \[\int_{\Sigma\cap B_r} |A|^{2p} \leq C 2^{2p} r^{-2p} \vol(\Sigma\cap B_{2r}) \leq C C_V 2^{n+2p} r^{n-2p}\] for some $2p>n$. Taking $r\to\infty$ implies $|A|\equiv 0$, which completes the proof. 
\end{proof}

\bibliography{bib}
\bibliographystyle{plain}

\end{document}